\documentclass[11pt]{amsart}
\usepackage{amsmath,amsfonts,amsthm,amssymb}

\newtheorem{defn}{Definition}[section]
\newtheorem{theorem}{Theorem}
\newtheorem{lemma}[defn]{Lemma}
\newtheorem{prop}[defn]{Proposition}
\newtheorem{proposition}[defn]{Proposition}
\newtheorem{assumption}{Assumption}

\newtheorem{example}[defn]{Example}
\newtheorem{remark}[defn]{Remark}
\newtheorem{corollary}[defn]{Corollary}

\def\R{{\mathbb R}}
\def\N{{\mathbb N}}
\def\C{{\cal C}}
\def\Z{{\mathbb Z}}
\def\1{{1\!\!1}}
\def\E{{\mathbb E}}
\def\P{{\mathbb P}}

\def\Q{{\mathbb Q}}
\def\ind{{1\!\!1}}
\def\func{\phi}
\def\cxtwo{c_{x}^2}
\def\cxone{c_{x}^1}
\def\nvec{v}
\def\csto{\eta}
\def\cstt{\lambda}

\def\MH{{\mathcal H}}

\def\MM{{\mathcal M}}
\def\MB{{\mathcal B}}

\def\MH{{\mathcal H}}
\def\MI{{\mathcal I}}
\def\MF{{\mathcal F}}

\def\MV{{\mathcal V}}
\def\MW{{\mathcal W}}

\def\ML{{\mathcal L}}
\def\MC{{\mathcal C}}
\def\MJ{{\mathcal J}}

\def\MU{{\mathcal U}}

\def\cal{\mathcal}

\def\dist{{\rm{dist}}}
\def\supp{{\rm{supp}}}
\def\endproof{{$\Box$}}

\newcommand{\noi}{\noindent }
\newcommand{\ba}{\begin{array}}
\newcommand{\ea}{\end{array}}
\newcommand{\bea}{\begin{eqnarray}}
\newcommand{\eea}{\end{eqnarray}}
\newcommand{\beas}{\begin{eqnarray*}}
\newcommand{\eeas}{\end{eqnarray*}}
\newcommand{\be}{\begin{equation}}
\newcommand{\ee}{\end{equation}}
\newcommand{\bt}{\begin{theorem}}
\newcommand{\et}{\end{theorem}}
\newcommand{\bc}{\begin{center}}
\newcommand{\ec}{\end{center}}
\newcommand{\ben}{\begin{enumerate}}
\newcommand{\een}{\end{enumerate}}
\newcommand{\lan}{\langle}
\newcommand{\ran}{\rangle}
\newcommand{\ei}{\end{itemize}}

\newcommand{\ds}{\displaystyle}

\newcommand{\ve}{\varepsilon}

\newcommand\ccspace{{\cal C}\left[0,\infty\right)}

\def\orig{0}
\def\newset{\Theta}

\def\newH{S}
\def\sm{\setminus}
\def\ra{\rightarrow}

\title[Stationary Distribution Characterization]{Characterization of stationary
   distributions of reflected diffusions}
\author{Weining Kang}
\address{Department of Mathematics \& Statistics\\
University of Maryland, Baltimore County\\
1000 Hilltop Circle\\
Baltimore, MD 21250 }
 \email{wkang@umbc.edu}

\author{Kavita Ramanan}
\address{Division  of Applied Mathematics\\
Brown University\\
Providence, RI 02118 \\
USA}
\email{kramanan@math.cmu.edu}

\thanks{The second author was partially supported by NSF grants
  CMMI-1059967 (formerly 0728064) and CMMI-1052750 (formerly 0928154)}

\subjclass[2010]{Primary: 60H10, 60J60, 60J65;  Secondary: 90B15, 90B22}
\keywords{Reflected diffusions, reflecting diffusions, stationary distribution, invariant
  distribution, submartingale  problem, stochastic differential
  equations with reflection, basic adjoint relation, queueing networks}

\begin{document}

\begin{abstract} Given a domain $G$, a reflection vector field
$d(\cdot)$ on $\partial G$, the boundary of $G$, and drift and dispersion coefficients $b(\cdot)$ and $\sigma(\cdot)$,
  let  $\ML$ be the usual second-order elliptic operator associated with $b(\cdot)$ and
$\sigma(\cdot)$.  Under suitable assumptions that,
  in particular,  ensure that the associated submartingale problem is
  well posed,
it  is shown that a probability measure $\pi$ on $\overline{G}$ is a stationary distribution
for the corresponding reflected diffusion  if and only if  $\pi (\partial G) = 0$ and
\[  \int_{\bar{G}} \ML f (x) \pi (dx) \leq 0
\]
for every $f$ in a certain class of test functions.  
Moreover, the assumptions  are shown to be satisfied
by a large class of reflected diffusions  in piecewise smooth
multi-dimensional domains with possibly oblique reflection.
 \end{abstract}

\maketitle




\section{Introduction}

\subsection{Description of Main Results}

The main focus of this work is to provide a simple characterization
of stationary distributions of a broad class of reflected diffusions.
Consider a domain $G\subset \R^J$, equipped with a vector field
$d(\cdot)$ on the boundary $\partial G$, and drift and dispersion coefficients $b: \overline{G} \mapsto \R^J$ and
$\sigma: \overline{G} \mapsto \R^{J} \times \R^{N}$, where $\overline{G}$ is the
closure of $G$.  A reflected diffusion associated with $(G,
d(\cdot))$, $b(\cdot)$ and $\sigma(\cdot)$ is, roughly
speaking, a continuous Markov
  process that behaves locally like a diffusion with state-dependent
  drift $b(x)$ and dispersion $\sigma(x)$, for $x$ in
  $G$, and is instantaneously constrained to stay inside $\overline{G}$  by a pushing term that is only allowed
to act when the process is on the boundary, and then only along the
directions specified by the vector field $d(\cdot)$ at that point on
the boundary.      One approach to making this heuristic description precise
is the so-called submartingale problem of Stroock and
Varadhan \cite{StrVar71},
which is a generalization of the martingale problem that was 
introduced  to characterize  the law of reflected diffusions in smooth domains.
A submartingale problem is said to be well posed if it has a unique solution, and this
 implies existence and uniqueness in law  of the associated reflected diffusion.
 A precise formulation of
 the submartingale problem in multi-dimensional domains is given in
 Section \ref{sec-rediff}.    The relation of this formulation
to stochastic differential equation characterizations of reflected
diffusions can be found in \cite{KanRam11b}.

For reflected diffusions in a bounded domain, the family of 
time-averaged occupation  measures is automatically tight, 
and the existence of a stationary distribution can be deduced as a
simple consequence. 
 On the other hand, for reflected diffusions in unbounded domains  suitable conditions on the
 drift and reflection vector field  need to
be imposed to guarantee positive recurrence (see, e.g.,
\cite{AtaBudDup01} for sufficient conditions and also \cite{Bra11},
which shows that the issue of stability 
can be quite subtle in the presence of oblique reflection).
In either case,
when the diffusion coefficient is uniformly elliptic,
  uniqueness of the stationary distribution
follows from standard results in ergodic theory.
The focus of the present paper is 
 on characterization of the stationary distribution.

Given continuous drift and dispersion coefficients $b: G \mapsto \R^J$ and
$\sigma: G \mapsto \R^{J \times N}$,
let $a:G \mapsto \R^{J \times J}$ be the associated diffusion
coefficient given by 
$a(\cdot) = \sigma (\cdot)\sigma^T (\cdot)$, where $\sigma^T(x)$
denotes the transpose of the matrix $\sigma(x)$, and let 
$\ML$ be  the usual associated second-order  differential operator 
\be \label{operL}
\ML f(x)\doteq\sum_{i=1}^J b_i(x)\frac{\partial f}{\partial
x_i}(x) +\frac{1}{2}\sum_{i,j=1}^J a_{ij}(x)\frac{\partial^2f}{\partial
x_i\partial x_j}(x), \qquad f\in \C^2_b(\overline G),
\ee
where $\C^2_b(\overline G)$ is the space of twice continuously
differentiable functions on $\overline G$ that, along with their first
and second partial derivatives, are bounded.  
The first main result of this paper, Theorem \ref{thm:SS},  
shows that under suitable conditions (see Assumptions \ref{ass:TF} and
\ref{ass:V}),  a probability measure  $\pi$ on $\overline{G}$ is a
stationary distribution 
for a reflected diffusion defined by a well posed submartingale problem if and only if
$\pi$ satisfies $\pi (\partial G)  = 0$ and
\be \label{mono0}
\int_{\overline G} \ML f(x)\,d\pi(x)\leq 0
\ee
for all $f$ belonging to $\MH$, the class of test functions defined in
(\ref{dis:H}).
The second result, Theorem \ref{thm:test},  shows that the
conditions of Theorem \ref{thm:SS}  are satisfied by a large
class of reflected diffusions in piecewise smooth domains described in
Definition \ref{def:A}, which includes  a
family of (possibly non-semimartingale) reflected Brownian motions (RBM) in convex polyhedral
domains with piecewise constant vector fields that arise in
applications.    
Illustrative examples of such reflected diffusions are provided  
 in Section \ref{subs-egs},  where it is additionally shown that 
the conditions of Theorem \ref{thm:SS} are also satisfied by 
some processes outside this class, including
 reflected diffusions  in two-dimensional domains with cusps.

If the stationary distribution $\pi$ of a reflected diffusion can be shown  to have a density
$p$ that is sufficiently regular, standard arguments can be used to show
that $p$ should be the solution to a certain partial differential
equation governed by the adjoint $\ML^*$ of $\ML$, and subject to
certain oblique derivative boundary conditions.
In non-smooth domains, such regularity properties are not always
satisfied and, even when satisfied, are typically hard to establish
{\em a priori}.
   Nevertheless, in some cases it is possible to write
down the formal adjoint partial differential equation and boundary conditions and find an explicit
   solution to it
(see, e.g.,  \cite{HarLanShe85}, \cite{KneMor03}, 
\cite{New79}, \cite{SchMie94}). 
In such  situations, it would be useful to have a result that guarantees
that this solution is indeed the stationary distribution of the
reflected diffusion.
Using Theorem \ref{thm:SS},  it is shown in Corollary \ref{Thm:SBR}  that any nonnegative integrable
solution of the adjoint partial differential equation, when
suitably normalized,  is indeed a stationary density  of the reflected
diffusion.

\subsection{Motivation and Prior Work} 

Our results can be viewed as a generalization of Echeverria's results
\cite{Ech82} for diffusions in $\R^J$ to the case of reflected
diffusions.   Given sufficiently regular drift  and dispersion coefficients
$b(\cdot)$ and $\sigma(\cdot)$,
Echeverria \cite{Ech82}  showed that a probability measure
$\pi$ is a stationary distribution for the associated diffusion in
$\R^J$ characterized by   the corresponding martingale problem if and only if
 (\ref{mono0}) holds with  inequality replaced by equality and
 for test functions $f \in \C^2_c(\R^J)$, the space of twice
 continuously differentiable functions on $\R^J$ with compact support.
An extension of this result to reflected diffusions in $\C^2$-domains
 was considered by Weiss in his unpublished Ph.D.\ thesis
\cite{Weiss81}.
However, the results of \cite{Weiss81} do not apply to reflected
diffusions in non-smooth domains in $\R^J$.
 Such reflected diffusions arise in many fields including finance \cite{Bou08}, economics
\cite{Ramsub00}, communications \cite{RamRei03} and
 operations research  \cite{HarWil87}, and it is of interest
to characterize their stationary distributions.
For the particular case of reflected Brownian motion (RBMs) in
convex polyhedral domains,  explicit expressions for the stationary distribution
 have been established only in some two-dimensional examples
 \cite{DieMor09}, \cite{HarLanShe85}, \cite{Wil85},  or when a certain skew-symmetry condition holds
\cite{Wil87}, in which case  the stationary distribution has a
density that is the product of its marginals.
In other cases,
numerical schemes have been proposed for the computation of
 stationary distributions  (see, e.g.,  \cite{DaiHar92} and \cite{SauGlyZee}).
The characterization of the stationary distribution established
in this paper could be used to provide 
a rigorous justification for these numerical
schemes.
For reflected Brownian motions (RBMs) in convex
polyhedrons that are semimartingales,  there has been some work on 
establishing the basic adjoint relationship using
so-called constrained martingale problems 
in the unpublished manuscript \cite{DaiKur94}.
In this paper, we establish a different characterization by a
different method of proof and deal with the more  general setting
of curved, non-smooth domains, and reflected  diffusions (with state dependent
coefficients) that are not necessarily semimartingales, both of which arise
in applications (see, e.g.,  \cite[Section 5.6]{KanKelLeeWil09}
and \cite{KanRam10, RamReed11, RamRei03, RamRei08}).

\subsection{Outline of the Paper}

 Section \ref{sec-rediff}  contains a precise definition of  the submartingale problem
and  the associated class of reflected diffusions.    In Section \ref{sec-mainres}  the
main results of the paper, Theorem \ref{thm:SS}, Theorem \ref{thm:test}
and  Corollary
\ref{Thm:SBR}, are stated.  The proof of Corollary \ref{Thm:SBR} is
given 
in Section \ref{sec-mainres}, whereas 
the proofs of Theorem \ref{thm:SS}  and Theorem \ref{thm:test}
are deferred to Section \ref{sec-proof1} and Section \ref{sec-proof2},
respectively.
Section \ref{subs-egs} contains illustrative examples of reflected diffusions for which the 
stationary distribution characterization established 
in this paper is valid. 
The proofs of some technical lemmas are relegated to the Appendix.
 First, in the next section, we summarize some common notation used in the paper.

\subsection{Notation and Terminology}

The following notation is used throughout the paper.
 $\Z$ is the set of integers, $\N$ is the set of positive integers, $\R$ is the set of real numbers, $\Z_+$ is the set of non-negative integers and
$\R_+$ the set of non-negative real numbers. For each $J\in \N$,
$\R^J$ is the $J$-dimensional Euclidean space and $|\cdot|$ and $\lan
\cdot,\cdot \ran$, respectively,
denote the Euclidean norm  and the inner product on $\R^J$.
For each vector $v\in \R^J$ and matrix $\sigma\in \R^J\times \R^N$,
$v^T$ and $\sigma^T$ denote the transpose of $v$ and $\sigma$,
respectively.
For each set $A\subset \R^J$, $A^\circ$, $\partial A$ and $\overline
A$ denote the interior, boundary and closure of $A$, respectively. For
each $x\in \R^J$ and $A\subset \R^J$, $\dist(x,A)$ is the distance
from $x$ to $A$ (that is, $\dist(x,A)=\inf\{y\in A:\ |y-x|\}$). For
each $A\subset \R^J$ and $r>0$, $B_r(A)=\{y\in \R^J:\ \dist(y,A)\leq
r\}$, and given $\ve > 0$ let $A^\ve \doteq \{y \in \R^J: \dist(y,A) <
\ve\}$ denote the (open) $\ve$-fattening of $A$.  If $A=\{x\}$, we
simply denote $B_r(A)$ by $B_r(x)$.  
We also let $\ind_B$ denote the indicator function of the set $B$
(that is, $\ind_B (x) = 1$ if $x \in B$ and $\ind_B(x) = 0$
otherwise).

Given  a domain $E$ in $\R^n$, for some $n \in \N$, and any
$m\in \Z_+\cup\{\infty\}$, let $\C^m(E)$ be the space of real-valued
functions that are continuous and  $m$ times continuously
differentiable on $E$ with partial derivatives of
order up to and including $m$.  Also, let $\C_b^m(E)$ be
 the subspace of $\C^m(E)$ consisting of bounded functions whose partial derivatives of order up to
and including $m$ are also bounded, let $\C_c^m(E)$ be the subspace of
$\C^m(E)$ consisting of functions that vanish outside compact sets,
and
let $\C_0^m(E)$ be the subspace of $\C^m(E)$ consisting of functions
$f$ that vanish at infinity.
In addition, let $\C_c^m(E)\oplus\R$ be the direct sum of $\C_c^m(E)$
and the space of constant functions, that is,
the space of functions that are sums of functions in $\C_c^m(E)$ and
constants in $\R$.
Likewise, let $\C_0^m(E)\oplus\R$ be the space of functions that are
sums of functions in $\C_0^m(E)$ and constants in $\R$.
If $m=0$, we  denote $\C^m(E)$, $\C_b^m(E)$,  $\C_c^m(E)$, $\C_0^m(E)$
, $\C_c^m(E)\oplus\R$ and $\C_0^m(E)\oplus\R$ simply by $\C(E)$, $\C_b(E)$, $\C_c(E)$, $\C_0(E)$, $\C_c(E)\oplus\R$, and $\C_0(E)\oplus\R$, respectively.
When $E$ is the closure of a domain, $\C^m(E)$ is to be interpreted as
the collection of  functions in $\cap_{\ve > 0} \C^m
(E^\ve)$, where $E^\ve$ is an open $\ve$-neighborhood of $E$, restricted to $E$.
The support of a function $f$ is denoted by $\supp(f)$ and the gradient of $f$ is denoted by $\nabla f$.

The space of continuous functions on $[0,\infty)$ that take values in $\R^J$
is denoted by $\ccspace$,  the Borel $\sigma$-algebra of $\ccspace$ is denoted by $\MM$, and the natural filtration on $\ccspace$ is denoted by $\{\MM_t\}$. The Borel $\sigma$-algebra of $\overline G$ is denoted by $\MB(\overline G)$.

\section{A Class of Reflected Diffusions}
\label{sec-rediff}

In this section we introduce  the class of reflected diffusions that we consider.
Let $G$ be a nonempty connected domain in $\R^J$, and let $d(\cdot)$ be
a set-valued mapping defined on $\overline G$, such
that  each $d(x)$, $x \in \partial G$,
is a non-empty closed convex cone in $\R^J$ with vertex at
 the origin $\orig$, $d(x) =
 \{\orig\}$ for each
$x$ in $G^\circ$,  and the graph of $d(\cdot)$ is closed, that is, the set
$\{(x,v): x \in \overline{G}, v \in d(x) \}$ is a closed subset of
$\R^{2J}$.
Let $\MV$ be a subset of $\partial G$.
As shown in Section \ref{subs-egs},  $\MV$ will typically
be a (possibly empty) subset of the non-smooth parts of the boundary of the domain $G$ where $d(\cdot)$ is
not sufficiently well behaved.
For each function $f$ defined on $\R^J$, we say $f$ is constant in a
neighborhood of $\MV$ if for each $x \in \MV$, $f$ is constant
in some open neighborhood of $x$.
Given measurable drift and dispersion coefficients $b:\R^J \mapsto \R^J$ and
$\sigma:\R^J \mapsto \R^J \times \R^N$, and $a=\sigma \sigma^T:\
\R^J \mapsto \R^J \times \R^J$,  let $\ML$ be the associated
differential operator defined in (\ref{operL}).
One way of characterizing a reflected diffusion is through the so-called
submartingale problem.


\begin{defn} (Submartingale Problem) \label{def-smg}
A family $\{\Q_z,z\in \overline G\}$ of probability measures on
$(\ccspace,\MM)$ is a solution to the submartingale problem
associated with $(G,d(\cdot))$, $\cal V$, drift $b(\cdot)$ and
dispersion $\sigma(\cdot)$ if   for each $A\in \MM$,
the mapping $z \mapsto \Q_z(A)$ is $\MB(\overline G)$-measurable and for each $z\in \overline G$,
$\Q_z$ satisfies the following three properties:
\begin{enumerate}
\item[1.] $\Q_z(\omega(0)=z)=1$;
\item[2.] For every $t\in [0,\infty)$ and $f\in \C_c^2(\R^J)$ such
that $f$ is constant in a neighborhood of $\MV$
and $\left< d,\nabla f(x)\right>\geq 0$ for all $d\in d(x)$ and
$x\in \partial G$, the process
\be
\label{term-mart}
f(\omega(t))-\int_0^t\ML
f(\omega(u))\,du,   \quad t \geq 0,
\ee
is a $\Q_z$-submartingale on
$(\ccspace,\MM,\{\MM_t\})$;
\item[3.] For every $z \in \overline G$,
\[\E^{\Q_z}\left[\int_0^\infty
\ind_{\MV}(\omega(s))\,ds \right]=0. \]
\end{enumerate}
In this case,
$\Q_z$ is said to be a solution to the submartingale problem
starting from $z$. Moreover, given a probability distribution $\pi$ on
$\overline G$, the probability measure $\Q_\pi$, defined by
 \be \label{Qpi}\Q_\pi(A)=\int_{\overline G} \Q_z(A)\,\pi(dz), \qquad
 \mbox{  for every }A\in \MM, \ee
 is said to be  a solution to the submartingale problem with initial distribution $\pi$.
\end{defn}

The submartingale problem  is a generalization of the martingale
problem that was first introduced in \cite{StrVar06} to characterize the law of reflected
diffusions in smooth domains.
Extensions of the submartingale problem to characterize reflected
Brownian motions in  non-smooth domains  were considered in previous
works such as \cite{VarWil85} for two-dimensional wedges and
\cite{Wil87} and \cite{Ram06} for certain
convex polyhedral domains.     Definition \ref{def-smg} generalizes these
formulations  further to accommodate a more general class of
multi-dimensional reflected diffusions.
The first condition in Definition \ref{def-smg} simply
states that the family of measures is parameterized by the initial
condition.   The second condition in Definition \ref{def-smg}
captures the notion of diffusive behavior in
the interior, and reflection along the appropriate directions on the
boundary.
 Since the ``test functions'' in property 2 are constant in a neighbourhood of $\MV$, this condition
does not provide information on the behavior of the diffusion in a
neighborhood of $\MV$.  The third condition
is imposed to ensure instantaneous reflection on the boundary (precluding the possibility of
absorption or partial reflection at the boundary).   
The set $\MV$ is typically a subset of the non-smooth parts of the
boundary.   In many cases, there is some flexibility in the choice of 
$\MV$, with several choices yielding equivalent characterizations that
are all compatible with the stochastic
differential equation (with reflection) formulation.  
As described in Remark \ref{rem-MV}, there is a canonical choice for $\MV$  that is suitable for all the considered
examples.  
It is typically harder to establish uniqueness, rather than existence, 
of solutions to the submartingale problem.  
In many cases, it is easier to establish uniqueness, and therefore, 
well-posedness of the submartingale problem 
 (see Definition \ref{def-wellposed} below) 
when the set $\MV$ is smaller. 
For further discussion on the formulation of the 
submartingale problem in the non-smooth setting, see \cite{KanRam11b}. 

\begin{defn}
\label{def-wellposed}
The submartingale problem
associated with $(G,d(\cdot))$, $\MV$, drift $b(\cdot)$ and
dispersion $\sigma(\cdot)$ is said to be well posed if  there exists exactly one solution to the submartingale problem.
\end{defn}

We will only consider submartingale problems
that are  well posed.  In addition, we will also 
assume throughout, without explicit mention, that the drift and
diffusion coefficients are continuous. 
Under this assumption, for every $f \in
\C_c^2(\R^J)$, the mapping
$x \mapsto  \ML f(x)$ is continuous, and so the integral
in (\ref{term-mart}) is clearly well defined.

We now consider  reflected
diffusions associated to the submartingale problem.

\begin{defn}
A stochastic process $Z$ defined on a probability space  $(\Omega, {\cal F}, \P)$ is said to be a reflected diffusion associated with $(G,d(\cdot))$, $\MV$, drift $b(\cdot)$ and
dispersion $\sigma(\cdot)$ if its family of distribution laws $\{\Q_z,\ z\in \overline G\}$ is the unique solution to the submartingale problem, where $\Q_z$, $z\in \overline G$, is the conditional distribution of $Z$ under $\P$,  conditioned on $Z(0)=z$.
\end{defn}

Reflected diffusions are sometimes also defined in terms of weak
or strong solutions to stochastic differential equations using the
Skorokhod problem or extended Skorokhod problem.  
The relationship between this formulation and the submartingale
problem characterization is studied in \cite{KanRam11b}. 
In particular, 
it is shown that weak solutions to  a large class of stochastic differential equations in piecewise
smooth domains are associated with well posed submartingale problems.

For our subsequent analysis, it will
be convenient to also consider a localized version of the
submartingale problem.
  For each $r>0$, consider the stopping time
 \be \label{chi}\chi^r=\inf\{t\geq 0:\ |\omega(t)|\geq r\},\qquad \omega\in \ccspace. \ee

\begin{defn}\label{def-smgl} (Localized Submartingale Problem)
Fix $r>0$. A family $\{\Q_z^r,z\in \overline G\}$ of probability measures on
$(\ccspace,\MM)$ is a solution to the submartingale problem
associated with $(G,d(\cdot))$, $\cal V$, drift $b(\cdot)$ and
dispersion $\sigma(\cdot)$ that is stopped at the time $\chi^r$ if
for each $A\in \MM$, the map $z \mapsto \Q_z^r(A)$ is $\MB(\overline
G)$-measurable and for each $z\in \overline G$,
$\Q_z^r$ satisfies the following three properties:
\begin{enumerate}
\item[1.] $\Q_z^r(\omega(0)=z)=1$;
\item[2.] For every $t\in [0,\infty)$ and $f\in \C_c^2(\R^J)$ such
that $f$ is constant in a neighborhood of $\MV$
and $\left< d,\nabla f(x)\right>\geq 0$ for all $d\in d(x)$ and
$x\in \partial G$,  the process
\[f(\omega(t\wedge \chi^r))-\int_0^{t\wedge \chi^r}\ML
f(\omega(u))\,du, \qquad t \geq 0, \]
is a $\Q_z^r$-submartingale on
$(\ccspace,\MM,\{\MM_t\})$;
\item[3.] For every $z \in \overline{G}$,  \[\E^{\Q_z^r}\left[\int_0^{\chi^r}
\ind_{\MV}(\omega(s))\,ds \right]=0. \]
\end{enumerate}
In this case,
$\Q_z^r$ is said to be a solution to the submartingale problem
starting from $z$ that is stopped at the time $\chi^r$.
Moreover, given a probability distribution $\pi$ and $\overline G$,
the probability measure $\Q^r_\pi$, defined as in (\ref{Qpi}), but with $\Q$
replaced by $\Q^r$, is said to
 be  a solution to the submartingale problem stopped at $\chi^r$ with initial distribution $\pi$.
\end{defn}

\begin{remark} \label{WPL} {\em When the submartingale problem associated with $(G,d(\cdot))$, $\cal V$, drift $b(\cdot)$ and
dispersion $\sigma(\cdot)$ is well posed, for each $r > 0$, the submartingale problem
stopped at the stopping time $\chi^r$ is also well posed.
This can be justified by using an argument similar to the one
used in the proof of Theorem 6.1.2 of \cite{StrVar06}.}
\end{remark}

\section{Main Results}
\label{sec-mainres}

The primary goal of this work is to provide a useful characterization of the
stationary distributions of a broad class of reflected diffusions.

\begin{defn} \label{def:1}
A probability measure $\pi$ on $\overline G$ is a stationary
distribution for the unique solution $\{\Q_z, z\in \overline G\}$ to a
well posed submartingale problem if  $\pi$ satisfies the
property that the law of $\omega(t)$ under $\Q_\pi$ is $\pi$ for each
$t\geq 0$.  In this case, $\pi$ is also said to be a stationary distribution
of any reflected diffusion associated with the well posed submartingale
problem.
\end{defn}


The  main result of this paper is a necessary and sufficient condition for a
probability measure $\pi$ to be a stationary distribution for the well posed
submartingale problem.
In what follows, recall that  $\C_c^2(\overline{G})\oplus\R$ is the
space of functions that are sums of functions in  $\C_c^2(\overline{G})$ and constants in $\R$,
and that $\nabla f$ denotes the 
gradient of a function $f$ on a domain in $\R^J$.   Given a subset
$\MV \subset \partial G$, let $\MH  =
\MH_{\MV}$ be  the set of functions 
\be \label{dis:H}
\MH \doteq  \ds \left\{\begin{array}{ll} f\in \C_c^2(\overline G)\oplus \R:\ & \mbox{$f$ is constant in a neighborhood of $\MV$}, \\ & \mbox{$\left<d, \nabla
f(y)\right>\leq 0$ for $d\in d(y)$ and $y\in \partial G$} \end{array}\right\}. \ee
It is easy to see that if the unique solution $\{\Q_z,z\in \overline
G\}$ to a well posed submartingale problem associated with
$(G,d(\cdot))$ and $\MV$ admits a stationary
distribution $\pi$, then $\pi$ must satisfy the inequality
(\ref{mono0})
 for all $f$ such that $-f\in \MH$,
where $\ML$ is the operator defined in (\ref{operL}).
Indeed, it follows from the second property in Definition
\ref{def-smg} that for each $f$ with $-f\in \MH$,
\[\E^{\Q_\pi}\left[f(\omega(t))-\int_0^t\ML
f(\omega(u))\,du\right]\geq \E^{\Q_\pi}\left[f(\omega(0))\right]. \]
Since  $\E^{\Q_\pi}\left[f(\omega(t))\right] = \E^{\Q_\pi} \left[
  f(\omega(0))\right]$ due to the stationarity of $\pi$, this
 establishes the inequality in (\ref{mono0}) for all functions $f$
 with $-f \in \MH$.
We will show that, under  two assumptions stated below,
 the  later condition is also  sufficient for any probability
 measure $\pi$ with $\pi(\partial G)=0$ to be a stationary
 distribution of $\{\Q_z,z\in \overline G\}$.

The first assumption is the existence of a family of
test functions, which is used in Section
\ref{subs:tight} to establish the tightness
of a  certain sequence of approximating probability measures.

\begin{assumption}\label{ass:TF}
For each $\varepsilon>0$ and $N\in \N$, there exists a family of
nonnegative functions $\{f_{x,\varepsilon} \in \C^2_b(\overline G),\
x\in \overline G\cap B_N(0)\}$ and constants $C(N,\varepsilon)>0$ and
$c(N,\varepsilon)>0$ such that the following three properties hold:
\begin{enumerate}
\item[1.] $f_{x,\varepsilon}$ is a finite or countable sum of functions in $\MH\cap \C^2_c(\overline G)$;
\item[2.]  For each $y\in \overline G$, $f_{x,\varepsilon}(y)=0$ if $|y-x|\leq \varepsilon/2$, and $f_{x,\varepsilon}(y)>c(N,\varepsilon)$ if $|y-x|>3\varepsilon$;
\item[3.] $|\ML f_{x,\varepsilon}(y)|\leq C(N,\varepsilon)$ for each $y\in \overline G$.
\end{enumerate}
\end{assumption}

The second assumption concerns the geometry of the set ${\cal V}$.

\begin{assumption} \label{ass:V}
The set $\MV$ is finite and for each $x\in \MV$, there exist a unit
vector $\nvec_x \neq 0$ and positive  constants $\alpha_x, r_x$ such
that $\lan \nvec_x,y-x\ran \geq  \alpha_x|y-x|$, 
$\nvec_x^T a (y) \nvec_x \geq \alpha_x$ 
and $\lan\nvec_x,d\ran\geq 0$ for all $d\in d(y)$ and all $y\in \overline G
\cap B_{r_x}(x)$.  
\end{assumption}

We now state the first main result of this paper. Its proof is given
in Section \ref{sec-proof1}.
Recall that we assume throughout that the drift and diffusion
coefficients are continuous.

\begin{theorem} \label{thm:SS}  Suppose we are given
$(G, d(\cdot))$, $b(\cdot)$, $\sigma (\cdot)$ and $\MV$ such that
the associated submartingale problem is well posed and
 Assumptions  \ref{ass:TF} and \ref{ass:V} are satisfied.
Let $\pi$ be a probability measure on $(\overline G,\MB(\overline G))$
with $\pi(\partial G)=0$.  Then $\pi$ satisfies the inequality (\ref{mono0})  
 for all  $-f\in \MH$ if and only if
$\pi$ is a stationary distribution for the unique solution to the
associated submartingale problem.
\end{theorem}

The necessity of the condition stated in Theorem
\ref{thm:SS} is, as shown above, straightforward.
On the other hand, the proof of  sufficiency, which is given in Section \ref{sec-proof1},
entails first constructing  
 a suitable sequence of Markov chains with stationary
distribution $\pi$ and then showing that the sequence converges to a solution
of the submartingale problem.
Both steps are more involved when the domain is not smooth, 
in part due to the difficulty of estimating the amount of time the reflected
diffusion spends near the boundary, especially in the neighborhood of
the set $\MV \subset \partial G$.  

We now introduce a broad class of data $(G, d(\cdot))$ and
 $\MV$ for which the stationary distribution characterization 
obtained in Theorem \ref{thm:SS} applies. 

\begin{defn} \label{def:A}
The pair $(G,d(\cdot))$ is said to be piecewise ${\cal C}^1$ with
continuous reflection if $G$ and $d(\cdot)$ satisfy the following
properties: 
\begin{enumerate}
\item[1.] The domain $G$ is a non-empty domain with
representation $ G=\bigcap_{i\in \MI}G_i$, where $\MI$ is a finite
index set and for each $i\in \MI$, $G_i$ is a non-empty domain with 
$\C^1$ boundary, that is, there exists a continuously differentiable function $\func^i$
on $\R^J$ such that 
\[   G_i = \{ x: \func^i(x) > 0 \} \quad \mbox{ and } \quad \partial
G_i = \{ x: \func^i(x) = 0 \}.  \]
  Let $n^i(x)$ denote the unit inward normal vector to
$\partial G_i$ at $x\in \partial G_i$  and
define
\[  \MI(x)  \doteq  \{i \in \MI:\ x\in \partial G_i\}, 
\]
and  for each $x\in \partial G$, let 
\[ n(x) \doteq   \left\{\sum_{i\in \MI(x)}s_in^i(x), s_i\geq 0,  i\in
  \MI(x)\right\}.
\]
\item[2.]  The direction vector field $d(\cdot)$ is given by
\[   d(x) \doteq   \left\{\sum_{i\in \MI(x)}s_i\gamma^i(x),\ s_i\geq 0, i\in
   \MI(x) \right\}, \qquad x \in \partial G,
\]
where
for each $i\in \MI$,  $\gamma^i(\cdot)$ is a continuous vector field
  associated with $G_i$  such that $\lan n^i(x),\gamma^i(x)\ran
 >0$ for each $x\in \partial G_i$.
\end{enumerate}
If, in addition,  $n^i(\cdot)$ is constant (so that the domain $G$ is a polyhedron)
and $\gamma^i(\cdot)$ is also constant, then the pair $(G,d(\cdot))$ will be said to be 
polyhedral with piecewise constant reflection. 
\end{defn}

Given $(G,d(\cdot))$ that is piecewise ${\cal C}^1$ with continuous
reflection,  we can assume, without loss of generality that
for each $i \in \MI$, 
\be
\label{normalize}
 \lan n^i(x),\gamma^i(x)\ran =1, \qquad \mbox{ for each }
x\in \partial G_i. 
\ee
In the analysis of reflected diffusions in non-smooth domains, 
a special role is played by the following set on the boundary: 
\be
\label{def-mu}
 \MU \doteq \{x \in \partial G: \exists n \in n(x)  \mbox{ such that }
  \lan n, d \ran > 0,\  \forall d \in d(x)\sm \{0\} \}. 
\ee

\noi 
{\bf Assumption 2'.}
$\MV$ is a finite set such that $\MV \supset \partial G \sm \MU$,  
and for each $x\in \MV$, there exist a unit vector $\nvec_x$,  
  constants $r_x>0$, $\alpha_x>0$ and $0<c^1_x<1< c_x^2<\infty$ such
  that for all $y\in \overline G \cap B_{r_x}(x)$, the following properties are
  satisfied with  $\Theta_x \doteq \{z\in \R^J:\ \lan z, \nvec_x \ran \geq
  0\}$: 
\begin{enumerate}
\item $\lan \nvec_x,y-x\ran \geq  \alpha_x|y-x|$; 
\item $\gamma^i(y)\in \newset_x$ for
  each $i\in \MI(y) \subset \MI(x)$;
\item
for every $r\in (0,r_x/c_x^2)$,
\end{enumerate}
\begin{eqnarray}
\label{include}
 \overline G \cap B_{c^1_xr}(x) & \subseteq &  \{y\in
\overline G\cap B_{r_x}(x):\ \dist(y,x+r \nvec_x+\Theta_x)>0\} 
\\
\nonumber
&\subseteq & \overline G \cap B_{c^2_xr}(x). 
\end{eqnarray}
\begin{enumerate}
\item[(4)]
$\nvec_x^T a (y) \nvec_x \geq \alpha_x$ for all $y \in
  B_{r_x}(x)$;  
\end{enumerate}

\begin{remark}
\label{rem-MV}
{\em  Note that 
Assumption 2' is trivially satisfied when $\partial G=\MU$, 
and $\MV = \emptyset$. 
The  condition that $\partial G = \MU$ 
ensures that at each $x \in \partial G$ there
  exists a normal vector $n \in   n(x)$ that makes a strictly positive inner
  product with all unit reflection vectors $d \in d(x)$.  
It can be viewed as 
a generalization of what is 
  known in the literature as the  completely-${\cal S}$
  condition. 
A canonical choice is to define $\MV$ to be 
equal to the set of points on $\partial G$ at which 
this condition fails to hold.  
In the context of certain polyhedral domains with piecewise constant
reflection fields, the stipulation  that the  
completely-${\cal S}$ condition hold has
  been shown to be necessary and sufficient for the associated
  reflected diffusion to be a semimartingale \cite{Ram06,
    ReiWil88,TaWi93}.  However,  in this work we also allow for cases when 
$\partial G \neq \MU$, 
thus providing a characterization of the stationary distribution for
  reflected diffusions that are not necessarily semimartingales
  \cite{BurKanRam08, KanRam10, Ram06}.
}
\end{remark}

We now state the second main result of this paper, whose proof is
given  in Section \ref{sec-proof2}.

\begin{theorem}  \label{thm:test}  Suppose that either $(G,d(\cdot))$ is  
  piecewise 
${\cal C}^1$ with continuous reflection and $G$ is bounded, 
or 
$(G, d(\cdot))$ is polyhedral with piecewise constant reflection. 
If Assumption 2' is satisfied and 
 the drift   and dispersion coefficients 
  $b(\cdot)$ and  $\sigma(\cdot)$ are uniformly
  bounded,   then $(G,d(\cdot))$ and $\MV$ satisfy Assumptions
  \ref{ass:TF} and \ref{ass:V}.  
\end{theorem}
\noi 
As an immediate consequence we see that Theorem
\ref{thm:SS} can be used to characterize the stationary distributions 
of reflected diffusions that satisfy the conditions of Theorem
\ref{thm:test}  and are associated with well-posed submartingale 
problems.  As shown in Section \ref{subs-egs}, 
this includes many classes of reflected diffusions that arise in applications.   
But Theorem \ref{thm:SS} is also applicable to reflected diffusions
outside the class described by Theorem \ref{thm:test}.   An illustrative example is given 
in Example \ref{ex:cusp} of Section \ref{subs-egs}.

To state the last main  result of this paper, we introduce $\ML^*$,  the
adjoint operator to $\ML$: for $p \in \C^2(\overline G)$, 
\[\ML^*p(x)=\frac{1}{2}\sum_{i,j = 1}^J \frac{\partial^2}{\partial
x_i\partial x_j}(a_{ij}(x)p(x))-\sum_{i=1}^J \frac{\partial}{\partial
x_i}(b_i(x)p(x)). \]
We now show that nonnegative and integrable solutions of the so-called basic adjoint relation
(BAR) are indeed stationary distributions for the submartingale
problem.
  In what follows, let ${\mathcal S}$ denote the smooth
parts of the boundary $\partial G$.

\begin{corollary} \label{Thm:SBR}
Suppose that the pair $(G,d(\cdot))$  is piecewise ${\cal C}^1$ with
continuous 
reflection,  (\ref{normalize}) is satisfied, $\MV \subset \partial G$
satisfies 
Assumption 2' 
and $b_i(\cdot),
a_{ij}(\cdot)\in \C^2(\overline G)$ for $i, j  = 1, \ldots, J$, and
the submartingale problem associated with 
$(G,d(\cdot))$ and $\MV$ is well posed. 
Furthermore, suppose there exists a non-negative
function $p\in \C^2(\overline G)$ that  satisfies $\int_{\overline G}
p(x) dx <\infty$ and the following partial differential equation with boundary conditions:
\begin{enumerate}
\item[1.] $\ML^*p (x) =0$ for  $x \in G$;
\item[2.] for each $i\in \MI$ and $x\in \partial G_i\cap {\mathcal S}$,
\begin{eqnarray*} && -2p(x)\left<n^i(x),b(x)\right>+(n^i(x))^Ta(x)\nabla p(x)+p(x)K_i(x) \\ && \qquad
-\nabla\cdot(p(x)(n^i(x))^Ta(x)n^i(x)\gamma^i(x)-p(x)a(x)n^i(x))=0, \end{eqnarray*} where
\[ K_i(x) \doteq \left<n^i(x),\sum_{j=1}^J \frac{\partial a_{\cdot
j}(x)}{\partial x_j}\right> = \sum_{k=1}^J n^i_k (x) \sum_{j=1}^J
\frac{\partial a_{kj}}{\partial x_j} (x);
\]
\item[3.] for each $i, j \in \MI$, $i \neq j$ and $x\in \partial G_i\cap \partial G_j \cap \partial G\setminus \MV$,
\begin{eqnarray*}
&& p(x)\left(\left<n^j(x),(n^i(x))^Ta(x)n^i(x)\gamma^i(x)-a(x)n^i(x)\right> \right. \\ && \qquad +\left.
\left<n^i(x),(n^j(x))^Ta(x)n^j(x)\gamma^j(x)-a(x)n^j(x)\right>\right)=0.
\end{eqnarray*}
\end{enumerate}
Then the probability measure on $\overline G$ defined by
\be
\label{def-corpi}
\pi(A) \doteq \ds \dfrac{\int_A p(x)dx}{\int_{\overline G} p(x) dx},
\qquad   A\in \MB(\overline G),
\ee
 is a stationary distribution for the well posed submartingale problem.
\end{corollary}
\proof By Theorems \ref{thm:SS} and \ref{thm:test}, it suffices to show that the
probability measure $\pi$ defined
in terms of $p$ via
 (\ref{def-corpi}) satisfies the inequality (\ref{mono0}) for all
 functions $f$ such that $-f\in \MH\cap \C^2_c(\overline G)$.  For any
 such function $f$,  straightforward calculations show that for each $x\in \overline G$,
\begin{eqnarray*} p(x)\ML f(x)-f(x)\ML^*p(x)
=\frac{1}{2}\nabla \cdot H(x),\end{eqnarray*} where $H = H^f$ is the vector
field such that its $i$th component is given by
\begin{eqnarray*}
H_i(x) & = & \sum_{j=1}^J\left(p(x)a_{ij}(x)\frac{\partial
f(x)}{\partial x_j}  -f(x)a_{ij}(x)\frac{\partial p(x)}{\partial
x_j}-f(x)p(x)\frac{\partial a_{ij}(x)}{\partial
x_j}\right)\\
&& \qquad +2b_i(x)f(x)p(x).
\end{eqnarray*}
Since $\ML^* p (x)= 0$ for  $x \in G$, the Divergence
Theorem  implies that
\begin{eqnarray} \label{statd}\int_{\overline G} \ML f(x) p(x)dx &=&
  \frac{1}{2} \int_{\partial G} \lan n(x), H(x) \ran\mu(dx) \\ & = &
-\frac{1}{2}\sum_{i \in \MI} \int_{\partial G_i\cap \partial G}
\left<n^i(x), H(x) \right> d\mu_i(x), \nonumber
\end{eqnarray}
where $n(\cdot)$ is the outward pointing unit normal field on $\partial G$, $\mu(dx)$ is the surface measure on $\partial G$, and $\mu_i(dx)$ is the surface measure on $\partial G\cap \partial G_i$ for each $i\in \MI$.
Note that for each $i\in \MI$ and $x\in \partial G\cap \partial G_i$,
\begin{eqnarray*}
 \left<n^i(x), H(x) \right>  &=& \quad p(x)(n^i(x))^Ta(x)\nabla f(x)
 -f(x)(n^i(x))^Ta(x)\nabla p(x)\\ && \qquad \qquad-f(x)p(x)K_i(x) +
 2f(x)p(x)\left<n^i(x),b(x)\right>.
\end{eqnarray*}
Combining the above display, (\ref{statd}) and condition 2 in the
theorem, we obtain
\[
\begin{array}{l} 
\ds \int_{\overline G} \ML f(x) p(x)dx  \\  =
\ds - \frac{1}{2}\sum_{i\in \MI} \int_{\partial G_i\cap \partial G}
p(x)(n^i(x))^Ta(x)\nabla f(x) d\mu_i(x) \\
 \ds \qquad  +\frac{1}{2} \sum_{i\in \MI} \int_{\partial G_i\cap \partial
   G} f(x) \nabla\cdot(p(x)(n^i(x))^Ta(x)n^i(x)\gamma^i(x) -p(x)a(x)n^i(x)) d\mu_i(x).
\end{array}
\]
Note that for each $i\in \MI$ and $x\in \partial G_i\cap \partial G$,
\begin{eqnarray*}
&& \nabla\cdot(f(x)p(x)[(n^i(x))^Ta(x)n^i(x)\gamma^i(x)-a(x)n^i(x)]) \\ && \quad =f(x)\nabla\cdot(p(x)[(n^i(x))^Ta(x)n^i(x)\gamma^i(x)-a(x)n^i(x)])\\ && \qquad + \lan \gamma^i(x),\nabla f(x)\ran p(x) (n^i(x))^Ta(x)n^i(x) -p(x) (\nabla f(x))^Ta(x)n^i(x).
\end{eqnarray*}
It follows from the last two equalities that
\begin{eqnarray*}  && \int_{\overline G} \ML f(x) p(x)dx \\ \ &=&
\frac{1}{2}\sum_{i \in \MI} \int_{\partial G_i\cap \partial G}
\nabla\cdot(f(x)p(x)((n^i(x))^Ta(x)n^i(x)\gamma^i(x)-a(x)n^i(x)))
d\mu_i(x) \nonumber \\ &&  \quad -\frac{1}{2} \sum_{i \in \MI} \int_{\partial G_i\cap \partial G} \lan \gamma^i(x),\nabla f(x)\ran p(x) (n^i(x))^Ta(x)n^i(x) d\mu_i(x). \nonumber
\end{eqnarray*}
The second term on the right-hand side of the above equality is non-positive since $-f\in \MH$, $p\geq 0$ and $a$ is positive semidefinite.
So, we shall focus on the first term on the right-hand side of the above display.
Now, for each $x\in \partial G_i\cap \partial G$,
$\lan n^i(x),(n^i(x))^Ta(x)n^i(x)\gamma^i(x)$ $-a(x)n^i(x))\ran$ $=0$
because of the assumed normalization $\lan n^i(x), \gamma^i(x) \ran =
1$.  Therefore,  the vector
$(n^i(x))^Ta(x)n^i(x)\gamma^i(x)-a(x)n^i(x)$  is parallel to $\partial
G_i$ at $x$, and the
divergence in the first term on the right-hand side of the above
display is equal to the
divergence taken in the $(J-1)$-dimensional manifold $\partial G_i\cap \partial G$.
Another application of the Divergence Theorem  then yields
\[  \begin{array}{l}
-\ds \sum_{i \in \MI} \int_{\partial G_i\cap \partial G}
\nabla\cdot(f(x)p(x)((n^i(x))^Ta(x)n^i(x)\gamma^i(x)-a(x)n^i(x)))
d\mu_i(x) \\  \ds =
\sum_{i,j \in \MI, i \neq j}\int_{F_{ij} \sm \MV}
f(x)p(x)\left<n^{ij}(x),(n^i(x))^Ta(x)n^i(x)\gamma^i(x)-a(x)n^i(x)\right>d\mu_{ij}(x),
\end{array}
\]
 where $F_{ij} \doteq \partial G_i\cap
\partial G_j\cap \partial G$,
$n^{ij}(x)$ denotes the unit vector that is
normal to both $F_{ij}$  and $n^i(x)$ at $x$ and points
into $\partial G_i\cap {\mathcal S}$ from $F_{ij}$, and $\mu_{ij}(dx)$ is the
surface measure on the $(J-2)$-dimensional manifold $F_{ij}$.
To prove the theorem, it suffices to show that the last equality in
the above display is zero.  To do this, it suffices to show that for
each $i,j\in \MI$ with $i\neq j$ and $x\in F_{ij}
\setminus \MV$,
\begin{eqnarray} \label{statd2}
&& p(x)\left(\left<n^{ij}(x),(n^i(x))^Ta(x)n^i(x)\gamma^i(x)-a(x)n^i(x)\right>\right. \\ && \qquad \left.+ \left<n^{ji}(x),(n^j(x))^Ta(x)n^j(x)\gamma^j(x)-a(x)n^j(x)\right>\right)=0. \nonumber
\end{eqnarray}
Since $n^{ij}(x)$ is normal to $\partial G_i\cap
\partial G_j$ at $x \in \partial G_i\cap \partial G_j$,  it must lie in the two dimensional space spanned by $n^i(x)$ and $n^j(x)$. In addition, $n^{ij}(x)$ is a unit vector normal to $n^i(x)$ and points into $\partial G_i$ from $\partial G_i\cap
\partial G_j$. Then we have \[n^{ij}(x)=(n^j(x)-\lan n^i(x),n^j(x)\ran n^i(x))/(1-\lan n^i(x),n^j(x) \ran^2)^{1/2}\] and \[n^{ji}(x)=(n^i(x)-\lan n^i(x),n^j(x)\ran n^j(x))/(1-\lan n^i(x),n^j(x) \ran^2)^{1/2}.\]
With the above two representations, we see that (\ref{statd2}) is
equivalent to condition (3) in the theorem. This yields the desired
result.  \endproof
\bigskip

\section{Examples}
 \label{subs-egs}

In this  section, we provide several examples of reflected
diffusions in general domains for which the submartingale
problem is well posed and Assumptions \ref{ass:TF} and \ref{ass:V} are
satisfied, so that Theorem \ref{thm:SS} provides a characterization
of their stationary distributions.   The examples serve to illustrate 
the range of applicability of the main result of the paper.  In the first three examples, the
reflected diffusions are semimartingales, whereas in the last two
examples they fail to be semimartingales,  two of the examples deal
with curved domains and three of the examples are multi-dimensional. 
  For simplicity, in all cases  we assume that $b$ and 
$\sigma$ are continuous and uniformly bounded and that when 
${\mathcal V} \neq \emptyset$,  $a$ is uniformly elliptic (although only 
the partial uniform ellipticity in the direction $v_x$ of Assumption
\ref{ass:V} is actually required).

We start with the simple case of smooth domains.

\begin{example}
\label{ex-zero}
{\em  Let $G$ be a bounded open set in $\R^J$ such that $G=\{x\in \R^J:\ \phi(x)>0\}$, where $\phi\in \C^2_b(\R^J)$ and $|\nabla \phi|\geq 1$ on $\partial G$. Thus, $\nabla \phi(x)$ is the
inward normal vector  at $x \in \partial G$. Let
$\gamma(\cdot)$ be a bounded Lipschitz continuous vector field that satisfies $\lan \nabla \phi (x), \gamma(x)  \ran > 0$ on $\partial G$,
and let $\MV = \emptyset$. Then it trivially follows that
$(G,d(\cdot))$ is a ${\cal C}^1$ domain with
continuous reflection and, since $\MU = \partial G$, Assumption 2' is trivially
satisfied with $\MV = \emptyset$.   By \cite{StrVar71} (see Theorems
3.1 and 5.4 therein) 
the associated submartingale problem is well posed. Moreover, since $G$ is bounded, Theorem \ref{thm:test}  shows that Assumptions \ref{ass:TF} and
\ref{ass:V} are satisfied (the latter holding trivially).
 }
\end{example}

\begin{example}
\label{ex-one}
{\em Consider a two-dimensional wedge $G$ given in polar coordinates
  by \[G=\{(r,\theta):\ 0\leq \theta\leq \zeta,\ r\geq 0\},\] where
  $\zeta\in (0,\pi)$ is the angle of the wedge.  Then $G$ admits the
  representation $G = G_1\cap G_2$, where $G_1$ and $G_2$ are the two half
  planes
\begin{eqnarray*}
G_1 & = & \{(r,\theta):\ 0\leq \theta\leq \pi,\
r\geq 0\}, \\
G_2 & = & \{(r,\theta):\ \zeta-\pi\leq \theta\leq
\zeta,\ r\geq 0\}.
\end{eqnarray*}
   Let the directions of reflection
on $\partial G_1$ and $\partial G_2$ be specified as constant vectors
$\gamma^1$ and $\gamma^2$ normalized such that for $j=1,2$, $\lan
\gamma^j,n^j\ran =1$. For $j=1,2$, define the angle of reflection
$\theta_j$ to be the angle between $n^j$ and $\gamma^j$, such that
$\theta_j$ is positive if and only if $\gamma^j$ points towards the
origin. Note that $-\pi/2<\theta_j<\pi/2$. Define
$\alpha=(\theta_1+\theta_2)/\zeta$.  
It was proved in \cite{VarWil85} that the submartingale problem is well posed when $\alpha<2$.

Let $\MV  = \{0\}$.  
For every $\alpha < 2$, it is easily verified that $\MU$ contains $\partial G\sm \{0\}$, and so $\MV$
satisfies the 
 first property stated in  Assumption 2'.  
In fact, when $\alpha < 1$,  $\MU = \partial G$  and so 
 there exists a positive combination $n_0$ of $n^1$ and $n^2$ such that
 $\lan n_0,\gamma^j \ran >0$ for $j=1,2$. Since $\zeta\in (0,\pi)$, it
 is obvious that $\lan n_0, x \ran >0$ for each $x\in \overline G$. 
On the other hand, when $\alpha = 1$,  $\gamma^1$ and
$\gamma^2$ are pointing towards each other.  In this case, $G$ lies on  one side of the line through the
vertex that is parallel to $\gamma^1$ and $\gamma^2$. Let $n_0$ be a
unit vector perpendicular to the line that points into 
the half-space that contains $G$, that is, let $\lan n_0, x \ran > 0$
for each $x\in \overline G\sm \{0\}$.
In either case, let $\newset (0)=\{y\in \R^2:\ \lan n_0,y\ran \geq
0\}$. 
It can be verified that the properties of Assumption 2' hold with the choice of $r_0=\infty$, $\alpha_0=\inf\{\left<n_0,x\right>:\ |x|=1, x\in \overline G\}>0$, $ c_0^1=1/2$ and $c_0^2=\max\{1/\sin(\zeta+\theta_1), 1/\sin(\theta_1)\}$. Thus, $\MV = \{0\}$ satisfies Assumption 2'. 
 Since $(G, d(\cdot))$ is polyhedral with constant reflection,  Assumptions \ref{ass:TF} and \ref{ass:V} hold by Theorem \ref{thm:test}.
}
\end{example}

\begin{example}
\label{ex-orthant}
{\em We now describe another class of RBMs in the orthant that
arise as diffusion approximations of queueing networks \cite{ReiWil88}
and  of Leontief systems in economics \cite{Ramsub00}.
In this case, $G = \R_+^J$ is the non-negative orthant in $\R^J$,
which admits the representation $G = \bigcap_{i=1}^J G_i$, where
$G_i \doteq \{x \in \R^J: x_i \geq 0\}$, and the direction vector
field $\gamma^i$ on $G_i$ is a constant vector field, pointing in a  direction
$d^i \in \R^J$.   Moreover, the matrix $D$ with column $d^i$ is
assumed to satisfy the
completely-${\cal S}$ condition, which implies that $\MU = \partial
G$.  Thus, by Remark \ref{rem-MV}, Assumption 2' is trivially
satisfied with $\MV  = \emptyset$. 
In this case, it was shown in  \cite{TaWi93} that 
 the stochastic differential equation
with reflection associated with  $G$ and $d(\cdot)$ admits a weak
solution that is unique in law.  By Theorem 2 of
\cite{KanRam11b} it then follows that the submartingale problem 
is well posed. Moreover, since $(G,d(\cdot))$ is polyhedral, once again Theorem \ref{thm:test} shows that
Assumptions \ref{ass:TF} and \ref{ass:V} are satisfied.  
}
\end{example}

\begin{example} \label{ex:GPS} {\em Generalized processor sharing (GPS) is a service
discipline used in high-speed networks that allows for  efficient
sharing of a single resource amongst traffic of different classes.
It was shown in \cite{dupram2,dupram4,dupram5,RamRei03} that the GPS discipline can be
modelled in terms of an associated (extended) Skorokhod problem or,
equivalently, a submartingale problem.   We now introduce a class of
RBMs that were shown  in \cite{RamRei03} and
\cite{RamRei08}  to arise as reflected diffusion approximations of
multiclass queueing
networks using this scheduling discipline.  The two-dimensional case
also corresponds to the case $\alpha = 1$ and $\pi = 90^\circ$ in Example \ref{ex-one}.

The state space $G$ associated with the GPS
ESP has the representation \[ G=\bigcap_{i=1}^{J+1}\{x\in \R^J:\
\left<x,n^i\right>>0\},\] where $n^i = e_i$ for $i = 1, \ldots, J$
(here $\{e_i, i = 1, \ldots, J\}$ is the standard orthonormal basis
in $\R^J$) and $n_{J+1} = \sum_{i=1}^J e_i/\sqrt{J}$. The reflection
vector field is piecewise constant on each face, governed by the
vectors
 $\{\gamma^i,\ i=1,...,J+1\}$ that are defined as follows:
$\gamma^{J+1} = \sum_{i=1}^J e_i/\sqrt{J}$ and $\{\gamma^i, i = 1,
\ldots, J\}$ are defined in terms of a ``weight'' vector
$\bar{\alpha} \in \R^J_+$ that satisfies $\bar \alpha_i>0$ for each
$i=1,...,J$ and $\sum_{i=1}^J \bar{\alpha}_i = 1$: for $i,j = 1,
\ldots, J$,
\[
\gamma^{i}_j = \left\{ \ba{rl}
-\dfrac{\bar{\alpha}_j}{1 - \bar{\alpha}_i} & \mbox{ for } j \neq i, \\
1 & \mbox{ for } j = i. \ea \right.
\]
It is easily verified that the generalized completely-${\cal S}$
condition is satisfied at all points in $\partial G \sm \{0\}$ (see Lemma 3.4  of
\cite{Ram06}).   Therefore,  $\MV = \{0\}$ satisfies the first
property in Assumption 2'. 
 Moreover, note that $d(0)=\{x\in \R^J:\ \lan x, d_{J+1} \ran \geq
 0\}$.
Choose $n_0=d_{J+1}$, $r_0=\infty$, $\alpha_0=1$, $c_0^1=1$, and $c_0^2=\sqrt{J}$ and $\newset(0)=d(0)$. Then
$\lan n_0,y\ran \geq  c_0|y|$ for every $y\in \overline G$ and $\overline G\cap B_{c_0^1r}(0) \subseteq \{ y\in
\overline G:\ \dist(y,rn_0+\newset(0))>0\}\subseteq B_{c_0^2r}(0)$ for each
$r\geq 0$. Then $\MV$ satisfies Assumption 2' and  $G, d(\cdot)$ is
polyhedral.   
Therefore,  Theorem
\ref{thm:test}  shows that Assumptions \ref{ass:TF} and \ref{ass:V}
hold.  The fact that the associated stochastic differential equation
with reflection has a pathwise unique solution  follows from
Corollary 4.4 of \cite{Ram06}.  Hence, well-posedness of the submartingale
problem follows from Theorem 2 of  \cite{KanRam11b}. }
\end{example}

\begin{example}\label{ex:cusp}
{\em Consider a two-dimensional domain $G$ with representation \[G=
  \{(x,y):\ x\geq 0,\ -x^\beta<y<x^\beta\},\qquad \beta>1.\] The
  domain $G$ has a cusp at the origin and $G=G_1\cap G_2$,
  where \begin{eqnarray*} G_1&=&\{(x,y):\ y<x^\beta \mbox{ when }
    x\geq 0 \mbox{ and } y<0 \mbox{ when } x<0\}, \\ G_2&=&\{(x,y):\
    y>-x^\beta \mbox{ when } x\geq 0 \mbox{ and } y>0 \mbox{ when }
    x<0\} .\end{eqnarray*} For each $j=1,2$ and $z\in \partial G_j$,
  let $n^j(z)$ be the inward unit normal vector to $\partial G_j$ and
  let $\gamma^j(z)$ make a constant angle $\theta_j\in (-\pi/2,\pi/2)$
  with $n^j(z)$. We take $\theta_j > 0$ if and only if the first
  component of $\gamma^j(z)$ is negative, that is, $\gamma^j(z)$
  points towards the origin, in small neighborhoods of the origin.
  Since $\theta_j \neq \pm \pi/2$, we can without loss of generality
  assume the normalization $\lan \gamma^j(z), n^j(z)\ran =1$ holds.  
 It was proved in \cite{DT} that the submartingale problem is well posed when $\theta_1+\theta_2\leq 0$.

With the choice $\MV=\{0\}$, it is clear that the first
condition of Assumption 2' is satisfied.  To see that the second set
of  properties is also satisfied,  let $n_0$ be a unit vector that is
perpendicular to $\gamma^1(0)$ and points towards $G$, and as usual
let $\newset(0)=\{x\in \R^2:\ \lan n_0,x\ran \geq 0\}$. It is easy to
see that $\gamma^j(0)\in \newset(0)$ for $j=1,2$ and  that there exist
$r_0>0$, $\alpha_0>0$, $c_0^1>0$ and $c_0^2>0$ such that $\lan
n_0,y\ran \geq  \alpha_0|y|$ for every $y\in \overline G\cap
B_{r_0}(0)$ and $\overline G \cap B_{c_0^1r}(0)\subseteq \{y\in
\overline G\cap B_{r_0}(0):\ \dist(y,rn_0+\newset(0))>0\}\subseteq
B_{c_0^2r}(0)$ for each $r\in (0,r_0/2)$.  Thus,  $(G,d(\cdot))$ is clearly piecewise ${\cal C}^1$ with continuous
reflection and $\MV$ satisfies Assumption 2', which immediately
implies Assumption \ref{ass:V} is satisfied. 
Since $G$ is neither bounded nor a polyhedral domain,  Theorem
\ref{thm:test} cannot be applied directly to show that Assumption
\ref{ass:TF} holds. But a similar argument as the one used in the proof of
Theorem \ref{thm:test} can be used to establish Assumption
\ref{ass:TF}. The details are  deferred to  Appendix \ref{app:cusp}.
}
\end{example}

\section{Proof of Theorem \ref{thm:SS}}
\label{sec-proof1}

Let $\pi$ be a probability measure on
$(\overline G,\MB(\overline G))$ such that
$\pi(\partial G)=0$.
The discussion prior to the statement of Theorem \ref{thm:SS} shows
that  (\ref{mono0})  is a necessary condition for $\pi$ to be a
stationary distribution of the
submartingale problem.
 So, it only remains to prove the sufficiency of the condition
(\ref{mono0}).
For the remainder of this section we assume that $\pi$ has the
additional property that it satisfies (\ref{mono0}) for every
$-f\in \MH$, and
 prove that then $\pi$ is a stationary distribution
for the well posed submartingale problem.
   The proof consists of two main steps.   In the first step, which is carried out in
Section \ref{sec:MC},   a sequence of discrete time Markov
chains is constructed such that $\pi$ is the stationary distribution for each Markov
chain  in the sequence.  In the second step, which is presented in
Sections \ref{subs:loc}--\ref{subs:pfss},  it is first shown that the sequence of continuous time
extensions of the Markov chains converges to
the unique solution of the submartingale problem with initial
distribution $\pi$.   The fact that $\pi$ is a stationary distribution
for each Markov chain in the
approximating sequence is then used to deduce that $\pi$ is a stationary
distribution for the submartingale problem.

\subsection{Construction of Markov Chains} \label{sec:MC}

In what follows, let $\MH$ be the class of functions defined in (\ref{dis:H})
and let $\lambda>0$ be a given constant.
We shall construct a discrete time Markov chain
$\{X^\lambda(n\lambda),\ n\in \Z_+\}$  that has $\pi$ as its
stationary distribution and satisfies an inequality that is analogous
to property 2 of the submartingale problem.

We first state a preliminary lemma.
This lemma is similar to Lemma 3 of \cite{Weiss81}, but
 the class of functions $\MH$ considered here also takes into account the set
$\MV$. For completeness, we provide the
 proof of the lemma in Appendix \ref{app:1}.

\begin{lemma} \label{lem:concave} Given $n\in \N$ and  $\ f_i\in \MH$, $i=1,\cdots,n$, for each concave function $\psi\in \C^2(\R^{n})$ that is monotone increasing in each variable separately, and any $\lambda>0$, we have
\[\begin{array}{l} \ds \int_{\overline G}\psi
  \left(f_1(y),\cdots,f_n(y)\right)d\pi(y) \\ \qquad \geq \ds
  \int_{\overline G}\psi\left(f_1(y)-\lambda \ML
    f_1(y),\cdots,f_n(y)-\lambda \ML f_n (y)\right)d\pi(y). \end{array} \]
\end{lemma}
\bigskip

\begin{lemma} \label{lem:cont}
Let $r_i$ be a bounded non-negative function defined on $\overline G$ for each $i=1,\cdots, n$ and $n\in \N$.  For $z=(z_1,\cdots,z_n)\in \R^n$, let the function $\psi:\
\R^n\rightarrow \R$ be defined by\[\psi(z_1,\cdots,z_n)\doteq \inf_{x\in
  \overline G}\left(\sum_{i=1}^n r_i(x)z_i\right).\] Then $\psi$ is continuous, concave and
monotone increasing in each variable separately.
\end{lemma}
\proof Since $\psi$ is the pointwise infimum of affine functions and $r_i$ is non-negative for each $i=1,\cdots,n$, then it is not hard to see that $\psi$ is concave and is monotone increasing in each variable separately. We next show that $\psi$ is continuous, which completes the proof of the lemma. Fix $z\in \R^n$ and let $\{z^m,\ m\in \N\}$ be a sequence of points in $\R^n$ such that $z^m\rightarrow z$ as $m\rightarrow \infty$. By the definition of $\psi(z)$, for each $\varepsilon>0$, there exists $x\in \overline G$ such that $\sum_{i=1}^n r_i(x)z_i<\psi(z)+\varepsilon$. Since $z^m\rightarrow z$ and $r_i$ is bounded for each $i=1,\cdots,n$, for all $m$ large enough, we have $\sum_{i=1}^n r_i(x)z^m_i<\psi(z)+2\varepsilon$. This shows that $\psi(z^m)<\psi(z)+2\varepsilon$ for all $m$ large enough. Thus, by taking limsup and then $\varepsilon \rightarrow 0$, we have $\limsup_{m\rightarrow \infty} \psi(z^m)\leq \psi(z)$. On the other hand, For each $x\in \overline G$, $\psi(z)\leq \sum_{i=1}^n r_i(x)z_i$ by definition. For each $\varepsilon>0$, $\sup_{x\in \overline G}\sum_{i=1}^n r_i(x)|z^m_i-z_i|<\varepsilon$ for all $m$ large enough since $r_i$ is bounded for each $i=1,\cdots,n$. Then $\psi(z)\leq \sum_{i=1}^n r_i(x)z^m_i+\varepsilon$ for all $x\in \overline G$ and all $m$ large enough. It follows that $\psi(z)\leq \inf_{x\in \overline G}\left(\sum_{i=1}^n r_i(x)z^m_i\right)+\varepsilon$ for all $m$ large enough and hence $\psi(z)\leq \liminf_{m\rightarrow \infty}\psi(z^m)$. As a result, we have $\lim_{m\rightarrow \infty}\psi(z^m)=\psi(z)$ and $\psi$ is continuous at $z\in \R^n$. This shows that $\psi$ is continuous. \endproof
\bigskip

 For each function $v\in \C_b(\overline G \times \overline G)$, define
$\Upsilon_\lambda (v)$ as follows:
\be \label{dis:up}  \Upsilon_\lambda(v) \doteq
\left\{ \begin{array}{c}  \ds (r_1, \cdots, r_n, f_1,\cdots, f_n,\newH):
  n \in \N,  f_i\in \MH,  \\
\ds   r_i  \in \C_b(\overline G), \ \inf_{x\in \overline G} r_i(x)\geq 0,\ i = 1, \ldots, n, \\
\ds \  \newH \in \C(\overline G), \mbox{ and } \forall
x, y\in \overline G, \\ \ds
 \sum_{i=1}^n r_i(x)(f_i(y)-\lambda \ML f_i(y))+\newH(y)\geq v(x,y)\
\end{array}
\right\}.
\ee
Thus, it is easy to see that $\Upsilon_\lambda (v)$ is non-empty for each $v\in \C_b(\overline G\times \overline G)$.

For each $v\in \C_b(\overline G \times \overline G)$, define
\be \label{dis:Q+}  Q_{\lambda+}(v) \doteq
\inf\left\{ \begin{array}{c}  \ds\int_{\overline G}
    \left(\sum_{i=1}^n r_i(x)f_i(x)+\newH(x)\right)d\pi(x): \forall n\in
    \N, \\
 \qquad  (r_1,\cdots,r_n,f_1,\cdots,f_n, \newH)\in
 \Upsilon_\lambda(v)\ \end{array} \right\}, \ee
and
\be  \label{dis:Q-} Q_{\lambda-}(v)\doteq -Q_{\lambda+}(-v).\ee
Then the functionals $Q_{\lambda+}$ and $Q_{\lambda-}$ are real-valued on $\C_b(\overline G\times \overline G)$ and satisfy the
following properties. We use $\bf 0$ and $\bf 1$, respectively,
 to represent the identically zero and the identically one functions
 on $\overline{G} \times \overline{G}$. Then ${\bf 0,\ \bf 1}\in \C_b(\overline G\times \overline G)$.

 \begin{prop} \label{prop:Q}
 For each $v,\ v_1,\ v_2\in \C_b(\overline G \times \overline G)$, $t\geq 0$ and $C, C_1, C_2\in \R$,
 \begin{enumerate}
 \item $Q_{\lambda+}({\bf 0}) \geq 0$;
 \item $Q_{\lambda+}(v_1+v_2)\leq Q_{\lambda+}(v_1)+Q_{\lambda+}(v_2)$ and $Q_{\lambda+}(tv)=tQ_{\lambda+}(v)$;
  \item $Q_{\lambda-}(v) \leq Q_{\lambda+}(v)$;
 \item $ Q_{\lambda-}(C{\bf 1}) = Q_{\lambda+}(C{\bf 1})= C$;
  \item if $v_1\leq v_2$, then $Q_{\lambda+}(v_1) \leq Q_{\lambda+}(v_2)$ and $Q_{\lambda-}(v_1) \leq Q_{\lambda-}(v_2)$;
\item if $C_1{\bf 1}\leq v \leq C_2{\bf 1}$, then $C_1\leq Q_{\lambda-}(v)\leq Q_{\lambda+}(v)\leq C_2$.
 \end{enumerate}
 \end{prop}
\proof We start with the proof of the first property.
For each $n\in \N$, let \[(r_1,\cdots,r_n,f_1,\cdots,f_n,\newH)\in
\Upsilon_\lambda(\bf 0).\]  It follows from Lemma \ref{lem:cont} that the function $\psi:\
\R^n\rightarrow \R$ defined by\[\psi(z_1,\cdots,z_n)\doteq \inf_{x\in
  \overline G}\left(\sum_{i=1}^n r_i(x)z_i\right)\] is continuous, concave and
monotone increasing in each variable separately. For each $\alpha>0$,
let $\phi_{\alpha}:\ \R^{n}\rightarrow \R$ be a function (mollifier) that satisfies $\phi_{\alpha}\in \C^\infty(\R^n)$,
\begin{eqnarray*}
 \phi_{\alpha}\geq 0,\qquad \mbox{supp}(\phi_{\alpha})\subset
 B_{\alpha}(0)\subset \R^{n},\qquad \int_{\R^{n}}\phi_{\alpha}(x)\,dx=1.
  \end{eqnarray*}
Letting $\ast$ stand for the convolution operation, we define
\[
\psi^\alpha \doteq \psi \ast \phi_{\alpha}.
\]
Then $\psi^\alpha\in \C^\infty(\R^n)$ and $\psi^\alpha$
is also concave and  monotone increasing in each variable separately.
It follows from Lemma \ref{lem:concave} that
\[\begin{array}{l} \ds \int_{\overline
    G}\psi^\alpha\left(f_1(y),\cdots,f_n(y)\right)d\pi(y) \\ \qquad
  \geq \ds \int_{\overline G}\psi^\alpha\left(f_1(y)-\lambda \ML
    f_1(y),\cdots,f_n(y)-\lambda \ML
    f_1(y)\right)d\pi(y). \end{array} \]
Since $f_i\in \C^2_c(\overline G)\oplus\R$ for each $i=1,\cdots,n$ and
$\psi^\alpha\rightarrow \psi$ pointwise  as
$\alpha\rightarrow 0$,  by sending $\alpha\rightarrow 0$ and applying
the dominated convergence theorem to both sides of the above
inequality, we have
\be
\label{ql-ineq}
\begin{array}{l} \ds \int_{\overline G}\inf_{x\in \overline
    G}\left(\sum_{i=1}^n r_i(x)f_i(y)\right)d\pi(y) \\ \qquad \geq \ds
  \int_{\overline G}\inf_{x\in \overline G}\left(\sum_{i=1}^n
    r_i(x)(f_i(y)-\lambda \ML f_i(y))\right)d\pi(y).
\end{array}
\ee
 Thus, for every $(r_1,\cdots,r_{n},f_1,\cdots,f_{n},\newH)\in
 \Upsilon_\lambda(\bf 0)$, $n\in \N$, using (\ref{ql-ineq}) and the
 definition of  $\Upsilon_\lambda(\bf 0)$ given in  (\ref{dis:up}),
it follows that
  \[\begin{array}{l} \ds \int_{\overline G}\left(\sum_{i=1}^n
      r_i(y)f_i(y)+\newH(y)\right)d\pi(y) \\ \qquad \geq \ds
    \int_{\overline G}\left(\inf_{x\in \overline G}\left(\sum_{i=1}^n
        r_i(x)f_i(y)\right)+\newH(y)\right)d\pi(y) \\ \qquad \geq \ds
    \int_{\overline G}\left(\inf_{x\in \overline G}\left(\sum_{i=1}^n
        r_i(x)(f_i(y)-\lambda \ML
        f_i(y))\right)+\newH(y)\right)d\pi(y)\\ \qquad  \geq \ds 0. \end{array} \]
Property 1 is then a consequence of the definition of $Q_{\lambda+}$ given in (\ref{dis:Q+}).

The second property follows from the definition of $Q_{\lambda+}$ in
(\ref{dis:Q+}) and the observations
that \[(r_1^j,\cdots,r_{n_j}^j,f_1^j,\cdots,f_{n_j}^j,\newH^j)\in
\Upsilon_\lambda(v_j) \qquad \mbox{ for } j=1,2,\]
imply \[(r_1^1,\cdots,r_{n_1}^1,r_1^2,\cdots,r_{n_2}^2,f_1^1,\cdots,f_{n_1}^1,f_1^2,\cdots,f_{n_2}^2,\newH^1+\newH^2)\in
\Upsilon_\lambda(v_1+v_2)\] and
that \[(r_1,\cdots,r_n,f_1,\cdots,f_n,\newH)\in \Upsilon_\lambda(v)\] if
and only if \[(r_1,\cdots,r_n,tf_1,\cdots,tf_n,t\newH)\in
\Upsilon_\lambda(tv).\] Now, the definition of $Q_{\lambda-}$ in
(\ref{dis:Q-}) and properties 1 and 2 of $Q_{\lambda+}$ just
established above
imply that
 \[Q_{\lambda+}(v)-Q_{\lambda-}(v) = Q_{\lambda+}(v)+Q_{\lambda+}(-v)\geq Q_{\lambda+}({\bf 0})\geq 0,\] which establishes the third property.

Since $(r_1,f_1,\newH)=({\bf 0}, {\bf 0}, C{\bf 1})\in
\Upsilon_\lambda(C\bf 1)$ and $\pi$ is a probability measure, it
follows from the definition (\ref{dis:Q+}) of $Q_{\lambda+}$ that \[
Q_{\lambda+}(C{\bf 1})\leq C.\] Replacing $C$ by $-C$ in the last
inequality and using (\ref{dis:Q-}), this implies
that \[Q_{\lambda-}(C{\bf 1}) = -Q_{\lambda+}(-C{\bf 1})\geq -(-C) =
C.\]  Thus, property 4 follows from property 3 and the last two
assertions. Since $\Upsilon_\lambda(v_2)\subseteq
\Upsilon_\lambda(v_1)$ when $v_1\leq v_2$, property 5 is an immediate
consequence of the definitions of $Q_{\lambda+}$ and $Q_{\lambda-}$
given in (\ref{dis:Q+}) and (\ref{dis:Q-}), respectively. Property 6
is easily deduced  from properties 3--5.
\endproof

\bigskip

Properties 2 and 4 of Proposition \ref{prop:Q} show that
$Q_{\lambda+}$ is a sublinear functional on $\C_b(\overline G \times
\overline G)$ and $C = Q_{\lambda+}(C{\bf 1}) = Q_{\lambda-}(C{\bf
  1})$ for each $C\in \R$. Note that $\C_b(\overline G \times
\overline G)$  is a vector space.
Let $\Lambda_\lambda$ be the linear functional on the space of
constant functions on $\overline G\times \overline G$ defined
by \[\Lambda_\lambda(C{\bf 1})=C.\] By the Hahn-Banach theorem and the
definitions of $Q_{\lambda+}$ and $Q_{\lambda-}$, $\Lambda_\lambda$
can be extended to a linear functional on $\C_b(\overline G\times
\overline G)$ that satisfies  \be \label{QL} Q_{\lambda-}\leq
\Lambda_\lambda \leq Q_{\lambda+}.\ee Together with property 6 of
Proposition \ref{prop:Q}, this implies that  $\Lambda_\lambda({\bf 1})=1$ and $\Lambda_\lambda$ is a
positive linear functional on $\C_b(\overline G\times
\overline G)$ and hence on $\C_0(\overline G\times
\overline G)$. Thus,  by the Riesz
representation theorem and the fact that $\Lambda_\lambda(C{\bf 1})=C$, there exists a probability measure
$\nu_\lambda$ on $\overline G\times \overline G$ with
\be \label{Lam}\Lambda_\lambda(v)=\int_{\overline G\times \overline
  G}v(x,y)\,d\nu_\lambda(x,y) \qquad\mbox{ for each }v\in \C_0(\overline
G\times \overline G)\oplus \R. \ee

The next lemma shows that both marginals of the probability measure $\nu_\lambda$ are equal to $\pi$.

\begin{lemma} For each $w\in \C_c(\overline G)$,
\be \label{dis:marg}\int_{\overline G\times \overline G}w(x)d\nu_\lambda(x,y)=\int_{\overline G\times \overline G}w(y)d\nu_\lambda(x,y)=\int_{\overline G} w(x)d\pi(x). \ee
\end{lemma}
\proof It suffices to prove (\ref{dis:marg}) for non-negative functions $w\in \C_c(\overline G)$ with $\sup_{x\in \overline G}w(x)\leq 1$. Let $\{f_n,\ n\in \N\}$ be a sequence of functions in $\C_c(\overline G)$ such that $0\leq f_n\leq f_{n+1}\leq 1$ for each $n\in \N$ and $f_n(x)\rightarrow 1$ as $n\rightarrow \infty$ for each $x\in \overline G$.  For each $n\in \N$, set \[ u_n(x,y)=f_n(y),\ v_n(x,y)=w(x)f_n(y), \ (x,y)\in \overline G \times \overline G,\] then $v_n,\ 1-v_n\in \C_c(\overline G\times \overline G)$. Set $v(x,y)=w(x), \ (x,y)\in \overline G \times \overline G $.  Since $(r_1,f_1,\newH_1)=(w,{\bf 1},{\bf 0})\in \Upsilon_\lambda(v)$ and  $(r_2,f_2,\newH_2)=(1-w,{\bf 1},{\bf 0})\in \Upsilon_\lambda(1-v)$, by (\ref{QL}), linearity of $\Lambda_\lambda$, property (5) of Proposition \ref{prop:Q} and the definition (\ref{dis:Q+}) of $Q_{\lambda+}$, we have
\be\label{LQ1}\Lambda_\lambda(v_n)\leq Q_{\lambda+}(v_n)\leq Q_{\lambda+}(v) \leq
\int_{\overline G}w(x)\,d\pi(x)\ee and
\begin{eqnarray} \label{LQ2}\Lambda_\lambda(1-v_n) &=&\Lambda_\lambda(1-v) + \Lambda_\lambda(v(1-u_n)) \\ &\leq & Q_{\lambda+}(1-v) +Q_{\lambda+}(1-u_n) \nonumber\\ &\leq &
\int_{\overline G}(1-w(x))\,d\pi(x) + Q_{\lambda+}(1-u_n).\nonumber \end{eqnarray}  The above two inequalities (\ref{LQ1}) and (\ref{LQ2}), together with (\ref{Lam}) and the facts that $\Lambda_\lambda({\bf 1})=1$
and $\pi$ is a probability measure, imply that
\be \label{LQ3} \int_{\overline G}w(x)\,d\pi(x) - Q_{\lambda+}(1-u_n) \leq \int_{\overline G\times \overline G}w(x)f_n(y)d\nu_\lambda(x,y)\leq \int_{\overline G}w(x)\,d\pi(x).\ee
Since $({\bf
  1},{\bf 0},1-f_n)\in \Upsilon_\lambda(1-u_n)$, by property (6) of Proposition \ref{prop:Q}  and the definition (\ref{dis:Q+}) of $Q_{\lambda+}$, we have \[0\leq Q_{\lambda+}(1-u_n) \leq
\int_{\overline G}(1-f_n(x))\,d\pi(x).\]
Applying the dominated convergence theorem to the right-hand side of the above inequality, we have \[\lim_{n\rightarrow \infty} Q_{\lambda+}(1-u_n) =0. \] Then another application of the dominated convergence theorem to the middle term of (\ref{LQ3}) yields that \[\int_{\overline G\times \overline G}w(x)d\nu_\lambda(x,y) = \int_{\overline G}w(x)\,d\pi(x).\]
In an exactly analogous fashion, we can define \[ v(x,y) = w(y),\ v_n(x,y)=f_n(x)w(y), \
(x,y)\in \overline G \times \overline G,\ n\in \N. \] Using the facts that $({\bf
  1},{\bf 0},w)\in \Upsilon_\lambda(v)$ and  $({\bf 1},{\bf 0},1-w)\in
\Upsilon_\lambda(1-v)$, the same argument given above can then be used to establish the second equality in (\ref{dis:marg}). This completes the proof of the lemma.  \endproof
\bigskip

Let $\{q_x^{\lambda,0}(dy),\ x\in \overline G\}$ be a regular
conditional probability of $\nu_{\lambda}$.
For each bounded and
measurable function $g$ and $x\in \overline G$, define
 \be \label{dis:pil}\pi_{\lambda}^0(g)(x)\doteq \int_{\overline G}
 g(y) q_x^{\lambda,0}(dy). \ee
Then for each $r\in \MC(\overline G)$
 with $ \inf_{x\in \overline G} r(x)\geq 0$ and $g\in \C_c(\overline
 G)$, we have\be \label{dis:chain3} \begin{array}{ll}\ds
   \int_{\overline G\times \overline G}r(x)g(y)d\nu_\lambda(x,y)\\
   \qquad \ds =\int_{\overline G\times \overline
     G}r(x)g(y)q_x^{\lambda,0}(dy)d\pi(x)\\ \qquad \ds
   =\int_{\overline G} r(x)\pi_\lambda^0(g)(x) d\pi(x). \end{array}\ee
For every $g \in \C_c (\overline G)$, substituting $w=g$ in (\ref{dis:marg}) and $r=\bf 1$ in
(\ref{dis:chain3}), it follows that
\be
\label{dis-pil2}
   \int \int_{\overline G \times \overline G} g(y) q_x^{\lambda, 0}
(dy) \pi(dx) = \int g(y) \pi (dy).
\ee
This shows that $\pi$ would serve as a stationary distribution for
any Markov chain with  $q_x^{\lambda,0}$ as its transition kernel.
However, in order to be able to show that the limit of the sequence of Markov
chains, as $\lambda \rightarrow 0$,
satisfies the submartingale problem, we will need to establish
an additional inequality.  As we show below, this will require us to
 modify  the definition of $q_x^{\lambda, 0}$ for $x$ on  a set of zero $\pi$-measure.

Let $r\in \C_c(\overline G)$ with $ \inf_{x\in \overline G}
r(x)\geq 0$, $g\in \C_c(\overline G)$, and $-f\in \MH$ such that
$f(y)-\lambda \ML f(y)\leq g(y)$ for all $y\in \overline G$.  Defining
 $v(x,y) = r(x) g(y)$ for $(x,y) \in \overline G \times
\overline G$, we see that $v\in \C_c(\overline G\times \overline G)$ and $(r,-f,{\bf 0}) \in
\Upsilon_{\lambda}(-v)$.
Therefore, using first  (\ref{dis:chain3}) and (\ref{Lam}), next (\ref{QL})
and (\ref{dis:Q-}), and then the  definition of $Q_{\lambda+}$
in (\ref{dis:Q+}), it follows that
\be \label{dis:chain6} \begin{array}{rcl}
\ds \int_{\overline G}r(x)\pi_{\lambda}^0(g)(x)d\pi(x)  =
\Lambda_\lambda(v) \geq  Q_{\lambda-}(v) & = &\ds -Q_{\lambda+}(-v)  \\
& \geq & \ds \int_{\overline G}r(x)f(x)d\pi(x). \end{array}\ee

We now show that the function $r$ in (\ref{dis:chain6}) can be replaced by the indicator function of any bounded Borel measurable set.

\begin{lemma} \label{lem:pif}
Suppose we are given $g\in \C_c(\overline G)$ and $-f\in \MH$ such that
$f(y)-\lambda \ML f(y)\leq g(y)$ for all $y\in \overline G$. Then for each bounded Borel measurable set $A$, \be \label{dis:chain61}  \int_{\overline G}\ind_A(x)\pi_{\lambda}^0(g)(x)d\pi(x) \geq \int_{\overline G}\ind_A(x)f(x)d\pi(x). \ee
\end{lemma}
\proof  Fix a bounded Borel measurable set $A$. For the space $\overline G$ and the finite measure $\pi$, it follows from Lusin's theorem that  there exists a sequence $\{r_n,\ n\in \N\}$ of continuous functions with compact support such that $\sup_{x\in \overline G}|r_n(x)|\leq 1$ for all $n\in \N$ and $\ind_A(x)=\lim_{n\rightarrow \infty} r_n(x)$ almost everywhere with respect to $\pi$. Since $\ind_A$ is non-negative, we have $\ind_A(x)=\lim_{n\rightarrow \infty} (r_n(x)\vee 0)$ almost everywhere with respect to $\pi$, where $r_n(\cdot)\vee 0$ is continuous and non-negative for each $n\in \N$. By (\ref{dis:chain6}), we have for each $n\in \N$,\[
\int_{\overline G}(r_n(x)\vee 0)\pi_{\lambda}^0(g)(x)d\pi(x) \geq \int_{\overline G}(r_n(x)\vee 0)f(x)d\pi(x). \] Thus, (\ref{dis:chain61}) follows from (\ref{dis:chain6}) and an application of the dominated convergence theorem. \endproof
\bigskip

For each $g\in \C_c(\overline G)$ and $y\in \overline G$, let
\be \kappa_\lambda(g)(y)\doteq \sup\left\{
\ba{c}
f(y): -f\in \MH, \mbox{ and for all }z\in
\overline G, \\
f(z)-\lambda \ML f(z)\leq g(z)
\ea
\right\}.
\ee

\begin{lemma} \label{lem:control}
There exists a Borel measurable set $U$ with $\pi(U) = 0$ such
 that for every $g \in \C_c (\overline G)$,
\be
\label{ineq-pi-kappa}
\pi_\lambda^0(g)(y)\geq
 \kappa_\lambda(g)(y),  \qquad   y\in \overline
 G\setminus U.
\ee
\end{lemma}
\proof
To establish the lemma it suffices to show that
for each $g\in \C_c(\overline G)$,
\be
\label{ineq-kappa1}
\int_{\overline G}\ind_A(y)\pi_{\lambda}^0(g)(y)d\pi(y)\geq
\int_{\overline G}\ind_A(y)\kappa_{\lambda}(g)(y)d\pi(y)
\ee
for every measurable subset $A$ of $\overline
G$.  Indeed, then fix a  countable dense subset
$\{g_n, n \in \N\}$ in $\C_c(\overline
 G)$.  For each $n \in \N$, the inequality (\ref{ineq-kappa1})  applied with $g = g_n$
 implies that there exists a  Borel measurable set $U_{g_n}$ such that
 $\pi(U_{g_n})=0$ and
$\pi_\lambda^0(g_n)(y)\geq \kappa_\lambda(g_n)(y)$
for all $ y\in \overline G\setminus U_{g_n}$.
If we set $U=\cup_{n\in \N}U_{g_n}$,  then  $\pi(U)=0$ and
\be
\label{key-inequality}
\pi_\lambda^0(g_n)(y)\geq \kappa_\lambda(g_n)(y), \qquad n \in \N, y
\in \overline G \setminus U.
\ee
We  now extend this inequality to all $g \in\C_c(\overline G)$ by a standard approximation argument.
Fix $g\in \C_c(\overline G)$ and $y \in \overline G\setminus U$.
Since $g\in \C_c(\overline G)$ and $\{g_n,n\in \N\}$ is a countable
dense set of $\C_c(\overline G)$, there exists a subsequence
$\{g_{n_k},k\in \N\}$ of $\{g_n,n\in \N\}$ such that
$g_{n_k}\rightarrow g$ uniformly as $k\rightarrow \infty$. Let
$\epsilon_k=\sup_{z\in \overline G}|g_{n_k}(z)-g(z)|$ for each $k\in
\N$. For each $-f\in \MH$ and
$f(z)-\lambda \ML f(z)\leq g(z)$  for all $z\in \overline G$, and for
each $k\in \N$, let $f_k=f-\epsilon_k$. Then $-f_k\in \MH$ and for each $z\in \overline G$, \[f_k(z)-\lambda
\ML f_k(z) = f(z)-\lambda \ML f(z) -\epsilon_k \leq g(z)
-\epsilon_k\leq g_{n_k}(z). \] Thus by (\ref{key-inequality}) we have
for each $k\in \N$,\[\pi_\lambda^0(g_{n_k})(y) \geq \kappa_\lambda
(g_{n_k})(y)\geq f_k(y) = f(y)-\epsilon_k.\]
 By taking the limit as $k\rightarrow \infty$ on both sides of the
 above inequality, we obtain (\ref{ineq-pi-kappa}).

It remains to show (\ref{ineq-kappa1}).
Fix $g\in \C_c(\overline G)$, and let $m_g$ denote the supremum of
$g$ on $\overline G$.  From (\ref{dis:pil}), it is clear that  $\pi_\lambda^0 (g)(y)\leq m_g$ for each $y\in \overline G$.
Suppose that (\ref{ineq-kappa1}) does not hold. Then there exist $\varepsilon>0$ and a bounded Borel set $A\subseteq \overline G$ such that $\pi(A)>0$ and for each $y\in A$, \[\pi_{\lambda}^0(g)(y)\leq \kappa_{\lambda}(g)(y)-\varepsilon.\]
Let $M_{g,A}$ be the essential supremum of $\kappa_\lambda(g)$ on $A$
under $\pi$. Then it follows from the definition of the essential
supremum that \[ \pi(A\cap
\{y \in \overline G:\kappa_\lambda(g)(y)>M_{g,A}-\varepsilon/4\})>0.\] Notice that the definition of  $\kappa_\lambda(g)$ implies that $\kappa_\lambda(g)$ is lower semicontinuous, and then  the set $\{y \in \overline G:\kappa_\lambda(g)(y)>M_{g,A}-\varepsilon/4\}$ is an open set. Since any open set in $\R^J$ is the union of an increasing sequence of closed sets and $A$ is bounded, then there exists a closed set $D$ such that $D$ is a subset of $\{y \in \overline G:\kappa_\lambda(g)(y)>M_{g,A}-\varepsilon/4\}$ and $\pi(A\cap D) >0$. We claim that there exists  $y\in D$ such that  $\pi(A \cap
B_r(y))>0$ for each $r>0$. Indeed, suppose  the claim does not hold.
In other words,  suppose that for every $y\in D$, there exists $r_{y}>0$ such that $\pi(A \cap
B_{r_y}(y))=0$.  Then, since $D$ is compact, by applying a finite
cover argument, it follows that $\pi(A\cap D)=0$, which is a
contradiction.  Hence, the claim holds. Now, choose  $y\in D$ that
satisfies the properties of the claim, and note that then
$\kappa_\lambda(g)(y)>M_{g,A}-\varepsilon/4$.  By the definition of
$\kappa_\lambda(g)(y)$, there exists  a function $f\in
\C_c^2(\overline G)$ such that $-f\in \MH$, $f(z)-\lambda \ML f(z)\leq
g(z)$  for all $z\in \overline G$ and $f(y)>
\kappa_\lambda(g)(y)-\varepsilon/4$.  Hence, it follows that $f(y)>M_{g,A}-\varepsilon/2$. By the continuity of $f$, there exists $r_1>0$ such that $f(z)\geq M_{g,A}-3\varepsilon/4$ for all $z\in A\cap B_{r_1}(y)$. As a consequence, $\pi$-almost surely on $A\cap B_{r_1}(y)$, \[\pi_\lambda^0(g)(z)\leq \kappa_{\lambda}(g)(z)-\varepsilon \leq M_{g,A}-\varepsilon \leq f(z) - \varepsilon/4. \] Thus, we have\[\int_{\overline G}\ind_{A\cap B_{r_1}(y)}(z)\pi_{\lambda}^0(g)(z)d\pi(z)\leq \int_{\overline G}\ind_{A\cap B_{r_1}(y)}(z)(f(z)-\varepsilon/4)d\pi(z).\]
On the other hand, it follows from Lemma \ref{lem:pif} that
\[\int_{\overline G}\ind_{A\cap
  B_{r_1}(y)}(z)\pi_{\lambda}^0(g)(z)d\pi(z)\geq \int_{\overline
  G}\ind_{A\cap B_{r_1}(y)}(z)f(z)d\pi(z).\]
But this contradicts the previous inequality since $\pi(A \cap
B_{r_1}(y))>0$.
This establishes (\ref{ineq-kappa1}) and the lemma follows.
\endproof
\bigskip

Let the measurable set  $U$ be as in Lemma \ref{lem:control},  and define
\be \label{def1} q_y^\lambda(dz)\doteq
 q_y^{\lambda,0}(dz) \qquad \mbox{ and }\qquad \pi_\lambda(g)(y)\doteq
 \pi_\lambda^0(g)(y), \qquad y\in \overline G\setminus U.\ee
We shall now extend the definition of $q_y^\lambda$ to
 $y$ for all $y\in\overline G$ in such a way
that the analogous inequality
 (\ref{ineq-pi-kappa}) holds
on all of $\overline G$.   Since the
 submartingale problem is well posed, let $\{\Q_z,\ z\in \overline
 G\}$ be the unique solution to the submartingale problem. For each
 $y\in U$, $t\geq 0$ and Borel set $A$, let
 $p(t,y,A)=\Q_y(\omega(t)\in A)$; thus, $p(t,y,\cdot)$ is the
 probability distribution of $\omega(t)$ under $\Q_y$. For $y\in U$
 and Borel set $A$, define \be q_y^\lambda(A)\doteq \frac{1}{\lambda}
 \int_0^\infty e^{-t/\lambda} p(t,y,A)dt \ee and for each  $g\in
 \C_b(\R^J)$, define
\be \label{def3}
\pi_\lambda(g)(y)\doteq \int_{\overline G}g(z) q_y^\lambda(dz)
=\frac{1}{\lambda}\int_{\overline
  G}e^{-t/\lambda}\E^{\Q_y}[g(\omega(t))]\,dt, \qquad y\in U,
\ee
where  the second equality above follows from an application of Fubini's theorem.

\begin{lemma} \label{lem:KI}
For each $g\in \C_c(\overline G)$ and $y\in \overline G$, $\pi_\lambda(g)(y)\geq \kappa_\lambda(g)(y)$.
\end{lemma}
\proof
For each $g\in \C_c(\overline G)$, let $-f\in \MH$ be such that $f(z)-\lambda \ML f(z)\leq g(z)$ for all $z\in \overline G$. It follows from property 2 of Definition \ref{def-smg} that \[f(\omega(t))-\int_0^t\ML
f(\omega(u))\,du \] is a $\Q_y$-submartingale on
$(\ccspace,\MM,\{\MM_t\})$, which implies that
 \[
\E^{\Q_y}\left[f(\omega(t))-\int_0^t\ML
f(\omega(u))\,du\right] \geq \E^{\Q_y}[f(\omega(0))] = f(y), \]
and hence, that  \be
\label{ineq-submg}
\frac{1}{\lambda}\int_0^\infty  e^{-t/\lambda}\E^{\Q_y}\left[f(\omega(t))-\int_0^t\ML
f(\omega(u))\,du\right]dt \geq  f(y). \ee
An application of Fubini's theorem shows that
\begin{eqnarray*}
&& \frac{1}{\lambda}\int_0^\infty  e^{-t/\lambda}\E^{\Q_y} \left[
  \int_0^t \ML f(\omega(u)) \, du \right] \, dt \\
&& \qquad  =
\frac{1}{\lambda} \int_0^\infty \E^{\Q_y} \left[ \ML f(\omega(u))
\right] \, \left(\int_u^\infty e^{-t/\lambda} \, dt \right) \, du \\
&& \qquad  =   \int_0^\infty e^{-u/\lambda} \,
 \E^{\Q_y} \left[ \ML f(\omega(u))
\right] du.
\end{eqnarray*}
When combined with (\ref{def3}) and (\ref{ineq-submg})  this implies that
\[  \pi_{\lambda} (f- \lambda \ML f) (y) \geq f(y),   \qquad y
\in U.
\]
Taking the supremum over all $f$ that satisfy $-f\in \MH$ and $f(z)-\lambda \ML f(z)\leq g(z)$ for
all $z\in \overline G$,  we conclude that $\pi_\lambda(g)(y)\geq
\kappa_\lambda(g)(y)$ for $y \in \overline G \setminus U$.
Together with (\ref{def1}) and  Lemma \ref{lem:control}, this completes the proof.
\endproof
\bigskip

We are now in a position to construct the desired Markov chain. Let
$\{X^\lambda(n\lambda),n \in \Z_+\}$ be a discrete time Markov chain
defined on some probability space $(\Omega, \MF, \P)$ with transition
kernel $q_x^{\lambda}(dy)$. Also, let \[ \MF^\lambda_n\doteq
\sigma(X^\lambda(j\lambda), j=0,1,\cdots, n),\ n\in \N. \]
The next result shows that this Markov chain has the  desired properties.

\begin{proposition}\label{prop:sub}
For each $\lambda>0$, $\pi$ is a stationary distribution of the Markov
chain $\{X^\lambda(n\lambda),n \in \Z_+\}$.  Moreover, for each $y\in
\overline G$,  $g\in \C_c(\overline G)$ and $f\in \C_c^2(\overline G)$
such that $-f\in \MH$ and $g=f-\lambda\ML f$, \be \label{dis:submart}
g(X^\lambda(n\lambda))-g(X^\lambda(0))-\lambda\sum_{j=0}^{n-1}\ML
f(X^\lambda(j\lambda)),  \quad n \in \N, \ee   and \be \label{dis:submart4}
f(X^\lambda(n\lambda))-f(X^\lambda(0))-\lambda\sum_{j=1}^{n}\ML
f(X^\lambda(j\lambda)), \quad n \in \N, \ee   are $\P$-submartingales
with respect to $\{\MF^\lambda_n\}$.
\end{proposition}
\proof For $g\in \C_c(\overline G)$,
substituting $w=g$ in (\ref{dis:marg}), $r=\bf 1$ in
(\ref{dis:chain3}), and using the relations
$q_y^\lambda(dz)=q_y^{\lambda,0}(dz)$ for $y\in \overline G\setminus
U$ and $\pi(U)=0$, it follows that
\[\E^\pi[g(X^\lambda(\lambda))]=\E^\pi [g(X^\lambda(0)].\]
This shows that $\pi$ is a stationary distribution of the Markov chain $\{X^\lambda(n\lambda),n \in \Z_+\}$.

To prove the second part of the lemma, fix
 $g\in\C_c(\overline G)$ and $f\in \C_c^2(\overline G)$ such that $-f\in
\MH$ and $g=f-\lambda\ML f$, and let $I$ denote the identity map
$I(g)=g$. Then, by the definition of $\pi_\lambda$ given in
(\ref{def1}), (\ref{dis:pil}) and (\ref{def3}),

\be \label{dis:mart}
g(X^\lambda(n\lambda))-g(X^\lambda(0))-\lambda\sum_{j=0}^{n-1}\frac{1}{\lambda}(\pi_\lambda-I)(g)(X^\lambda(j\lambda))
\ee is a $\P$-martingale with respect to $\{\MF^\lambda_n\}$.  By
Lemma \ref{lem:KI},  for each $z\in \overline G$,
$\pi_\lambda(g)(z)\geq f(z)$, and hence $(\pi_\lambda - I) g(z)\geq
\lambda \ML f(z)$. This establishes the submartingale property for
the process in (\ref{dis:submart}).  In turn, since $f
= g + \ML f$, this immediately implies the  submartingale property for the process in
(\ref{dis:submart4}).
\endproof

\subsection{Localization and Conditioning}
\label{subs:loc}

Let $\{X^\lambda(n\lambda),n \in \Z_+\}$ be the Markov chain with
transition kernel $\{q_x^{\lambda}(dy),\ x\in \overline G\}$ and
initial distribution $\pi$ as constructed in the last section, and let
$X^\lambda$ denote the continuous time extension obtained by linearly
interpolating
$\{X^\lambda(n\lambda),n \in \Z_+\}$ between time points $n\lambda$ and $(n+1)\lambda$ for $n\in \Z_+$. Let \[ \MF^\lambda_t\doteq \sigma(X^\lambda(n\lambda),n\leq \lceil t \rceil),\ t\geq 0.\] Note that, by construction, $X^\lambda$ has continuous paths.  Let $\{\lambda_m, m\in \N\}$ be a sequence of positive decreasing real numbers such that $\lambda_m\rightarrow 0$ as $m\rightarrow \infty$.
For each $m \in \N$, let $\Q^{m}$ denote the probability measure on  $(\ccspace,\MM)$
induced by $X^{\lambda_m}(\cdot)$ when $X^{\lambda_m}(0)$ has distribution $\pi$.
By (\ref{Qpi}), $\Q_{\pi}$ is the
integral of $\Q_y$, the solution to the submartingale problem for a given initial condition $y$, with
respect to the probability measure $\pi$.
  It will prove convenient to represent $\Q^{m}$ in a similar fashion.
For each  $m \in \N$ and $\omega \in \ccspace$, let
$\Q^m_{\omega'}$ be a regular conditional probability
 distribution of $\Q^m$ given $\MM_0$.
 Then, for each $\omega' \in \ccspace$,
\be \label{initial}\Q^{m}_{\omega'}(\omega(0)=\omega'(0))=1.\ee
Moreover,  disintegrating $\Q^m$ and using the fact that the distribution of
$\omega(0)$ under $\Q^m$ is $\pi$, we obtain
\be
\label{rep-qm}
\Q^m (\cdot) = \int_{\ccspace}  \Q^m_{\omega'} (\cdot) \Q^m(d\omega') = \int_{\ccspace}
\Q^m_{\omega'} (\cdot) \P^\pi (d\omega'),
\ee
where  $\P^{\pi}$ is the probability measure on
$(\ccspace, \MM_0)$ obtained as the restriction of $\Q^m$ to $\MM_0$
defined as follows: for every $A_0 \in {\cal
  B}(\R^J)$,
\be
\label{def-ppi}
\P^{\pi} (A) \doteq \pi(A_0 \cap \overline G),  \qquad \mbox{ if } A = \{\omega \in
\ccspace: \omega(0) \in A_0\}.
\ee

Since $G$ may be unbounded, we now carry out a localization. For each
 $N\in \N$ and $\omega\in \ccspace$, let
\[
\kappa^{N,\lambda_m}(\omega)=\inf\left\{t\geq 0: \ba{c}
\ t=j\lambda_m \mbox{ for some } j\in \Z_+\mbox{ and } \\
\omega(t) \notin B_N(0)
\ea
\right\}.
\]

To ease the notation, for each $N,m\in \N$, we use $\Q^{N,m}$ to
denote the probability measure on  $(\ccspace,\MM)$ induced by
$\omega(\cdot\wedge \kappa^{N,\lambda_m})$.   Also,
for each $m, N \in \N$, let $\Q^{N,m}_{\omega'}$ be the law of
$\omega(\cdot\wedge \kappa^{N,\lambda_m})$ under $\Q^m_{\omega'}$.
Then $\{\Q^{N,m}_{\omega'}\}$ is a regular conditional probability
distribution of $\Q^{N,m}$ given $\MM_{0}$.
Now,  for each $N \geq 0$ recall the stopping time $\chi^N$ defined in
(\ref{chi}). It is an immediate consequence of the definitions that
for each $m, N\in \N$, $\chi^N\leq \kappa^{N,\lambda_m}$ and hence,
\be
\label{eq-law}
\Q_{\omega'}^{N,m} (A) = \Q_{\omega'}^m (A),   \qquad A \in
\MM_{\chi^N}.
\ee
 Moreover, analogous to (\ref{rep-qm}), $\Q^{N,m}_{\omega'}$ satisfies the integral
 representation
\be
\label{rep-qnm}
\Q^{N,m} (\cdot) = \int_{\ccspace}  \Q^{N,m}_{\omega'} (\cdot) \Q^{N,m}(d\omega') = \int_{\ccspace}
\Q^{N,m}_{\omega'} (\cdot) \P^\pi (d\omega').
\ee

We now state two main results and show that they imply  Theorem
\ref{thm:SS}.   The first is a compactness result, which is proved in Section
\ref{subs:pftight}.

\begin{proposition}
\label{lem:ctightxnm}
Suppose that Assumption \ref{ass:TF} holds. Then for each $N\in \N$,
 for each $\omega'\in \ccspace$, the sequence of
 probability measures $\{\Q^{N,m}_{\omega'},m\in \N\}$ on
  the Polish space $(\ccspace, \MM)$, equipped with the uniform
  topology on compact sets,
  is precompact.
\end{proposition}

The second result is a characterization of limit points of the
sequence, and is proved in Section \ref{subs-propconv}.

\begin{prop}
\label{prop:conv}
Suppose Assumptions  \ref{ass:V} holds.  Then
there exists $F_0 \in \MM_0$ with $\P^{\pi}
(F_0) = 0$  such
that for every $N \in \N$ large enough such that $\MV\subset B_N(0)$ and each $\omega'\notin F_0$, any limit point
$\Q^{N,*}_{\omega'}$ of the sequence
$\{\Q^{N,m}_{\omega'}, m \in \N\}$ satisfies
\be
\label{restrict}
\Q^{N,*}_{\omega'} (A) = \Q_{\omega'(0)} (A),  \qquad  A \in
\MM_{\chi^N},
\ee
where for $z \in \overline G$, $\Q_{z}$ is the unique solution to the submartingale problem
with initial condition $z$.
\end{prop}

\begin{proof}[Proof of Theorem \ref{thm:SS}]
To prove Theorem \ref{thm:SS}, it suffices to show that the inequality (\ref{mono0}) is sufficient for
$\pi$ to be a stationary distribution for the unique solution to the
associated submartingale problem. For each $N \in \N$,  by Lemma 11.1.2 of \cite{StrVar06} the stopping
time $\chi^N$ is lower semicontinuous and, for every
 $\omega \in \ccspace$,
$\{ \chi^N(\omega), N \in \N\}$ is a nondecreasing sequence that
increases to $\infty$.
Due to  (\ref{eq-law}), Proposition \ref{lem:ctightxnm} and  Proposition \ref{prop:conv},
Lemma 11.1.1 of \cite{StrVar06}  (applied to $\Q^{N,m}_{\omega'}, \Q^m_{\omega'(0)}$ and
$\chi^N$, respectively, in place of $P^{n,k}, P^n$ and $\tau^k$ in \cite{StrVar06})
shows that for $\P^{\pi}$ almost every
$\omega'$, as $m \ra \infty$,  $\Q^{m}_{\omega'} \Rightarrow
\Q_{\omega'(0)}$.   By the  Portmanteau theorem, this implies that for $\P^{\pi}$ almost every
$\omega'$ and every open set $A \in \MM$,
\[ \liminf_{m \ra \infty}  \Q^{m}_{\omega'} (A) \geq \Q_{\omega'(0)}
(A).
\]
When combined with the representations (\ref{Qpi}) and (\ref{rep-qm})
for $\Q_{\pi}$ and $\Q^{m}$, respectively, and
 Fatou's lemma, this implies that  for every open set $A \in \MM$,
\begin{eqnarray*}
\liminf_{m \ra \infty} \Q^{m} (A) & = &  \liminf_{m \ra \infty}
\int_{\ccspace} \Q^{m}_{\omega'} (A) \P^{\pi} (d\omega') \\ & \geq &
  \int_{\ccspace} \Q_{\omega'(0)} (A) \P^{\pi}
(d\omega')  \\
&= & \int_{\overline{G}} \Q_{z} (A) \pi(dz) =    \Q_{\pi} (A).
\end{eqnarray*}
Another application of the Portmanteau theorem then shows that, as $m
\ra \infty$, $\Q^{m} \Rightarrow \Q_{\pi}$.

Fix $t > 0$ and $f \in {\cal C}_b^1 (\overline G)$.
To show that $\pi$ is a stationary distribution of the unique solution
to the submartingale problem, it suffices to show that
\be
\label{toshow-stat}
 \E^{\Q_{\pi}} \left[ f(\omega(t)) \right] = \int_{\overline G}
f(z) \pi(dz).
\ee
Since $\pi$ is a stationary distribution for each Markov chain, it
follows that
\[  \lim_{m \ra \infty}  \E^{\Q^m} \left[ f(w( \lfloor t/\lambda_m
  \rfloor \lambda_m ))\right] = \int_{\overline G} f(z) \pi(dz),
\]
and the convergence $\Q^m \Rightarrow \Q_{\pi}$ established
above implies
\[  \lim_{m \ra \infty} \left| \E^{\Q_\pi}\left[  f(w(t)) \right] -  \E^{\Q^m} \left[ f(w( t))\right] \right| = 0.
\]
 Thus, to establish (\ref{toshow-stat}),   it suffices to show that for every
$f \in {\cal C}_b^1 (\overline G)$,
\be
\label{toshow-stat2}  \lim_{m\ra \infty} \E^{\Q^m} \left[ \left| f( \omega(t)) - f(w( \lfloor t/\lambda_m
  \rfloor \lambda_m ))\right|\right] = 0.
\ee
Now, for each $m \in \N$, $\delta>0$ and $\rho > 0$, define
\[ K^m_{\rho}  \doteq  \left\{ \omega: \sup_{s: |t-s| \leq \lambda_m}
  |\omega(t) - \omega(s)| < \rho \right\}
\] and \[ K_{\rho,\delta}  \doteq  \left\{ \omega: \sup_{s: |t-s| \leq \delta}
  |\omega(t) - \omega(s)| < \rho \right\}.\]
Then $K^{m}_{\rho},\ K_{\rho,\delta}$ are two open sets in $\MM$ and $K_{\rho,\delta} \subset K^{m}_{\rho}$ if $\lambda_m<\delta$.  On the set $K^m_{\rho}$,
\[  \left| f( \omega(t)) - f(w( \lfloor t/\lambda_m
  \rfloor \lambda_m ))\right|\leq ||f^\prime||_{\infty} \rho.
\] On the other hand, the convergence $\Q^m \Rightarrow \Q_{\pi}$ and an application of the Portmanteau theorem show that \[\limsup_{m\rightarrow \infty}\Q^m(K_{\rho,\delta}^c) \leq \Q_{\pi}(K_{\rho,\delta}^c).\] Since  $\Q_{\pi}$ is a probability measure on $\ccspace$, then $\lim_{\delta\rightarrow 0}\Q_{\pi}(K_{\rho,\delta}^c) = 0.$
Putting these all together, we see that
\begin{eqnarray*}
 \lim_{m\ra \infty} \E^{\Q^m} \left[ \left| f( \omega(t)) - f(w( \lfloor t/\lambda_m
  \rfloor \lambda_m ))\right|\right] & \leq & \lim_{m \ra \infty}
2||f||_{\infty} \Q^m
((K^m_{\rho})^c)  + ||f^\prime||_{\infty} \rho  \\
& \leq & 2||f||_{\infty} \Q^m
(K_{\rho,\delta}^c)  + ||f^\prime||_{\infty} \rho.
\end{eqnarray*}
Sending $\delta \downarrow 0$, then $\rho \downarrow 0$ we obtain (\ref{toshow-stat2}) and, hence,
(\ref{toshow-stat}), which
proves that $\pi$ is a stationary distribution for the submartingale
problem.
\end{proof}

\subsection{Precompactness}
\label{subs:tight}

The proof of Proposition \ref{lem:ctightxnm} is given in Section
\ref{subs:pftight}.  It makes
use of a general sufficient condition for  the precompactness of
a sequence of probability measures, which is first stated in
Section \ref{subs:suffcond}.

\subsubsection{A Sufficient Condition for Precompactness}
\label{subs:suffcond}

 For $\rho>0$, let $\tau^m_0(\rho)\doteq 0$ and for $m\in \N$, let
\be
\label{def-taum}
\tau^m_n(\rho)\doteq \inf\left\{
\ba{c}
t\geq \tau^m_{n-1}(\rho):\ t=j\lambda_m
\mbox{ for some } j\in \Z_+ \mbox{ and } \\
|\omega(t)-\omega(\tau^m_{n-1}(\rho))|\geq \rho/4
\ea
\right\}.
\ee

\begin{lemma}\label{prop:ctight}
A sequence of probability measures  $\{\P^m,m\in \N\}$ is precompact in the Prohorov topology on the Polish space $(\ccspace,\MM)$ if the following conditions hold:
\begin{enumerate}
\item for any $\rho>0$, $\delta=j\lambda_m$ for some $j\in \N$, and
  $n\in \Z_+$, \[\P^{m}(\tau^m_{n+1}(\rho)-\tau^m_n(\rho)\leq
  \delta|\MM_{\tau^m_n(\rho)})\leq \delta A_{\rho/4} \]
$\P^m$ almost surely on $\{\tau^m_n(\rho)<\infty\}$,  where $A_{\rho/4}$ is a
constant depending only on $\rho$;
\item for each $T>0$ and $\varepsilon>0$, \[\lim_{m\rightarrow \infty}\sum_{0\leq j\lambda_m\leq T}\P^{m}(|\omega((j+1)\lambda_m)-\omega(j\lambda_m)|\geq \varepsilon)=0;\]
\item \[\lim_{l\rightarrow \infty}\sup_{m\in \N}\P^{m}(|\omega(0)|\geq l)=0.\]
\end{enumerate}
\end{lemma}
\proof The lemma follows directly from Theorem 1.4.11 of
\cite{StrVar06}. Note that Theorem 1.4.11 of \cite{StrVar06} contains
the additional assumption that the hypotheses 1.4.8 and 1.4.9 (with
$h=h_n$) and the choice of the constants $A_f$ can be made independent
of $n$. However, this assumption is used only to prove Lemma 1.4.10
of \cite{StrVar06}.
Here, we put the result of Lemma 1.4.10 therein as one of the conditions of the lemma. \endproof

\subsubsection{Proof of Proposition \ref{lem:ctightxnm}} \label{subs:pftight}

 Fix $n\in \N$. For each $x\in \overline G$ and $\varepsilon\in (0,\infty)$, let
 $f_{x,\varepsilon}$ be the function from the family defined in
 Assumption \ref{ass:TF}.
The proof consists of three main claims. \\
{\em Claim 1.}  For each $\omega'\in \C[0,\infty)$, $x\in \overline G$ and $\varepsilon\in (0,\infty)$,
\be \label{q2}
f_{x,\varepsilon}(\omega(\left\lceil
  t/\lambda_m\right\rceil\lambda_m))-f_{x,\varepsilon}(\omega(0))-\lambda_m\sum_{j=1}^{\left\lceil
    t/\lambda_m\right\rceil \wedge
  (\kappa^{N,\lambda_m}/\lambda_m)}\ML
f_{x,\varepsilon}(\omega(j\lambda_m)) \ee
is an integrable $\Q^{N,m}_{\omega'}$-supermartingale. \\
{\em Proof of Claim 1.}  Fix $\omega'\in \C[0,\infty)$, $x\in \overline G$ and $\varepsilon\in (0,\infty)$. By the construction of
 $\{X^{\lambda_m}(i\lambda_m),\ i\in \Z_+\}$, under $\Q^{m}$,
 $\{\omega(i\lambda_m),\ i\in \Z_+\}$ is a Markov chain. Thus,  $\{\omega(i\lambda_m),\ i\in \Z_+\}$ is still a Markov chain under $\Q^{m}_{\omega'}$. It follows from (\ref{dis:submart4})
 that for each $f\in \MH\cap \C^2_c(\overline G)$,\be \label{mart5}
 f(\omega(\left\lceil
   t/\lambda_m\right\rceil\lambda_m))-f(\omega(0))-\sum_{j=1}^{\left\lceil
     t/\lambda_m\right\rceil}\ML f(\omega(j\lambda_m))\lambda_m
 \ee is a $\Q^m_{\omega'}$-supermartingale. By property (1) of Assumption \ref{ass:TF}, $f_{x,\varepsilon}$ is a finite or countable sum of functions in $\MH\cap \C^2_c(\overline G)$. Since $f_{x,\varepsilon} \in \C^2_b(\overline G)$, then (\ref{q2}) follows from (\ref{mart5}), an application of $L^1$ convergence theorem for a sequence of supermartingales and an application of optional stopping theorem.

We now show that the sequence
$\{\Q^{N,m}_{\omega'},m\in \N\}$ satisfies all three conditions stated
in Lemma \ref{prop:ctight}, and hence is precompact.  \\
\noi
{\em Claim 2.}   $\{\Q^{N,m}_{\omega'},m\in \N\}$ satisfies condition 1 of Lemma
\ref{prop:ctight}. \\
{\em Proof of Claim 2.}
Fix $\rho>0$ and $\delta>0$ of the form $\delta=k\lambda_m$ for some
$k\in \N$, and $m\in \Z_+$.
For $n, m \in \N$, we recall the definition of the stopping time
$\tau_n^m(\rho)$ given in (\ref{def-taum}), and
 let $\{\Q^{N,m}_{\omega^*,n}\}$ and $\{\Q^{m}_{\omega^*,n}\}$, respectively, denote regular conditional
 probability distributions of $\Q^{N,m}_{\omega'}$ and $\Q^{m}_{\omega'}$ given
 $\MM_{\tau^m_n(\rho)}$. It is easy to see that \[
\Q_{\omega*,n}^{N,m} (A) = \Q_{\omega*,n}^m (A),   \qquad A \in
\MM_{\tau^m_n(\rho)}.
\] Notice that $\{\omega(i\lambda_m),\ i\geq \tau^m_n(\rho)(\omega^*)/\lambda_m\}$ is a Markov chain under $\{\Q^{m}_{\omega^*,n}\}$. The same argument in proving {\em Claim 1}, together with the optional stopping theorem, shows that for each $x\in \overline G$ and $\varepsilon\in (0,\infty)$,
 \be \label{q10}\Q^{N,m}_{\omega^*,n}(\tau^m_n(\rho)(\omega)=\tau^m_n(\rho)(\omega^*),\
 \omega(t)=\omega^*(t),\ 0\leq t\leq \tau^m_n(\rho)(\omega))=1,\ee
 and   \begin{eqnarray} \label{q20} & &
   f_{x,\varepsilon}(\omega(\left\lceil (\tau^m_n(\rho)+
     t)/\lambda_m\right\rceil\lambda_m))-f_{x,\varepsilon}(\omega(\tau^m_n(\rho)))\\
   & & \qquad \qquad -\sum_{j=(\left\lceil
       (\tau^m_n(\rho))/\lambda_m\right\rceil+1)\wedge
     (\kappa^{N,\lambda_m}/\lambda_m)}^{\left\lceil
       (\tau^m_n(\rho)+t)/\lambda_m\right\rceil \wedge
     (\kappa^{N,\lambda_m}/\lambda_m)}\ML
   f_{x,\varepsilon}(\omega(j\lambda_m))\lambda_m
   \nonumber\end{eqnarray}
is a
 $\Q^{N,m}_{\omega^*,n}$-supermartingale. Let
$x^*=\omega^*(\tau^m_n(\rho)(\omega^*))$. If $x^*\in \overline
G\setminus B_N(0)$, under $\Q^{N,m}_{\omega^*,n}$,  by the definition
of  $\kappa^{N,\lambda_m}$,   $\omega(\left\lceil (\tau^m_n(\rho)+
  t)/\lambda_m\right\rceil\lambda_m)$ is identically equal to $x^*$
for all $t\geq 0$. Therefore, $\tau_{n+1}^m(\rho)=\infty$ and
$\Q^{N,m}_{\omega^*,n}(\tau^m(\rho)\leq \delta) = 0$, where
\[\tau^m(\rho)=\inf\left\{
\ba{l} \ds
t\geq 0: \tau^m_n(\rho)(\omega^*)+t=j\lambda_m \mbox{ for some } j\in
\Z_+ \mbox{ and } \\ \ds
  |\omega(\tau^m_n(\rho)(\omega^*)+t)-\omega(\tau^m_n(\rho)(\omega^*))|\geq
  \rho/4
\ea
\right\}.
\]
 On the other hand, suppose that $x^*\in \overline G\cap B_N(0)$. Let
 $\varepsilon = \rho/24$. It is easy to see that
 $f_{x^*,\varepsilon}(x^*)=\ML f_{x^*,\varepsilon}(x^*) = 0$.
 Applying the optional stopping theorem to the supermartingale in
 (\ref{q20}),
we obtain
 \[ \begin{array}{l} \ds
  \E^{\Q^{N,m}_{\omega^*,n}}[f_{x^*,\varepsilon}(\omega(\left\lceil
    (\tau^m_n(\rho)+ \tau^m(\rho)\wedge    \delta)/\lambda_m\right\rceil\lambda_m))] \\ \ds \qquad \leq
  \E^{\Q^{N,m}_{\omega^*,n}} \left[\sum_{j=(\left\lceil
        (\tau^m_n(\rho))/\lambda_m\right\rceil+1)\wedge
      (\kappa^{N,\lambda_m}/\lambda_m)}^{\left\lceil
        (\tau^m_n(\rho)+\tau^m(\rho)\wedge
        \delta)/\lambda_m\right\rceil\wedge
      (\kappa^{N,\lambda_m}/\lambda_m)}\ML
    f_{x^*,\varepsilon}(\omega(j\lambda_m))\lambda_m\right] \end{array}.\]
By property 3 of Assumption \ref{ass:TF}, there exists $C<\infty$
(depending only on $N$ and $\rho$) such that \[\sup_{y\in \overline G}
|\ML f_{x^*,\varepsilon}(y)|<C. \]   By property 2 of
Assumption \ref{ass:TF}, there exists $c > 0$ (depending only on $N$
and $\rho$) such that \[f_{x^*,\varepsilon}(y)\geq
c\ind_{\{|y-x^*|>3\varepsilon\}}, \qquad \mbox{ for all } y\in
\overline G.\]
Moreover, on the set $\{\tau^m(\rho)\leq
\delta\}$,\[ \begin{array}{ll}
  \ds|\omega(\tau^m_n(\rho)(\omega^*)+\tau^m(\rho))-x^*| \geq  \rho/4> 3
  \varepsilon.\end{array} \]
Combining the last four displays, we conclude that
   \be \label{prop1}
\ba{rcl}
\ds c \,\Q^{N,m}_{\omega^*,n}(\tau^m(\rho)\leq \delta)&  \leq & \ds
 \E^{\Q^{N,m}_{\omega^*,n}}\left[f_{x^*,\varepsilon}\left(\omega\left(\left\lceil
   \frac{\tau^m_n(\rho)+ \tau^m(\rho)\wedge \delta}{\lambda_m}\right\rceil \lambda_m\right)\right)\right] \\
& \leq & C\delta,
\ea
\ee
where the last inequality follows since $\delta$ is a multiple of $\lambda_m$.
This shows that condition 1 of Lemma \ref{prop:ctight} holds with $A_{\rho/4}=C/c$, which is a constant depending only on $\rho$ and $N$.

\noi
{\em Claim 3.}   For each $T>0$ and
$\varepsilon>0$, \[\lim_{m\rightarrow \infty}\sum_{0\leq
  j\lambda_m\leq
  T}\Q^{N,m}_{\omega'}(|\omega((j+1)\lambda_m)-\omega(j\lambda_m)|\geq
\varepsilon)=0.\]
{\em Proof of Claim 3.}
For each $0\leq j\lambda_m\leq T$, we have \[\begin{array}{ll} \ds\Q^{N,m}_{\omega'}(|\omega((j+1)\lambda_m)-\omega(j\lambda_m)|\geq \varepsilon)\\ \qquad \ds =\E^{\Q^{N,m}_{\omega'}}\left[\Q^{N,m}_{\omega'}\left(|\omega((j+1)\lambda_m)-\omega(j\lambda_m)|\geq \varepsilon | \MM_{j\lambda_m}\right)\right].\end{array} \]
By replacing $\tau^m(\rho)$ by $j\lambda_m$ in the
argument in {\em Claim 2}, for a regular conditional probability distribution
$\{\Q^{N,m}_{\omega^*,j}\}$  of $\Q^{N,m}_{\omega'}$ given
$\MM_{j\lambda_m}$, we have  that for each $x\in \overline G$ and $\varepsilon\in (0,\infty)$, \be \label{q3}\Q^{N,m}_{\omega^*,j}(\omega(t)=\omega^*(t),\
0\leq t\leq j\lambda_m)=1,\ee and  \begin{eqnarray} \label{q4} & &
  f_{x,\varepsilon}(\omega(\left\lceil (j\lambda_m+
    t)/\lambda_m\right\rceil\lambda_m))-f_{x,\varepsilon}(\omega(j\lambda_m))
  \\ & & \qquad \qquad -\sum_{k=(j+1)\wedge
    (\kappa^{N,\lambda_m}/\lambda_m)}^{\left\lceil
      (j\lambda_m+t)/\lambda_m\right\rceil\wedge
    (\kappa^{N,\lambda_m}/\lambda_m)}\ML
  f_{x,\varepsilon}(\omega(k\lambda_m))\lambda_m
  \nonumber \end{eqnarray} is a
$\Q^{N,m}_{\omega^*,j}$-supermartingale. Let $y^*=\omega^*(j\lambda_m)$, $\varepsilon_{l_1} = \varepsilon/3$ and $\varepsilon_{l_2}=\varepsilon/24$, then $3\varepsilon_{l_2}<\varepsilon_{l_1}/2$. By
an argument similar to that used to derive the first inequality in
(\ref{prop1}), there exists $c>0$ (depending only on $\varepsilon$ and
$N$) such that
 \be
\label{disp:claim3}
c
\Q^{N,m}_{\omega^*,j}(|\omega((j+1)\lambda_m)-\omega(j\lambda_m)|\geq
\varepsilon )\leq
\E^{\Q^{N,m}_{\omega^*,j}}[f_{y^*,\varepsilon_{l_1}}(\omega((j+1)\lambda_m))].
\ee
On the other hand, the supermartingale property in (\ref{q4})  implies
that
\be \label{est:f3}\E^{\Q^{N,m}_{\omega^*,j}}[f_{y^*,\varepsilon_{l_1}}(\omega((j+1)\lambda_m))]\leq
\lambda_m\E^{\Q^{N,m}_{\omega^*,j}}[\ML
f_{y^*,\varepsilon_{l_1}}(\omega((j+1)\lambda_m))].\ee
By property 2 of Assumption \ref{ass:TF}, $\ML
f_{y^*,\varepsilon_{l_1}}(y)=0$ for $|y-y^*|\leq \varepsilon_{l_1}/2$
and $f_{y^*,\varepsilon_{l_2}}(y)\geq c(N,\varepsilon_{l_2})>0$ for
$|y-y^*|>3\varepsilon_{l_2}$. When combined with property 3 of
Assumption \ref{ass:TF}, it follows that there exists
\[ K_N=\max\{C(N,\varepsilon_{l_1})/c(N,\varepsilon_{l_2}),
C(N,\varepsilon_{l_2})\}\] such that $|\ML
f_{y^*,\varepsilon_{l_1}}(y)|\leq K_N f_{y^*,\varepsilon_{l_2}}(y)$
and $|\ML f_{y^*,\varepsilon_{l_2}}(y)|\leq K_N$ for every $y\in
\overline G$. Together with (\ref{disp:claim3}) and (\ref{est:f3}),
applied first with $\varepsilon_{l_1}$ and then with
$\varepsilon_{l_2}$ in place of $\varepsilon_{l_1}$,
this implies that
\begin{eqnarray*}
  c\Q^{N,m}_{\omega^*,j}(|\omega((j+1)\lambda_m)-\omega(j\lambda_m)|\geq
  \varepsilon )&\leq & \lambda_m K_N\E^{\Q^{N,m}_{\omega^*,j}}[
  f_{y^*,\varepsilon_{l_2}}(\omega((j+1)\lambda_m))] \\ &\leq &
  \lambda_m^2 K_N\E^{\Q^{N,m}_{\omega^*,j}}[ \ML
  f_{y^*,\varepsilon_{l_2}}(\omega((j+1)\lambda_m))] \\ &\leq &
 \lambda_m^2 K_N^2.\end{eqnarray*}
Hence, we have \[\sum_{0\leq
  j\lambda_m\leq
  T}\Q^{N,m}_{\omega'}(|\omega((j+1)\lambda_m)-\omega(j\lambda_m)|\geq
\varepsilon)\leq TK_N^2\lambda_m/c,\] which converges to zero as
$m\rightarrow \infty$.  This establishes condition 2 of Lemma
\ref{prop:ctight}.

Finally,
(\ref{initial})  shows that condition 3 of Lemma \ref{prop:ctight} holds automatically.
Thus,  $\{\Q^{N,m}_{\omega'},m\in
\N\}$ satisfies all the conditions of Lemma \ref{prop:ctight} and
therefore is precompact.

\bigskip

\subsection{Convergence of the Approximating Sequence}
\label{subs:pfss}

We now turn to the proof of Proposition \ref{prop:conv}.
That any weak limit of $\{\Q^{N,m}, m \in \N\}$
satisfies the first two properties of
the submartingale problem  will be deduced primarily from results already
obtained in Section \ref{subs:tight}.  The verification of the third
property relies on some preliminary estimates that are first
established in Section \ref{subs-prelim}, and the proof  is completed in Section
\ref{subs-propconv}.


\subsubsection{Preliminary Estimates}
\label{subs-prelim}

\begin{lemma}
\label{lem:tildef}
There exists a set $F_0\in \MM_0$
with $\P^{\pi}(F_0)=0$ such that  for each
$\omega'\in \ccspace\sm F_0$,
\be \label{est:bd}\lim_{m\rightarrow \infty}\E^{\Q^{N,m}_{\omega'}}\left[\lambda_m\sum_{j=0}^{\left\lceil t/\lambda_{m_k}\right\rceil\wedge (\kappa^{N,\lambda_{m}}/\lambda_{m})}
\ind_{B_{\frac{1}{m}}(\partial G)}\left(\omega\left(j\lambda_{m}\right)\right) \right] =0.\ee
\end{lemma}
\begin{proof}
 Since $\pi$ is  a stationary distribution for each Markov chain,
the distribution of $\omega(j \lambda_m)$ under $\Q^m$ is equal to
$\pi$ for every $j, m \in \N$.
Together with (\ref{rep-qnm}) and  (\ref{restrict})  this implies that for
each $t\geq 0$,
\begin{eqnarray*}
&& \limsup_{m\rightarrow \infty}\int_{\ccspace} \E^{\Q^{N,m}_{\omega'}}\left[\lambda_m\sum_{j=0}^{\left\lceil t/\lambda_{m_k}\right\rceil\wedge (\kappa^{N,\lambda_{m}}/\lambda_{m})}
\ind_{B_{\frac{1}{m}}(\partial
  G)}\left(\omega\left(j\lambda_{m_k}\right)\right)
\right]\P^{\pi}( d \omega') \\
&& \qquad
 = \limsup_{m\rightarrow \infty}\E^{\Q^{N,m}}\left[\lambda_m\sum_{j=0}^{\left\lceil t/\lambda_{m}\right\rceil\wedge (\kappa^{N,\lambda_{m}}/\lambda_{m})}
\ind_{B_{\frac{1}{m}}(\partial
  G)}\left(\omega\left(j\lambda_{m}\right)\right)\right] \\
&& \qquad \leq  \limsup_{m\rightarrow \infty}\E^{\Q^{m}}\left[\lambda_m
  \sum_{j=0}^{\lceil t/\lambda_{m} \rceil }
\ind_{B_{\frac{1}{m}}(\partial
  G)}\left(\omega\left(j\lambda_{m}\right)\right)\right] \\
&& \qquad \leq t \lim_{m\rightarrow \infty} \int_{\partial G}
    \ind_{B_{\frac{1}{m}} (\partial G)} (z ) \, \pi (dz) \\
&& \qquad = 0,
\end{eqnarray*}
where the last equality  follows from the assumption
that $\pi (\partial G) = 0$.
Since the integrand in the first term above is nonnegative, this
completes the proof of the lemma.
\end{proof}
\bigskip

The following family of test functions will be used in the proof of
Theorem \ref{thm:SS}.
The proof relies on Assumption \ref{ass:V} and is purely analytic, and
hence is relegated to Appendix \ref{sec:lemg0}.

\begin{lemma} \label{lem:g0}
Suppose Assumption \ref{ass:V} holds.  For $x \in \MV$,  let
$\nvec_x\neq 0$ be the unit vector stated in Assumption \ref{ass:V} and let
$h(y)\doteq\lan \nvec_x,y-x\ran$ on $\overline G$.
Then for each $x \in \MV$,
there exist constants $c>0$ and $C>0$ such that for all
$\varepsilon>0$ sufficiently small and
$\delta\in (0,\varepsilon)$, there exists a non-negative
function $g_{\delta,\varepsilon}\in \C^2_c(\overline G)\oplus \R$ with
$-g_{\delta,\varepsilon}\in \MH$ such that
\begin{enumerate}
\item
$\sup_{y\in \overline G} g_{\delta,\varepsilon}(y)\leq C\varepsilon$;
\item $\sup_{y\in \overline G}|\nabla g_{\delta,\varepsilon}(y)|\leq
  C\sqrt{\varepsilon}$;
\item
$\sup_{y\in \overline G}|\lan b(y), \nabla g_{\delta,\varepsilon}(y)
\ran|\leq C\sqrt{\varepsilon}$;
\item for each  $y\in \overline G \cap B_{r_x}(x)$,
\begin{enumerate}
\item 
 $\sum_{i,j=1}^J a_{ij}(y)\frac{\partial^2g_{\delta,\varepsilon}(y)}{\partial
x_i\partial x_j} \geq c$ if
$\delta+2\sqrt{\delta}<h(y)<\varepsilon/2$,
\item
$\left|\sum_{i,j=1}^J a_{ij}(y)\frac{\partial^2g_{\delta,\varepsilon}(y)}{\partial
x_i\partial x_j}\right|\leq C\sqrt{\varepsilon}$ if   $h(y)\geq
\varepsilon$;
\item
$\sum_{i,j=1}^J a_{ij}(y)\frac{\partial^2g_{\delta,\varepsilon}(y)}{\partial
x_i\partial x_j} \geq 0$ otherwise.
\end{enumerate}
\end{enumerate}
 \end{lemma}
\bigskip

\subsubsection{Proof of Proposition \ref{prop:conv}}
\label{subs-propconv}

Let $F_0$ be as in Lemma \ref{lem:tildef}.
 Fix $\omega'\notin F_0$, and
let $\Q^{N,*}_{\omega'}$ be a weak limit of a convergent
subsequence $\{\Q^{N,m_k}_{\omega'},k\in \N\}$ in $\{Q^{N,m}_{\omega'}, m
\in \N\}$.
It suffices to show that $\Q^{N,*}_{\omega'}$ satisfies the three
properties of the  submartingale
problem stopped at $\chi^N$ with initial condition $\omega'(0)$ (see
Definition \ref{def-smgl}).
The first property follows directly from (\ref{initial}).
Also, by Proposition
\ref{prop:sub}, for each $k\in \N$,
for every $f\in \C_c^2(\overline G)$ such that $-f\in \MH$,
\be \label{dis:submart3} f(\omega(\left\lceil
  t/\lambda_{m_k}\right\rceil\lambda_{m_k}))-\lambda_{m_k}\sum_{l=1}^{\left
\lceil
      t\wedge \kappa^{N,\lambda_{m_k}}/\lambda_{m_k} \right\rceil}\ML
f(\omega(l\lambda_{m_k}))\ee   is a
$\Q^{N,m_k}_{\omega'}$-submartingale. Since $\Q^{N,m_k}_{\omega'}
\Rightarrow \Q^{N,*}_{\omega'}$, a standard convergence argument
together with an application of the optional stopping theorem
shows that for every $f\in \C_c^2(\overline G)$ such that $-f\in \MH$,\[f(\omega(t\wedge \chi^N))-\int_0^{t\wedge \chi^N}\ML
f(\omega(s))\,ds \] is a $\Q^{N,*}_{\omega'}$-submartingale, which
establishes the second property of the local submartingale problem.

We now turn to the proof of the third property.   Fix $x \in \MV$.
Let the constants $C, c$, the function $h$ and, for  each $\varepsilon>0$ sufficiently small and $\delta\in
(0,\varepsilon)$,  the function $g_{\delta,\varepsilon}$ be the
associated quantities from Lemma \ref{lem:g0}.
Since $g_{\delta,\varepsilon}$ is a nonnegative function that lies in ${\cal C}_c^2(\overline G)$,
Proposition \ref{prop:sub}, with $f$ replaced by $g_{\delta,\ve}$, and
property 1 of Lemma \ref{lem:g0} show that
\[
\begin{array}{l} \ds
  \E^{\Q^{N,m_k}_{\omega'}}\left[\lambda_{m_k}\sum_{l=1}^{\left
\lceil
      t\wedge \kappa^{N,\lambda_{m_k}}/\lambda_{m_k} \right\rceil}
\ML
    g_{\delta,\varepsilon}(\omega(l\lambda_{m_k}))\right]
  \\\ds \qquad \leq
  \E^{\Q^{N,m_k}_{\omega'}}[g_{\delta,\varepsilon}(\omega(\left\lceil
    t/\lambda_{m_k}\right\rceil\lambda_{m_k}))]-g_{\delta,\varepsilon}(\omega'(0))
\leq C \ve.
\end{array}
\]
Substituting the definition of the operator $\ML$ in (\ref{operL}),
and then using  property 2 of Lemma \ref{lem:g0},
 this implies that
\be
\label{est-aij}
\begin{array}{l}
\ds \E^{\Q^{N,m_k}_{\omega'}}\left[\lambda_{m_k} \sum_{l=1}^{\left
\lceil
      t\wedge \kappa^{N,\lambda_{m_k}}/\lambda_{m_k} \right\rceil}
\sum_{i,j=1}^J a_{ij}(\omega(l\lambda_{m_k}))\frac{\partial^2g_{\delta,\varepsilon}(\omega(l\lambda_{m_k}))}{\partial
x_i\partial x_j}\right] \\\ds \qquad\leq2C\varepsilon-
2\E^{\Q^{N,m_k}_{\omega'}}\left[\lambda_{m_k}\sum_{l=1}^{\left
\lceil
      t\wedge \kappa^{N,\lambda_{m_k}}/\lambda_{m_k} \right\rceil}
\lan b(\omega(l\lambda_{m_k})), \nabla
  g_{\delta,\varepsilon}(\omega(l\lambda_{m_k}))\ran \right]
\\\ds \qquad \leq  2C\varepsilon +2C\sqrt{\varepsilon}
(t+\lambda_{m_k}).
\end{array}
\ee
For each $k \in \N$,
using properties 4(a),  4(b) and 4(c) of Lemma \ref{lem:g0} and then
(\ref{est-aij}), we have
\begin{eqnarray*}
&& c\E^{\Q^{N,m_k}_{\omega'}}\left[\lambda_{m_k} \sum_{l=1}^{\left
\lceil
      t\wedge \kappa^{N,\lambda_{m_k}}/\lambda_{m_k} \right\rceil}\ind_{B_{r_x}(x)}(\omega(l\lambda_{m_k}))
\ind_{(\delta + 2 \sqrt{\delta}, \ve/2)} (h(\omega (l \lambda_{m_k}))) \right]
\\
& \leq &\E^{\Q^{N,m_k}}_{\omega'}\left[\lambda_{m_k}\sum_{l=1}^{\left
\lceil
      t\wedge \kappa^{N,\lambda_{m_k}}/\lambda_{m_k} \right\rceil}
\sum_{i,j=1}^J a_{ij}(\omega(l\lambda_{m_k}))\frac{\partial^2g_{\delta,\varepsilon}(\omega(l\lambda_{m_k}))}{\partial
x_i\partial x_j}\ind_{(0, \ve)} (h(\omega (l\lambda_{m_k}))) \right] \\
& \leq &\E^{\Q^{N,m_k}}_{\omega'}\left[\lambda_{m_k}\sum_{l=1}^{\left
\lceil
      t\wedge \kappa^{N,\lambda_{m_k}}/\lambda_{m_k} \right\rceil}
\sum_{i,j=1}^J a_{ij}(\omega(l\lambda_{m_k}))\frac{\partial^2g_{\delta,\varepsilon}(\omega(l\lambda_{m_k}))}{\partial
x_i\partial x_j}\right] \\ & & \qquad +  C \sqrt{\ve} (t + \lambda_{m_k}) \\
& \leq & 2C\varepsilon +3C\sqrt{\varepsilon}
(t+\lambda_{m_k}).
\end{eqnarray*}
Taking limits as $\delta \downarrow 0$ in the last display, we obtain
\begin{eqnarray*}
&& \E^{\Q^{N,m_k}}_{\omega'}\left[\lambda_{m_k}\sum_{l=1}^{\left
\lceil
      t\wedge \kappa^{N,\lambda_{m_k}}/\lambda_{m_k} \right\rceil}\ind_{B_{r_x}(x)}(\omega(l\lambda_{m_k}))
\ind_{(0, \ve/2)} (h(\omega (l \lambda_{m_k}))) \right] \\ & &  \quad \leq \dfrac{C}{c}
\left(2\varepsilon +3\sqrt{\varepsilon}(t+\lambda_{m_k})\right).
\end{eqnarray*}
Since, on $B_{r_x}(x)$,  $h(y) = 0$ if and only if $y = x$ and $x \in {\cal V}
\subset \partial G$,
combining the above inequality with
Lemma \ref{lem:tildef} we obtain
\begin{eqnarray*}
&&\lim_{k \ra \infty} \E^{\Q^{N,m_k}_{\omega'}}\left[\lambda_{m_k}\sum_{l=1}^{\left\lceil
      t/\lambda_{m_k}\right\rceil\wedge
    (\kappa^{N,\lambda_{m_k}}/\lambda_{m_k})}\ind_{B_{r_x}(x)}(\omega(l\lambda_{m_k}))
 \ind_{[0, \ve/2)} (h(\omega (l \lambda_{m_k}))) \right] \\ &&\quad \leq   \dfrac{C}{c}
\left(2\varepsilon +3\sqrt{\varepsilon}t\right).
 \end{eqnarray*}
Notice that  
\[ 
 \begin{array}{l}
 \ds\E^{\Q^{N,m_k}_{\omega'}}\left[\int_0^{t}\ind_{B_{r_x}(x)}(\omega(\lceil s/\lambda_m \rceil \lambda_m))\ind_{[0, \ve/2)} (h(\omega (\lceil s/\lambda_m \rceil \lambda_m)))ds \right]  \\ \ds\quad \leq
\E^{\Q^{N,m_k}_{\omega'}}\left[\int_0^{t \wedge  \kappa^{N,\lambda_{m_k}}}\ind_{B_{r_x}(x)}(\omega(\lceil s/\lambda_m \rceil \lambda_m))\ind_{[0, \ve/2)} (h(\omega (\lceil s/\lambda_m \rceil \lambda_m)))ds \right]  \\ \ds \quad \leq \E^{\Q^{N,m_k}_{\omega'}}\left[\lambda_{m_k}\sum_{l=1}^{\left\lceil
      t/\lambda_{m_k}\right\rceil\wedge
    (\kappa^{N,\lambda_{m_k}}/\lambda_{m_k})}\ind_{B_{r_x}(x)}(\omega(l\lambda_{m_k}))
 \ind_{[0, \ve/2)} (h(\omega (l \lambda_{m_k}))) \right],
\end{array}
 \]
where the first inequality follows from the fact that under $\Q^{N,m_k}_{\omega'}$, $h(\omega(t))>\varepsilon/2$ when $t\geq \kappa^{N,\lambda_{m_k}}$ and $\omega(t)\in B_{r_x}(x)$ and the second inequality follows the fact that the integrand is non-negative. Let $f\in \C(\R)$ be such that $0\leq f\leq \ind_{(-\varepsilon/2,\varepsilon/2)}$ and $f(0)=1$ and let $g \in\C(\R^J)$ be such that $0\leq g\leq \ind_{B_{r_x}}$ and $g(x)=1$. Since $h$ is non-negative, we have \begin{eqnarray*} && \E^{\Q^{N,m_k}_{\omega'}}\left[\int_0^{t}(gf) (h(\omega (\lceil s/\lambda_m \rceil \lambda_m)))ds \right]\\ && \qquad \leq \E^{\Q^{N,m_k}_{\omega'}}\left[\int_0^{t}\ind_{B_{r_x}(x)}(\omega(\lceil s/\lambda_m \rceil \lambda_m))\ind_{[0, \ve/2)} (h(\omega (\lceil s/\lambda_m \rceil \lambda_m)))ds \right].\end{eqnarray*}
 Since $\Q^{N,m_k}_{\omega'} \Rightarrow  \Q^{N,*}_{\omega'}$ as $k \ra \infty$,
by applying the Skorokhod representation theorem and Lemma A.4 of \cite{KanWil06} to the left-hand side of the above inequality and using the fact that $h(y)=0$ if and only if $y=x$ and $y\in B_{r_x}(x)$,
we conclude that
 \[\E^{\Q^{N,*}_{\omega'}}\left[\int_0^{t\wedge \chi^N}
\ind_{\{x\}}(\omega(s))\,ds \right]\leq \frac{1}{c}\left(2C\varepsilon
+2C\sqrt{\varepsilon} t\right).\]
Sending first $\varepsilon\rightarrow 0$ on the right-hand side of
the above inequality and then $t\rightarrow \infty$, we have
\[\E^{\Q^{N,*}_{\omega'}}\left[\int_0^{\chi^N} \ind_{\{x\}}(\omega(s))\,ds
\right]=0. \]
Because there are only a finite number of $x \in \MV$, summing the above
over $x \in \MV$ shows that $\Q^{N,*}_{\omega'}$ also satisfies
the third property in Definition \ref{def-smgl} with initial condition
$z = \omega'(0)$.  Since, for each $N \in \N$, the unique solution $\chi^N_z$ to the local submartingale
problem stopped at $\chi^N$ coincides on $\MM_{\chi^N}$ with the unique solution
$\Q_{z}$  to the submartingale problem with the same initial
condition,
this completes the proof of the proposition. \endproof

\section{Proof of Theorem \ref{thm:test}} 
\label{sec-proof2}

Since Assumption \ref{ass:V} follows directly from Assumption 2', 
 the proof of Theorem \ref{thm:test} essentially reduces to the
 construction of a test function that satisfies  Assumption \ref{ass:TF}.   
This construction involves patching together certain local test
functions, whose existence is first established in Proposition
\ref{prop-testfn}. The  proof of Proposition \ref{prop-testfn} is purely analytic, and is thus relegated to Appendix
\ref{ap-testfn}.

\begin{prop}
\label{prop-testfn}
Suppose $(G,d(\cdot))$ is piecewise ${\cal C}^1$ with continuous
reflection and $\MV \subset \partial G$ satisfies Assumption 2'.  
Then for every $x \in \overline{G}$ and $\varepsilon > 0$, there exists 
 a nonegative function $g_{x}^\ve \in {\cal H}\cap {\cal C}_c^2(\overline{G})$  such that 
\begin{enumerate}
\item 
$\supp[ g_{x}^\ve ]\cap \overline G \subseteq B_{\ve} (x) \cap \overline G$; 
\item 
there exists an open neighbourhood  ${\cal O}_{x, \ve}$ of $x$ such that 
\[  g_{x}^\ve (y)  > \frac{1}{2} \, \mbox{ for all }  y \in
{\cal O}_{x, \ve}.;
\]
\item there exists a constant $A(x,\ve)>0$ such that \[\sup_{y\in \overline G}|g_{x}^\ve(y)|\leq A(x,\ve), \ \sup_{y\in
\overline G}|\nabla g_{x}^\ve(y)|\leq A(x,\ve),\ \sup_{y\in
\overline G}\sum_{i,j=1}^J\left|\frac{\partial^2g_{x}^\ve(y)}{\partial
y_i\partial y_j}\right|<A(x,\ve).\]
\end{enumerate}
Moreover, when $(G, d(\cdot))$ is a polyhedral domain with piecewise constant
reflection, for every $\delta > 0$  there exists an open neighborhood
${\cal O}_{\emptyset,\delta,\ve}$ and a constant $A(\emptyset,\delta,\ve)$ such that 
${\cal O}_{x, \ve}={\cal O}_{\emptyset,\delta,\ve}$ and
$A(x,\ve)=A(\emptyset,\delta,\ve)$ for each  $x \in G \sm N_{\delta}
(\partial G)$  and for every $x \in \partial G \sm N_{\delta}
(\MV\cup\cup_{i\notin \MI(x)}(\partial G\cap \partial G_i))$,  
 there exists a finite collection of open neighborhoods and constants 
$\{{\cal O}_{F, \delta,\ve},\ A(F,\delta, \ve):\  F\in \{\MI(z):\
z\in \partial G \setminus \MV\}\}$ such that ${\cal O}_{x,\ve} = {\cal
  O}_{{\cal I}(x), \delta,\ve}$ and 
$A(x,\ve)=A(\MI(x),\delta,\ve)$.
\end{prop}
\bigskip

We now use Proposition \ref{prop-testfn} to prove Theorem
\ref{thm:test}.  Suppose first that $(G,d(\cdot))$ is piecewise ${\cal C}^1$ with continuous
reflection and $G$ is bounded.
Fix $\varepsilon>0$ and $N\in \N$.  
 Let the family of functions 
$g_{x}^\ve$ and associated sets ${\cal O}_{x,\ve}$, $x \in
\overline{G}$,  be as in Proposition
\ref{prop-testfn}.  
For each $z \in \overline{G}$, 
the set $\overline
G \cap\{x\in \R^J:\ |x-z|\geq 2\varepsilon\}$ is
compact. Therefore there exists a finite set 
\[  S_z = \{x_j\in \overline G,  |x_j-z|\geq 2\varepsilon, j=1,\cdots, l\} \]
such that 
\be
\label{ox-cover}
\cup_{x_j \in S_z} {\cal O}_{x_j,\ve}  \supset \overline
G \cap\{x\in \R^J:\ |x-z|\geq 2\varepsilon\}. 
\ee
Now, for each $z\in \overline G$, define 
\be \label{h}
h_{z}^\ve(y)=\sum_{x_j\in S_{z}}g_{x_j}^{\ve} (y), \qquad y\in \overline G.
\ee 
Now,  since $\overline G \cap B_N(0)$
is compact, there exists a finite cover of $\overline G \cap
B_N(0)$ of the form $\{B_{\varepsilon/2}(z_k), z_k\in
\overline G \cap B_N(0), k=1,\cdots, m_N\}$. For each $x\in
\overline G\cap B_N(0)$, if $x$ is contained in
$B_{\varepsilon/2}(z_k)$ for only one $k$, then  define
$f_{x,\varepsilon}=h_{z_k}^\ve$, and if $x$ is contained in
$B_{\varepsilon/2}(z_k)$ for more than one $k$, define
$f_{x,\varepsilon}=h_{z_{k^*}}^\ve$ where $k^*$ is the smallest index among
those $k$'s.  

We now verify that the functions $\{f_{x,\varepsilon}, x \in \overline{G} \cap B_N(0),
\varepsilon > 0\}$ satisfy the properties stated in Assumption
\ref{ass:TF}.   Since each $g_{x}^\ve$ lies in ${\cal H} \cap {\cal C}_c^2(\overline{G})$ by Proposition
\ref{prop-testfn}, 
it follows that each $h_z^\ve$, and therefore each 
$f_{x,\varepsilon}$, is a finite sum of functions in $ {\cal H}$, and
the first property follows. 
Now, fix $x \in \overline{G}\cap B_N(0)$ and let $z_k$  be such that 
$|z_k - x| < \ve/2$ and $f_{x,\ve}  = h_{z_k}^\ve$.  
Note that $|y-x| > 3 \ve$ implies $|y-z_k| > 2 \ve$, which in turn 
implies that $y \in \cup_{x_j \in S_{z_k}^N} {\cal O}_{x_j,\ve}$ due to 
(\ref{ox-cover}).     The definition (\ref{h}) of $h_{z_k}^\ve$, the fact that
each $g_{x_j}^\ve$ is nonnegative and property 2 of Proposition
\ref{prop-testfn} then imply that $h_{z_k}^\ve (y) > 1/2$. 
On the other hand, suppose $|y-x| < \ve/2$.  Then $|y-z_k| < \ve$. 
However, the inequalities $|x_j - z_k| > 2 \ve$ and  
$\supp [g_{x_j}] \subset B_{\ve}(x)$ (with the latter inequality resulting from  property 1 of Proposition
\ref{prop-testfn}) for all $x_j \in S_{z_k}^N$
imply that $h_{z_k} (y) = 0$ for $|y-z_k| < \ve$.  Thus, we have 
shown that $f_{x,\ve}$ satisfies property 2 of Assumption
\ref{ass:TF}. 
Lastly, since each $g_{x}^{\ve} \in {\cal C}_c^2 (\overline{G})$ and $b$
and $\sigma$ are continuous, there exists $C_{x,\ve} < \infty$ such that 
$|\ML g_{x}^{\ve} (y)| \leq C_{x,\ve}$ for every $y \in \overline{G}$.  
Thus, if we set $C (N,\ve) \doteq \max_{k=1, \ldots, m_N} \sum_{x \in
  S_{z_k}} C_{x,\ve}$, it is clear that $f_{x,\ve}$ satisfies
property 3 of Assumption \ref{ass:TF}. 
This completes the proof of Theorem \ref{thm:test} for the case of a
bounded domain.  

Next, suppose that $(G, d(\cdot))$ is a polyhedral domain with piecewise constant
reflection. Since $G$ is a polyhedron and the set $\MV$ is finite, it follows from the second part of Proposition \ref{prop-testfn} that  
there exist a constant integer $m(z,\ve)>0$, a constant $C_{z,\ve} < \infty$ and a countable set 
\[  S_z = \{x_j\in \overline G:  |x_j-z|\geq 2\varepsilon, j\in \N\} \]
such that 
\be
\label{ox-cover2}
\cup_{x_j \in S_z} {\cal O}_{x_j,\ve}  \supset \overline
G \cap\{x\in \R^J:\ |x-z|\geq 2\varepsilon\},
\ee
for each $x \in \overline G$ such that  $|x-z|\geq 2\varepsilon$,  
there are at most $m(z,\ve)$ open sets  in $\{{\cal O}_{x_j,\ve}:\ x_j\in S_z\}$ that contain $x$, and 
 $\sup_{x_j\in S_z}|\ML g_{x_j}^{\ve} (y)| \leq C_{z,\ve}$ for every $y\in \overline G$.
Thus, we can define the function $h_{z}^\ve$ as in (\ref{h}) for each $z\in \overline G$ and follow the previous argument for the case when $(G,d(\cdot))$ is piecewise ${\cal C}^1$ with continuous
reflection and $G$ is bounded to complete the proof of Theorem \ref{thm:test} with $c(N,\ve)=1/2$ and $C(N,\ve)\doteq \max_{k=1, \ldots, m_N} m(z_k,\ve) C_{z_k,\ve}$. \endproof

\appendix

\section{Construction of Local Test Functions} 
\label{ap-testfn} 
 
This section is devoted to the proof of Proposition
\ref{prop-testfn}.   
Consider $(G,d(\cdot))$ that are piecewise $\C^1$ with continuous 
reflection.  
We first construct a family of test functions $g_{x,r}$, for
$x \in \overline{G}$ and sufficiently small $r > 0$. 
The nature of the construction is different for the cases $x \in G$, $x \in
\partial G\setminus \MV$ or $x \in  \MV$, and is presented below in Propositions 
\ref{lem:TF3}, \ref{lem:test1} and \ref{prop:TF2}, respectively. 
The proof of Proposition \ref{prop-testfn} is given at the end of the
section.

\begin{prop} \label{lem:TF3}
For each $x\in G$,  there is a constant $A_x <  \infty$ such that 
 for every $r \in (0, r_x)$, where $r_x \doteq \left(\dist (x, \partial G)\right)^2$, there exists a nonnegative function
$g_{x,r}\in \C^2_c(\overline G)$ such that 
\begin{enumerate}
\item
$\mbox{supp}[g_{x,r}]\cap \overline G \subset B_{\sqrt{r}}(x)\subset G$;
\item $0 \leq g_{x,r}(y) \leq 1$ for all $y \in\R^J$ and
  $g_{x,r}(y)=1$ for each $y\in B_{\sqrt{r}/2}(x)$; 
\item the following bounds are satisfied: 
 \[\sup_{y\in \overline G}|g_{x,r}(y)|\leq A_x, \ \sup_{y\in
\overline G}|\nabla g_{x,r}(y)|\leq \frac{A_x}{r},\ \sup_{y\in
\overline G}\sum_{i,j=1}^J\left|\frac{\partial^2g_{x,r}(y)}{\partial
y_i\partial y_j}\right|<\frac{A_x}{r^2};\]
\end{enumerate}
\end{prop}
\begin{proof}  Let  $\xi$ be a bounded $\C^\infty$ function
on $\R$ such that $\xi(z) = 1$ when $z \leq 1/2$, $\xi(z) = 0$ 
when $z>1$, and $\xi$ is strictly decreasing in the
interval $(1/2,1)$, and note that then $\|\xi^\prime\|_{\infty} <
\infty$ and $\|\xi^{\prime \prime}\|_{\infty} < \infty$.  
For each $x\in G$ and $0<r<r_x$,  
define $g_{x,r}(y)\doteq\xi(|y-x|^2/r)$ for $y\in \R^J$. 
The three properties of $g_{x,r}$ are then easily verified.  The first
inclusion in the first property 
 holds because $|x-y|^2/r > 1$ when $y \not \in B_{\sqrt{r}} (x)$ and 
 $\xi (z) = 0$  when $z > 1$, whereas the second inclusion holds
 because  $\sqrt{r} < \sqrt{r_x} =  \dist(x, \partial G)$.   The second property
is satisfied because $y \in B_{\sqrt{r}/2} (x)$ implies $(y-x)^2/r \leq 1/4$,
and $\xi(z) = 1$ for $z \leq 1/4$, and   the last property holds with 
$A_x \doteq  \max(1, 2\|\xi^\prime\|_{\infty} \dist(x, \partial G), (4(J^2-J)
\|\xi^{\prime \prime}\|_{\infty} + 2J\|\xi^\prime\|_{\infty}) \left(\dist (x, \partial G)\right)^2)$.  
\end{proof}
\begin{remark} \label{rem:gxr}
{\em 
The function $g_{x,r}$ in Proposition \ref{lem:TF3} is translation invariant in $G$ in the sense that $g_{x,r}(y)=g_{x+\delta,r}(y+\delta)$ for each $y\in B_{\sqrt{r}}(x)$ and $r<r_x\wedge r_{x+\delta}$.}
\end{remark}
\bigskip

We now paraphrase a result from \cite{Ram06} that will be used to construct local functions associated with points 
$x \in \partial G$.  

\begin{lemma}
\label{prop:TF}  
Let $\overline{{\cal C}}$  be a closed convex cone with vertex at the
origin and a boundary that is ${\cal C}^\infty$, except possibly at
the vertex.   Given any  closed, convex, compact subset  ${\cal  K}$ of the interior of 
$\overline{{\cal C}}$, constants $0 < \csto < \cstt < \infty$ and $\ve > 0$, there exists
a ${\cal C}^\infty$ function $\ell$ on the set 
\[  \Lambda \doteq \{ y \in \R^J: \csto < \dist (y, \overline{{\cal C}})
< \cstt \} 
\]
that satisfies the following properties: 
\begin{enumerate}
\item 
$\sup_{z \in \Lambda} \left( \left|\ell(z) - \dist (z, \overline{{\cal C}}) \right|
\vee \left|\nabla \ell(z) - 1 \right|\right)\leq  \ve$; 
\item 
for every $j, k \in \{1, \ldots, J\}$, $\sup_{z \in \Lambda} 
 \left| \frac{\partial^2 \ell(z)}{\partial    z_j \partial z_k}\right| < \frac{3}{\eta}+1$. 
\item 
there exists $\theta > 0$ such that 
\[  \lan \nabla \ell (z), p \ran \leq - \theta, \qquad \mbox{ for } p
\in {\cal K} \mbox{ and } z \in \Lambda.
\]
\end{enumerate}
On the other hand, if  $\overline{{\cal C}}$ is a half-space so that
its boundary is ${\cal
  C}^\infty$ everywhere, given any subset ${\cal K}$  of
$\overline{{\cal C}}$,   the 
function $\ell (x) \doteq \dist (x, \overline{{\cal C}})$, $x \in \Lambda$, is  a ${\cal C}^2$ function 
on $\Lambda$ that satisfies properties (1) and (2) above and 
property (3) with $\theta = 0$. 
\end{lemma}

\noi 
As shown in \cite{Ram06}, a function $\ell$ with the  properties
stated above can be constructed as a suitable mollification of the
distance function to the cone $\overline{{\cal C}}$.  
Indeed, Lemma \ref{prop:TF} follows immediately from the proof of Lemma 6.2 of
\cite{Ram06},  with  $g_C$, $L_{C,\delta_C}$, $K_{C}^{\delta_C/3}$, 
$\tilde \eta_C$, $\tilde \lambda_C$ and $\tilde
 \varepsilon_C$ therein replaced by  $l(\cdot)$, $\overline{{\cal C}}$,  ${\cal
   K}$, $\csto$,  $\cstt$ and $\ve$, respectively. 

We now construct the second class of test functions associated with $x
\in \MV$.

\begin{prop} \label{lem:test1} Suppose Assumption 2' holds with unit vector $\nvec_x$ and 
  associated constants $r_x > 0$, $0< \cxone < \cxtwo < \infty$,
  $\alpha_x > 0$, $x \in \MV$.  
For each $x\in \MV$ there is a constant
$A_x  < \infty$ such that for every $r \in (0, r_x/\cxtwo)$, there exists a 
nonnegative function $g_{x,r}\in \C^2_c(\overline G)$ such  that the following three properties hold:
\begin{enumerate}
\item $\supp[g_{x,r}]\cap \overline G\subset \overline
  B_{\cxtwo 
    r}(x) \cap \overline G \subset B_{r_x} (x) \cap \overline{G}$;
\item $g_{x,r}(y)=1$ for each $y\in B_{\cxone r}(x) \cap \overline G$;
\item the following bounds are satisfied: 
\[\sup_{y\in \overline G}|g_{x,r}(y)|\leq A_x, \ \sup_{y\in
\overline G}|\nabla g_{x,r}(y)|\leq \frac{A_x}{r},\ \sup_{y\in
\overline G}\sum_{i,j=1}^J\left|\frac{\partial^2g_{x,r}(y)}{\partial
y_i\partial y_j}\right|<\frac{A_x}{r^2}.\]
\item $\lan d, \nabla g_{x,r}(y) \ran \leq 0$ for $d\in d(y)$ and
  $y\in \partial G$. 
\end{enumerate}
\end{prop}
\begin{proof} Fix $x\in \MV$ and $r\in (0,r_x/c_x^2)$.  
 By  (\ref{include})  of Assumption 2',  we have 
\begin{eqnarray}\sup_{y\in \overline G\cap B_{r_x}(x)}
  \dist(y,x+r\nvec_x+\Theta_x) &= & \sup_{y\in \overline G \cap
    B_{\cxtwo r}(x)} \dist(y,x+r\nvec_x+\Theta_x)\nonumber \\ &\leq & \sup_{y\in
    \overline G \cap B_{\cxtwo r}(x)} |y-x-r\nvec_x|\leq
  (\cxtwo+1)r. \label{thetainclude}\end{eqnarray} 
We first argue by contradiction to establish the claim that the quantity 
 \[\kappa_x\doteq \inf_{r\in (0,r_x/\cxtwo)}\inf_{y\in \overline G:\
  |y-x|/r\leq \cxone }\dist((y-x)/r,\nvec_x+\Theta_x)\]
is strictly positive. Suppose, instead, that
$\kappa_x=0$. Then there exists a sequence $\{(y_m,r_m):\ m\in \N\}$
such that for each $m\in \N$, $r_m\in (0,r_x/\cxtwo)$, $y_m\in
\overline G$,  $|y_m-x|/r_m\leq \cxone $ and
$\dist((y_m-x)/r_m,\nvec_x+\Theta_x) \rightarrow 0$ as $m\rightarrow
\infty$. Since $(y_m-x)/r_m$ 
is bounded by $\cxone$ uniformly in $m$, by choosing a 
subsequence if necessary, we may assume without loss of generality
that there exists $y^*$ such that $|y^*|\leq \cxone $ and
$(y_m-x)/r_m\rightarrow y^*$ as $m\rightarrow \infty$.  This implies
that $\dist(y^*,\nvec_x+\Theta_x)=0$ or, in other words, that 
 $y^*=\nvec_x+z^*$ for some $z^*\in \Theta_x$. Thus, since $\nvec_x$ is a unit
 vector and $z^* \in \Theta_x$ implies $\lan \nvec_x, z^* \ran \geq 0$, 
 \[|y^*|^2=\lan\nvec_x+z^*,\nvec_x+z^*\ran = 1 +
2\left<\nvec_x,z^*\right> +|z^*|^2 \geq 1.\] But this contradicts the
inequalities $|y^*|\leq \cxone  <1$.   Thus, we have
shown that $\kappa_x>0$.  Since $\Theta_x$ is a half-plane that goes
through the origin, this
implies that 
\be  \label{kappar}\dist(y,x+r\nvec_x+\Theta_x) \geq \kappa_x r \mbox{ for each } y\in \overline G \cap B_{\cxone r}(x). \ee

Consider the set 
\[  \Lambda_x \doteq \left\{ z \in \R^J:  \frac{\kappa_x}{32} < \dist(z,
  \Theta_x) < \overline{c}_x + 3 \right\}, 
\]
and define $\ell_x (z) \doteq \dist (z, \Theta_x)$ for $z \in
\Lambda_x$.   Then, since $\Theta_x$ is a half-space and property 2 of Assumption 2' shows that 
$\cup_{y \in \overline G\cap B_{r_x}(x)} d(y) \subseteq \Theta_x$, it follows 
from  Lemma \ref{prop:TF} that $\ell_x$ satisfies properties (1)--(3) 
therein with $\Lambda = \Lambda_x$, $\overline{{\cal C}} = \Theta_x$, ${\cal K} = \cup_{y \in \overline G\cap B_{r_x} (x)} d(y)$, 
$\csto = \kappa_x/32$, $\cstt = \cxtwo+3$, $\ve = \kappa_x/64$ and
$\theta = 0$. 
Now, let $\zeta$ be a $\C^\infty$ function defined on $\R$ such that
$\zeta(s)= 0$  when $s<1/2$, $\zeta(s) = 1$ 
when $s>1$, and $\zeta$ is strictly increasing in the interval
$(1/2,1)$. Then, for each $y\in \R^J$, define 
\[ g_{x,r}(y)\doteq \left\{\begin{array}{ll}
    \zeta\left(\frac{2}{\kappa_x} \ell_x\left(\frac{y-x-r
          \nvec_x}{r}\right)\right) & \mbox{ if } \frac{y-x-r \nvec_x}{r}\in
    \Lambda_x \mbox{ and } y \in B_{r_x}(x), \\ 0 & \mbox{ otherwise.} \end{array}\right.\]  

The first property of $\ell_x$ from Lemma \ref{prop:TF} and the definition of $\zeta$ imply that
 \[ \begin{array}{rcl} \ds \supp[g_{x,r}]\cap \overline G 
   & \subseteq & \{y\in \overline G \cap B_{r_x}(x):\ l_x\left(\frac{y-x - r
       \nvec_x}{r}\right)> \frac{\kappa_x}{4}\} \\ 
& \subseteq &   \{y\in \overline G \cap B_{r_x} (x):\ \dist(y,x + r\nvec_x+\Theta_x)>
\frac{\kappa_x r}{8}\}. 
\end{array}
\]
However, by the property (\ref{include}) of $\cxtwo$ stated in Assumption 2'
and the fact that $r \cxtwo < r_x$, it follows that
\[    \{y\in \overline G \cap B_{r_x} (x):\ \dist(y,x + r\nvec_x+\Theta_x)>
\kappa_x r/8\} \subset \overline{G} \cap B_{\cxtwo r}(x). 
\]
The last two assertions and (\ref{thetainclude}) together imply that $\supp[g_{x,r}]\cap \overline G \subset \Lambda_x$. This and the fact that 
$\zeta \in {\cal C}^\infty
(\R)$ and $\ell_x \in {\cal C}^\infty (\Lambda_x)$ imply that $g_{x,r}$
is  ${\cal C}^\infty_c (\R^J)$. The last two assertions also imply that $g_{x,r}$ satisfies
the first property stated in the proposition. 
   Next, for $y \in \overline{G} \cap B_{\cxone r} (x)$, by
(\ref{kappar}) it follows that $\dist(\frac{y-x-r \nvec_x}{r}, \Theta_x)\geq
\kappa_x$, and by the first property  of $\ell_x$ from Lemma \ref{prop:TF}, we have 
$\ell_x(\frac{y-x-r \nvec_x}{r}) \geq  63\kappa_x/64$, 
which in turn implies
$\zeta\left(\frac{2}{\kappa_x}\ell_x\left(\frac{y-x-r
      \nvec_x}{r}\right)\right)=1$. Thus, $g_{x,r}(y)=1$ for all $y\in
\overline G \cap B_{\cxone r}(x)$, 
showing that the second property holds.  
The third property is easily verified using the form of $g_{x,r}$, 
the fact that  $g_{x,r}
\in {\cal C}_c^\infty (\overline{G})$.  
Finally,  the fourth property of $g_{x,r}$ holds because $\ell$ satisfies the third
property of Lemma \ref{prop:TF} with $\theta = 0$, and  $\zeta$ is
non-decreasing.  
\end{proof}
\bigskip

We now turn to  the construction of  local test functions associated
with $x \in \partial G\setminus \MV$. 
For this, we first introduce some geometric objects associated with
the directions of reflection,  similar to those introduced in 
 Section 6.1 of \cite{Ram06} in the context of polyhedral domains.  
For $x \in \partial G$,  let 
\be \label{Kx}K_x \doteq \left\{-\sum_{i\in \MI(x)}a_i\gamma^i(x):
  a_i\geq 0, i \in \MI(x), \sum_{i\in \MI(x)}a_i=1 \right\}. \ee
Note that $K_x$ is a convex, compact subset of $\R^J$.  
  Therefore, 
 there exist $\delta_x>0$ and a compact, convex
set $K_{x,\delta_x}$ such that $K_{x,\delta_x}$ has $\C^\infty$
boundary and satisfies \be \label{func1}K_x^{\delta_x/2}\subset
(K_{x,\delta_x})^\circ\subset K_{x,\delta_x} \subset
K_x^{\delta_x}, \ee 
where we recall that $K_x^\ve$ denotes the $\ve$-fattening of the set
$K_x$ for every $\ve > 0$.  
Now, if $x \in \partial G\setminus \MV\subset \MU$, then it is easy to see that 
$0 \not \in K_x$ and 
\be \label{sep0} \min_{i\in \MI(x)}\lan n^i(x),d\ran <0 \qquad
 \mbox{for every } d\in K_x.\ee
Therefore,  $\delta_x > 0$ can be chosen such that $0 \notin K_{x,\delta_x}$ and 
\be \label{est:test11}\min_{i\in \MI(x)}\lan
n^i(x),d\ran <0 \qquad \mbox{for every } d\in K_x^{\delta_x},\ee
We first establish an elementary result that will be used in the
construction.   

\begin{lemma}\label{lem:A1} For $x \in \partial G\setminus \MV$, 
there exist $R_x\in (0,1)$ and $\beta_x>0$ such that 
\be \label{sep2} \min_{i\in \MI(x)}\lan n^i(x),d\ran <-2\beta_x|d|
\qquad \mbox{for every } d\in \cup_{t\in [0, R_x]}tK_{x,\delta_x}\ee
and 
\be \label{sep1}
\left(x+\cup_{t\in [0, R_x]} tK_{x,\delta_x}\right) \cap \overline G = \{x\}. \ee
Moreover, if $n^i(\cdot)$ and $\gamma^i(\cdot)$ are constant vector fields for each $i\in \MI$,  then $R_x$ and $\delta_x$ depend on $x$ only through $\MI(x)$.
\end{lemma}
\begin{proof} We first use an argument by contradiction to prove that 
\be \label{sep3}\sup_{d\in K_{x,\delta_x}} \min_{i\in \MI(x)}\lan n^i(x),d/|d|\ran  < 0. \ee
Suppose (\ref{sep3}) does not hold. Then there exists a sequence
$\{d_n,\ n\geq 1\}\subset K_{x,\delta_x}$ such that $\min_{i\in
  \MI(x)}\lan n^i(x),d_n/|d_n|\ran \geq 0$ for all $n\geq 1$.  
Since $K_{x,\delta_x}$ is compact, by choosing a subsequence if
necessary, 
we may assume without loss of generality that $d_n\rightarrow d \in
K_{x,\delta_x}$ as $n\rightarrow \infty$. Thus, $\min_{i\in
  \MI(x)}\lan n^i(x),d/|d|\ran$   $\geq 0$, which contradicts
(\ref{est:test11}).
Thus, (\ref{sep3}) holds and, in turn, this implies that    there exists $\beta_x>0$
such that  (\ref{sep2})  holds. 

On the other hand, for each $i \in
\MI(x)$, 
since $\partial G_i$ is $\C^1$ near $x$, it follows that 
\[\lim_{\delta \rightarrow 0}\inf_{y\in \overline G:\ |y-x|\leq
  \delta} \min_{i\in \MI(x)}\lan n^i(x), (y-x)/|y-x| \ran \geq 0. \]  
Together with (\ref{sep3}) this shows that there exists $R_x\in (0,1)$ such that 
\be \label{sep4}
\inf_{y\in \overline G:\ |y-x|\leq R_x(\sum_{i\in \MI(x)}|\gamma^i(x)|+\delta_x)} \min_{i\in \MI(x)}\lan
n^i(x), \frac{y-x}{|y-x|} \ran > \sup_{d\in K_{x,\delta_x}} \min_{i\in
  \MI(x)}\lan n^i(x),\frac{d}{|d|}s\ran.\ee
 We now use this to prove (\ref{sep1}) by contradiction.  Suppose that
 (\ref{sep1}) does not hold. Then there exists $d\in \cup_{t\in [0,
   R_x]}tK_{x,\delta_x}$ such that $d\neq 0$ and $x+d\in \overline
 G$. Since $d\in \cup_{t\in [0, R_x]}tK_{x,\delta_x}$, there is
 $t^*\leq R_x$ and $d^*\in K_{x,\delta_x}$ such that $d=t^*d^*$ and
 $|d|\leq R_x(\sum_{i\in \MI(x)}|\gamma^i(x)|+\delta_x)$.  
Then (\ref{sep4}) implies that 
\[ \min_{i\in \MI(x)}\lan n^i(x), d^*/|d^*| \ran= \min_{i\in
  \MI(x)}\lan n^i(x), d/|d| \ran > \sup_{d\in K_{x,\delta_x}}
\min_{i\in \MI(x)}\lan n^i(x),d/|d|\ran, \] which contradicts the fact
that $d^* \in K_{x,\delta_x}$. 

The last statement in the lemma follows directly from the constancy of $n^i(\cdot)$, $i\in \MI$, and the fact that $K_{x,\delta_x}$ and $\delta_x$ can be chosen to depend on $x$ only through $\MI(x)$ in this case. \end{proof}
\bigskip

For each $i\in \MI(x)$,  since $\partial G_i$ is $\C^1$ near
$x\in \partial G$,  the hyperplane $\{y\in \R^J:\ \lan n^i(x),y-x\ran = 0\}$ is  the tangent plane to $\partial G_i$ at $x$ for each $i\in \MI(x)$. Let
\be \label{Sx} S_x \doteq \cap_{i\in \MI(x)}\{y\in \R^J:\ \lan n^i(x),y-x\ran
\geq 0\}. \ee Then  $\overline G$ can be locally approximated near $x$
by the polyhedral cone $S_x$ in the sense that for each $N>0$,
\be \label{approxG}\{x+(y-x)/r \in \R^J:\ y\in \overline G,\ |y-x|\leq
N r\} \rightarrow S_x \cap B_N(x)\ \mbox{ as } r\rightarrow 0,\ee where the convergence is under the Hausdorff distance.
 In view of (\ref{sep1}), it follows that there exist $0<r_x<\dist(x,
 \MV \cup 
 \cup_{i\notin \MI(x)}(\partial G\cap \partial G_i))$ and
 $\lambda_x\in (0,1)$ such that for each $r\in (0,r_x)$,
 \be \label{include1}\left\{y\in \R^J:\ \dist(y,x+\cup_{t\leq
     R_x}tK_{x,\delta_x})\leq 3\lambda_x r\right\}\cap \partial G
 \subset B_r(x)\cap \partial G\ee and 
\be \label{include2}\left\{y\in \R^J:\ \dist(y,x+\cup_{t\leq R_x}tK_{x,\delta_x})\leq 3\lambda_x r\right\}\cap \overline G \cap \partial B_r(x)=\emptyset.\ee
Let $L_{x,\delta_x}$ be a truncated (half) cone with vertex at the
origin defined by  
\be \label{Lx} L_{x,\delta_x}\doteq \cup_{t\leq R_x/2}tK_{x,\delta_x}.\ee
Then (\ref{sep1}) implies 
\[(x+L_{x,\delta_x})\cap\overline G = \{x\}.\]

\begin{prop}\label{prop:TF2} For each $x\in \partial G\setminus \MV$, there exist
  $0<r_x<\dist(x, \MV \cup \cup_{i\notin \MI(x)}(\partial G\cap \partial
  G_i))$,  $A_x\in (0,\infty)$,  and a family of 
nonnegative functions $\{g_{x,r}\in \C^2_c(\overline G):\ r\in
(0,r_x]\}$ that satisfy the
following additional properties:
\begin{enumerate}
\item
$\mbox{supp}[g_{x,r}]\cap \overline G \subset B_{r}(x)\cap \overline
G$; 
\item $g_{x,r}(y)=1$ for each $y\in B_{\delta_{x,r}}(x) \cap \overline
  G$ for some constant $\delta_{x,r}>0$; 
\item
for $r\in (0,r_x)$, \[\sup_{y\in \overline G}|g_{x,r}(y)|\leq A_x, \ \sup_{y\in
\overline G}|\nabla g_{x,r}(y)|\leq \frac{A_x}{r},\ \sup_{y\in
\overline G}\sum_{i,j=1}^J\left|\frac{\partial^2g_{x,r}(y)}{\partial
y_i\partial y_j}\right|<\frac{A_x}{r^2};\]
\item $\lan d, \nabla g_{x,r}(y) \ran \leq 0$  for all $d\in d(y)$ and
  $ y\in \partial G$. 
\end{enumerate}
\end{prop}
\begin{proof} Fix $x\in \partial G\setminus \MV$. Let $q_x$ be a unit vector in the set $K_x$
defined in (\ref{Kx}) such that $-q_x$ points into $G$ from $x$, and
for each $r\in (0,1)$, define \be \label{Mxr} M(x,r)\doteq x-\lambda_x
\frac{R_x}{2}rq_x+rL_{x,\delta_x} \ee and  \[ M^{2\lambda_x
  r}(x,r)\doteq \{y\in \R^J:\ \dist(y,M(x,r))\leq 2\lambda_x r\}.\]
Recalling $R_x < 1$, it is easy to see that for each $y\in
M^{2\lambda_x r}(x,r)$, \[\dist(y,x+\cup_{t\leq R_x}tK_{x,\delta_x})
\leq 2\lambda_x r + \left|\lambda_x \frac{R_x}{2}rq_x\right|\leq 3 \lambda_x r.\]
Thus, \begin{eqnarray*}M^{2\lambda_x r}(x,r)\subseteq \{y\in \R^J:\
  \dist(y,x+\cup_{t\leq R_x}tK_{x,\delta_x})\leq 3\lambda_x
  r\}\end{eqnarray*} 
and hence, by (\ref{include1})--(\ref{include2}) we have 
\[ M^{2\lambda_x r}(x,r) \cap \overline G \cap \partial
B_r(x)=\emptyset,\] and \be \label{sub1} M^{2\lambda_x r}(x,r)\cap
\overline G \subset B_r(x) \cap \overline G. \ee 

It follows from Lemma \ref{prop:TF} with $\overline{{\cal C}} = \cup_{t \geq 0} t
K_{x,\delta_x}$,   ${\cal K} = K_x^{\delta_x/3}$, $\lambda =
2\lambda_x$,  $\eta = \eta_x \in (0,\lambda_x)$, $\Lambda  =
\Lambda_x\doteq \{y\in \R^J:\ \eta_x<\dist\left(y, \cup_{t\geq 0}tK_{x,\delta_x}\right)\leq 2\lambda_x\}$  and
$\varepsilon_x=\lambda_x/12\wedge \eta_x/2$ that  there exists a
function $\ell_x:\ \Lambda_x \rightarrow \R$, that satisfies all the
properties stated in Lemma \ref{prop:TF}. 
Let \be \label{Oxr} O(x,r)\doteq \overline G \cap (M^{2\lambda_x r}(x,r) \setminus M^{\eta_x r}(x,r))\subset B_r(x) \cap \overline G.\ee
Note that by the local approximation of $\overline G$ by $S_x$ at $x$
in (\ref{approxG}),  by possibly making $r_x$ and $\lambda_x$ smaller,
we may assume that for each $r\in (0,r_x)$ and $y\in O(x,r)$, that is,
$y\in \overline G$ such that \be
\label{star}
\eta_x<\dist\left(\frac{y-x}{r}+\lambda_x
  \frac{R_x}{2} q_x, L_{x,\delta_x}\right) \leq 2\lambda_x,
\ee the projection of
$(y-x)/r+\lambda_x (R_x/2)q_x$ to $L_{x,\delta_x}$ 
coincides with the projection of $(y-x)/r+\lambda_x (R_x/2)q_x$ onto
$\cup_{t\geq 0}tK_{x,\delta_x}$ since $L_{x,\delta_x}$ is the portion
of $\cup_{t\geq 0}tK_{x,\delta_x}$ truncated near the vertex. Let
$k_{x,r}$ be the function on $O(x,r)$ given by 
\be  \label{kxr} k_{x,r}(y)\doteq \ell_x\left(\frac{y-x}{r}+\lambda_x
  \frac{R_x}{2} q_x\right), \qquad y \in \R^J.\ee  Then the properties of
  $\ell_x$ stated in Lemma \ref{prop:TF} and (\ref{Oxr}) imply that
  $k_{x,r} \in \C^\infty (O(x,r))$ and $k_{x,r}$ satisfies the following
  additional 
  properties: 
\begin{enumerate}
\item there exists $\theta_x>0$ such that
\[  \lan \nabla  k_{x,r} \left( y\right), p \ran \leq -\theta_x/r  \qquad \mbox{ for $p\in K_x^{\delta_x/3}$ and $y\in O(x,r)$;}\]
\item for every $i,j\in \{1,\cdots,J\}$, 
\[  \sup_{y\in O(x,r)}\left|\frac{\partial^2 k_{x,r} \left(y
    \right)}{\partial y_i y_j} \right| <\left(\frac{3}{\eta_x}+1\right)\frac{1}{r^2};\]
\item $\sup_{y\in O(x,r)}(|k_{x,r}(y)-\dist((y-x)/r+\lambda_x (R_x/2)q_x,L_{x,\delta_x})|\vee (r|\nabla k_{x,r}(y)|-1))\leq \lambda_x/12$.
\end{enumerate}
From the first property of $k_{x,r}$ stated above, it follows that 
 \[\lan r\nabla k_{x,r}(y),\gamma^i(x)\ran \geq \theta_x\mbox{ for $i\in \MI(x)$ and $y\in O(x,r)$}.\] 
Since $\gamma^i(\cdot)$ is  continuous for each $i\in \MI$, by possibly making $r_x$ yet smaller and using the third property of $k_{x,r}$, we may assume that
for each $r\in (0,r_x)$,
\be \label{func2}\lan r\nabla k_{x,r} (y),\gamma^i(y)\ran \geq \theta_x/2\mbox{ for $i\in \MI(y)$ and $y\in O(x,r)$}.\ee

Now, choose $\zeta_x\in \C^\infty(\R)$ to be a decreasing function such
that \be
\label{def-hx}
\zeta_x(s)=\left\{\begin{array}{ll} 1 & \mbox{ if } s\in
    (-\infty,5\lambda_x/4],\\ \mbox{ strictly decreasing } & \mbox{ if
    } s\in (5\lambda_x/4, 23\lambda_x/12], \\ 0 & \mbox{ if } s\in
    (23\lambda_x/12,\infty), \end{array}\right.
\ee
and define $g_{x,r}:\ \R^J\rightarrow \R_+$ to be
\be 
\label{gxr2}
g_{x,r}\doteq \left\{\begin{array}{ll} \zeta_x(k_{x,r}(y)) & \mbox{ if } y\in O(x,r), \\
1 & \mbox{ if } y\in \overline G \cap M^{\eta_x r}(x,r), \\
0 & \mbox{ otherwise.} \end{array} \right.\ee
It follows from the definition of $g_{x,r}$, the properties of $\zeta_x$
and $k_{x,r}$, the definitions of $M(x,r)$ and $O(x,r)$ given in
(\ref{Mxr}) and  (\ref{Oxr}), respectively, and property (\ref{sub1}) that 
\be 
\label{support}
\begin{array}{l} \ds supp[g_{x,r}]\cap \overline G \\ \qquad \ds
  \subset \left\{y\in \overline G:\ k_{x,r}(y)\leq
    23\frac{\lambda_x}{12}\right\} \\ 
\qquad \ds \subset   \{y\in \overline G:\
\dist\left(\frac{y-x}{r}+\frac{\lambda_x
    R_x}{2}q_x,L_{x,\delta_x}\right)\leq
\frac{23\lambda_x}{12}+\frac{\lambda_x}{12}\} \\ 
\qquad \ds =  \left\{y\in \overline G:\ \dist\left(y-x+\frac{\lambda_x
      R_x}{2} rq_x, rL_{x,\delta_x}\right)\leq 2\lambda_xr\right\} \\ \qquad \ds = M^{2\lambda_xr}(x,r)\cap \overline G \\ \qquad \ds \subset B_r(x)\cap \overline G. \end{array} \ee
This establishes property (1) of the lemma.  In addition, 
\be \begin{array}{l} \ds \{y\in \overline G:\ k_{x,r}(y)\geq
  5\lambda_x/4\} \\ \quad \ds \subset   \left\{y\in \overline G:\
    \dist\left(\frac{y-x}{r}+\frac{\lambda_x
        R_x}{2}q_x,L_{x,\delta_x}\right)\geq
    \frac{5\lambda_x}{4}-\frac{\lambda_x}{12}\right\} \\ \quad \ds
  \subset  \left\{y\in \overline G:\ \dist((y-x)+\frac{\lambda_x
      R_x}{2}rq_x, rL_{x,\delta_x})\geq \lambda_xr\right\}. \end{array} \ee 
Thus, the set on which $g_{x,r}$ is not constant is a strict subset of
$O(x,r)$. Combining this with (\ref{support}) and the properties  $\zeta_x\in \C^\infty(\R)$
and $k_{x,r}\in \C^\infty(O(x,r))$,  it follows that $g_{x,r}\in
\C^\infty(\overline G)$.

Since $x$ is an interior point of $M(x,r)$, there exists
$\delta_{x,r} > 0$ such that $B_{\delta_{x,r}}(x)\subset M(x,r)$. 
For $y \in B_{\delta_{x,r}}(x)$,  property 3 of $k_{x,r}$ implies 
$k_{x,r}(y) \leq \lambda_x /12$, which when combined with 
(\ref{gxr2}) and  (\ref{def-hx}), implies $g_{x,r} (y) = 1$. 
Thus, $g_{x,r}$ satisfies property 2 of the lemma.  On the other
hand, $g_{x,r}$ satisfies property 3 because of  (\ref{gxr2}) and 
properties 2 and 3 of $k_{x,r}$.   Finally, for each $y\in O(x,r)$, a simple
calculation shows that\[\nabla g_{x,r}(y) =
\zeta_x'(k_{x,r}(y))\nabla k_{x,r} (y). 
\]
Together with (\ref{func2}) and the fact that $\zeta_x$ is non-increasing,
this implies that \[\lan \nabla g_{x,r}(y), \gamma^i(y) \ran \leq
0\mbox{ for $i\in \MI(y)$ and $y\in O(x,r)$.}\] Thus, $g_{x,r}$ also
satisfies the fourth property stated in the lemma. 
\end{proof}

\begin{corollary} \label{cor:A1}
If $n^i(\cdot)$ and $\gamma^i(\cdot)$ are constant vector fields for
each $i\in \MI$,  then the constants $A_x\in (0,\infty)$ and $r_x>0$ in Proposition \ref{prop:TF2} can be chosen to satisfy $A_x=A_{x'}$ and $r_x=r_{x'}$ if $\MI(x)=\MI(x')$, $r_x<\dist(x,\MV\cup\cup_{i\notin \MI(x)}(\partial G\cap \partial G_i))$ and $r_{x'}<\dist(x',\MV\cup\cup_{i\notin \MI(x')}(\partial G\cap \partial G_i))$. Moreover,  the family
of nonnegative functions $\{g_{x,r}\in \C^2_c(\overline G):\ r\in
(0,r_x]\}$ in Proposition \ref{prop:TF2} is translation invariant in the sense that $g_{x,r}(y)=g_{x+\delta,r}(y+\delta)$ if $\MI(x)=\MI(x')$ and $r_x=r_{x'}$.
\end{corollary}
\begin{proof} Suppose that $n^i(\cdot)$ and $\gamma^i(\cdot)$ are constant
vector fields for each $i\in \MI$. Notice that the sets $K_x$ and
$K_{x,\delta_x}$ and the constant $\delta_x$ can be chosen to depend
on $x$ only through $\MI(x)$.  
Similarly, the constants $R_x$ and $\beta_x$ in Lemma
\ref{lem:A1}, and hence the set $L_{x,\delta_x}$ in (\ref{Lx}), depend
on $x$ only through $\MI(x)$. It is obvious that $\overline G=S_x$
near $x$ and $S_x$ depends on $x$ only through $\MI(x)$. In addition,
the constant  $\lambda_x \in (0,1)$ in (\ref{include1}) and
(\ref{include2}) depends on $x$ only through $\MI(x)$ and $r_x$ in
(\ref{include1}) and (\ref{include2})  depends on $x$ only through
$\MI(x)$ when $r_x<\dist(x,\MV\cup \cup_{i\notin \MI(x)}(\partial
G\cap \partial G_i))$. Also, notice that by the proof of Lemma 6.2 of
\cite{Ram06}, the constants $\varepsilon_x\in (0,1),\ \eta_x>0$, the
set $\Lambda_x$ and the function $\ell_x$ depend on $x$ only through
$\MI(x)$. By examining the proof of Proposition \ref{prop:TF2}, we see
that the vector $q_x$, the set $M(x,r)$ in (\ref{Mxr}), the set
$O_{x,r}$ in (\ref{Oxr}) and the function $\zeta_x$ depend on $x$ only
through $\MI(x)$. 
This establishes the corollary.  \end{proof}

\begin{proof}[Proof of Proposition \ref{prop-testfn}]
Set $r(x,\ve) \doteq \left( (r_x/2)  \wedge \ve\right)/\max(\cxtwo, 1)$ for $x \in \overline{G}$ and $\ve
  > 0$.  
  Now, for $x \in G$ and $x \in \partial G\setminus \MV$,  let $g_x^\ve \doteq g_{x,
  r(x,\ve)}$ be the corresponding functions in 
${\cal C}_c^2 (\overline{G})$ constructed in 
Propositions \ref{lem:TF3} and \ref{prop:TF2}, respectively, 
and for $x \in \MV$, let $g_{x,r(\ve)}$ be as constructed in Proposition 
\ref{lem:test1}.   
 Note that the first property established in Proposition
\ref{lem:TF3}   shows that when $x \in G$, $g_x^\ve(y) = 0$ for $y \in \partial
G$.   Together with the finiteness of $\MV$, the second property of Proposition
\ref{lem:test1},  the definition of $r_x$ for $x \in \partial G\setminus \MV$,
and the first property of 
Proposition \ref{prop:TF2}, this shows  that $g_x^\ve$ is constant in a
neighborhood of $\MV$.   On the other hand, when combined with the first property of Proposition \ref{lem:TF3} and the  fourth properties of
Propositions \ref{lem:test1} and \ref{prop:TF2}, it shows that 
$\lan d, \nabla g_x^\ve (y)
\ran \leq 0$ for $d \in d(y)$ and $y \in \partial G$.   Thus, 
$g_x^\ve \in {\cal H}$.  In addition,   the choice of
 $r(\ve)$ and the first two properties of $\{g_{x,r(\ve)}, x \in
\overline{G}\}$ established in Propositions  \ref{lem:TF3},
\ref{lem:test1} and \ref{prop:TF2}
immediately imply that 
 $\{g_x^\ve, x \in \overline{G}\}$ satisfies 
the first two properties stated in 
Proposition \ref{prop-testfn}. 
The second part of the proposition follows directly from  Remark \ref{rem:gxr} and Corollary \ref{cor:A1}. 
\end{proof}

\section{Verification of Assumption \ref{ass:TF} for Example
  \ref{ex:cusp}.}
\label{app:cusp}

In this section we show that Assumption \ref{ass:TF} holds for the 
two-dimensional domain  $G$ with a cusp at the origin described in Example
\ref{ex:cusp}.  
 The argument relies on the construction of a family of  functions
 $\{g_{x,r}\in \C^2_c(\overline G):\ r\in (0,r_x),\ x\in \overline
 G\}$ that is similar to the family constructed in Appendix
 \ref{ap-testfn}.   Once again, the nature of the construction is
 different, depending on whether $x \in G$, $x \in \partial G \sm \MV$ or $x \in
 \MV$. For  $x\in G$, 
let $\{g_{x,r}\in \C^2_c(\overline G):\ r\in (0,r_x)\}$ be
the 
family of functions constructed in Proposition \ref{lem:TF3}.  Now, clearly $(G,d(\cdot))$  is piecewise ${\cal C}^1$ with continuous
reflection and, as discussed in Example \ref{ex:cusp}, Assumption 2'
is satisfied with $\MV=\{0\}$. Therefore, 
there exists a family of functions $\{g_{0,r}, r\in (0,r_0/c^2_0)\}$
that satisfy the properties stated  in Proposition \ref{lem:test1}.

It remains to consider $x \in \partial G \sm \MV$.  
We first consider the case $x \in \partial G_1\setminus \{0\}$. As $x$ moves to the right along $\partial G_1$ to
infinity, the curvature of $\partial G_1$ tends to $0$.
In other words,  the larger $|x|$ is, the flatter $\partial G_1$ is in
any neighbourhood of  $x\in \partial G^1$.  The angle between
$n^1(x)$ and $\gamma^1(x)$ is fixed and is equal to  $\theta_1\in
(-\pi/2,\pi/2)$.  Let $K_x\doteq \{-\gamma^1(x)\}$.  Then
$|\gamma^1(x)|=1/\cos(\theta_1)$ because $\lan n^1(x),\gamma^1(x)\ran = 1$.  By
the geometry of $\partial G_1$ and Lemma \ref{lem:A1},
 there exist $\delta>0$, $R\in (0,1)$ and
$\beta>0$  (all depending only on $\partial G^1$ and $\theta_1$) such
that  with $K_{x,\delta}\doteq  B_\delta(-\gamma^1(x))$, (\ref{sep2})
and (\ref{sep1}) hold with $\delta_x$ and $\beta_x$ replaced by $\delta$
and $\beta$, respectively. 
Under the local  coordinates at $x$ (i.e., taking the tangent line to
$\partial G_1$ as the $x$-axis and $-n^1(x)$ as the $y$-axis), the set $x+\cup_{t\leq R}tK_{x,\delta}$ is identical for each $x\in \partial G_1\setminus \{0\}$.
 Thus, it follows that there exist $\bar r\in (0,1)$ and $\lambda\in
 (0,1)$ (both depending only on $\partial G^1$ and $\theta_1$) such that
 for each $r\in (0,\bar r)$, (\ref{include1}) and (\ref{include2}) hold with $\delta_x$, $R_x$ and $\lambda_x$ replaced by $\delta$, $R$ and $\lambda$, respectively. Let $L_{x,\delta}$ be a truncated (half) cone with vertex at the origin defined by\[ L_{x,\delta}\doteq \cup_{t\leq R/2}tK_{x,\delta}.\]
Thus, $(x+L_{x,\delta})\cap\overline G = \{x\}$ and $L_{x,\delta}$ is identical under  the local  coordinates at $x$ for each $x\in \partial G_1\setminus \{0\}$. For any $\ve\in (0,1)$ and $0<\eta<2\lambda<\infty$, it follows from Lemma \ref{prop:TF} with $\overline{{\cal C}} = \cup_{t \geq 0} t
K_{x,\delta}$,   ${\cal K} = B_{\delta/3}(-\gamma^1(x))$ and $\Lambda  =
\Lambda_x\doteq \{y\in \R^J:\ \eta<\dist\left(y, \cup_{t\geq 0}tK_{x,\delta}\right)\leq 2\lambda\}$ that  there exists a
function $\ell_x:\ \Lambda_x \rightarrow \R$, that satisfies all the
properties stated in Lemma \ref{prop:TF}.  Notice that $\ell_x$ and
$\Lambda_x$ are identical under  the local  coordinates at $x$ for
each $x\in \partial G_1\setminus \{0\}$ and by (\ref{sep2}), $\theta$
in property (3) of $\ell_x$  depends only on $\partial G^1$ and
$\theta_1$. Then, using an argument similar to that used in the proof
of Proposition \ref{prop:TF2}, there exist $0<r_x<\dist(x,\partial
G_2)$ and $A\in (0,\infty)$ and a family  of functions $\{g_{x,r}\in
\C^2_c(\overline G):\ r\in (0,r_x]\}$ satisfying all the properties in
Proposition \ref{prop:TF2}. Here, for each $x_1,x_2\in \partial
G_1\setminus \{0\}$ and $r<r_{x_1}\wedge r_{x_2}$, $g_{x_1,r}$ and
$g_{x_2,r}$ are identical under  the local  coordinates at $x_1$ and
$x_2$, respectively.  By symmetry,  a family of functions
$\{g_{x,r}\in \C^2_c(\overline G):\ r\in (0,r_x],\ x\in \partial
G_2\setminus\{0\}\}$ that satisfy  analogous properties can also be established. Using the family of functions $\{g_{x,r}\in \C^2_c(\overline G):\ r\in (0,r_x], \ x\in \overline G\}$, we can follow a similar argument as the one in the proof of Theorem \ref{thm:test} for the case when $(G, d(\cdot))$ is a polyhedral domain with piecewise constant
reflection to establish Assumption \ref{ass:TF}.

\section{Proof of Lemma \ref{lem:concave}} \label{app:1}
Let
 \[\xi(\lambda)\doteq \int_{\overline G}\left[\psi(f_1(y)-\lambda \ML
   f_1(y),\cdots,f_n(y)-\lambda \ML f_1(y))-\psi(f_1(y),\cdots,f_n(y))
  \right]d\pi(y).\]
Since $\psi$ is concave, it is clear that the map  $\lambda \mapsto
\xi(\lambda)$ is also  concave.
To prove the lemma, we need to show that $\xi(\lambda)\leq 0$.
In turn, to establish this, it suffices to  show that $\xi'(0)\leq 0$ because $\xi$ is concave and $\xi(0)=0$.

Now, $\psi \in \C^2_b (\overline G)$ and for $i = 1, \ldots, n$, $f_i \in \MH$
implies $f_i \in \C^2_b(\overline G)$.
Thus,  the function $\Psi(\cdot)\doteq
\psi\left(f_1(\cdot),\cdots,f_n(\cdot)\right)$ lies in
$\C^2_b(\overline G)$.
 In addition, since $\psi$ is monotone increasing in each variable
 separately, and $f_i \in \MH$ for each $i = 1, \ldots, n$,  it
 follows that
for each $y\in \partial G$ and $d\in d(y)$,
\[\left<d,\nabla \Psi(y)\right> = \sum_{i=1}^n\left<d,\frac{\partial
    \psi }{\partial z_i} \left(f_1(y),\cdots,f_n(y)\right)\nabla
  f_i(y)\right>\leq 0.\]
Moreover, since $f_i\in \MH$ implies $f_i$ is constant in a neighborhood of $\MV$, it follows that $\Psi$ is also constant in a neighborhood of $\MV$. Since $f_i\in \MH$ also implies that $f_i$ is constant outside some compact set, it follows that $\Psi$ is also constant outside some compact set. These show that $\Psi\in \MH$ and therefore, by (\ref{mono0}), that \[\int_{\overline G} \ML \Psi(y) d\pi(y) \geq 0.\]
Now, \[\xi'(0)=-\int_{\overline G}\left(\sum_{i=1}^n\frac{\partial
    \psi}{\partial z_i}\left(f_1(y),\cdots,f_n(y)\right)\ML
  f_i(y)\right)d\pi(y).\] Thus, if we can show that for each $y\in G$,
\be \label{claim:2}\sum_{i=1}^n\frac{\partial
  \psi}{\partial z_i}\left(f_1(y),\cdots,f_n(y)\right)\ML f_i(y)\geq
\ML \Psi(y),\ee it follows that $\xi^\prime (0) \leq 0$.
 In the rest of the proof, we establish (\ref{claim:2}).

Fix $y\in G$. A straightforward calculation shows that
\begin{eqnarray*} \ML \Psi(y) &= & \sum_{i=1}^J b_i(y)\sum_{k=1}^n\frac{\partial \psi}{\partial
z_k}\left(f_1(y),\cdots,f_n(y)\right)\frac{\partial f_k}{\partial x_i} (y)\\ & & +\frac{1}{2}\sum_{i,j=1}^J a_{ij}(y)\left( \sum_{k,l=1}^n\frac{\partial^2\psi}{\partial
z_k\partial z_l}\left(f_1(y),\cdots,f_n(y)\right)\frac{\partial f_k}{\partial x_i}(y)\frac{\partial f_l}{\partial x_j}(y)\right)\\ & & +\frac{1}{2}\sum_{i,j=1}^J a_{ij}(y)\left( \sum_{k=1}^n\frac{\partial \psi}{\partial
z_k}\left(f_1(y),\cdots,f_n(y)\right)\frac{\partial^2 f_k}{\partial x_i\partial x_j}(y)\right)\\ &=& \sum_{i=1}^n\frac{\partial \psi}{\partial
z_i} \left(f_1(y),\cdots,f_n(y)\right)\ML f_i(y)\\ & &
+\frac{1}{2}\sum_{i,j=1}^J a_{ij}(y)\left( \sum_{k,l=1}^n
\frac{\partial^2\psi}{\partial
z_k\partial z_l}\left(f_1(y),\cdots,f_n(y)\right)\frac{\partial f_k}{\partial x_i}(y)\frac{\partial f_l}{\partial x_j}(y)\right).\end{eqnarray*}
Hence, to establish (\ref{claim:2}) it suffices to show that
\[ \sum_{i,j=1}^J a_{ij}(y)\left(
  \sum_{k,l=1}^n\frac{\partial^2\psi}{\partial z_k\partial z_l}\left(f_1(y),\cdots,f_n(y)\right)\frac{\partial f_k}{\partial x_i}(y)\frac{\partial f_l}{\partial x_j}(y)\right)\leq 0. \]
Since the matrix $a(y)$ is positive semidefinite, let $a^{1/2}(y)$ be its
positive semidefinite square root. Then, we have
\begin{eqnarray*}\sum_{i,j=1}^J a_{ij}(y) \frac{\partial f_k}{\partial
    x_i}(y)\frac{\partial f_l}{\partial x_j}(y) &=&  (\nabla
  f_k(y))^Ta(y) \nabla f_l(y) \\ &=& \left(a^{1/2}(y)\nabla
    f_k(y)\right)^T \left(a^{1/2}(y)\nabla f_l(y)\right) \\ &=&
  \sum_{m=1}^J c^k_mc^l_m,\end{eqnarray*} where $c^k =
a^{1/2}(y)\nabla f_k(y)$ for each $k=1,\cdots,n$.
Now, we can see that
\begin{eqnarray*}
&&\sum_{i,j=1}^J a_{ij}(y)\left( \sum_{k,l=1}^n\frac{\partial^2\psi}{ \partial
z_k\partial z_l}
\left(f_1(y),\cdots,f_n(y)\right)\frac{\partial f_k}{\partial x_i}(y) \frac{\partial
f_l}{\partial x_j}(y)\right)\\ && \qquad \qquad =
\sum_{k,l=1}^n\frac{\partial^2\psi}{\partial z_k\partial z_l}
\left(f_1(y),\cdots,f_n(y)\right)\left( \sum_{i,j=1}^J a_{ij}(y)\frac{\partial
   f_k}{\partial x_i}(y)\frac{\partial f_l}{\partial x_j}(y)\right)\\
&&\qquad\qquad = \sum_{k,l=1}^n\frac{\partial^2\psi}{\partial
z_k\partial z_l}\left(f_1(y),\cdots,f_n(y)\right)\sum_{m=1}^J c^k_mc^l_m \\ &&\qquad \qquad =\sum_{m=1}^J \left(\sum_{k,l=1}^n\frac{\partial^2\psi}{\partial
z_k\partial z_l} \left(f_1(y),\cdots,f_n(y)\right)c^k_mc^l_m\right) \leq 0,
\end{eqnarray*}
where the last inequality holds because the Hessian matrix of $\psi$
is negative semidefinite since $\psi$ is concave.
\endproof
\bigskip

\section{Proof of Lemma \ref{lem:g0}}
\label{sec:lemg0}

 Fix $x\in \MV$.  By Assumption \ref{ass:V} there exist positive
constants $r_x,\ c_x, \alpha_x$ that satisfy $h(y) \geq  c_x|y-x|$,
$\lan \nvec_x, a(y) \nvec_x\ran \geq \alpha_x$ and   $\lan \nvec_x,d\ran\geq 0$ for all $d\in
d(y)$ and all $y\in \overline G
\cap B_{r_x}(x)$. Let $\varepsilon_0\in (0,r_x)$ be such that
$\varepsilon_0$ is less than the distance between any two points in
$\MV$. For each $\varepsilon\in (0,1)$ such that
$2(\varepsilon+\sqrt{\varepsilon})<c_x\varepsilon_0$ and $\delta\in
(0,\varepsilon)$ sufficiently small such that
$\delta+\sqrt{\delta}<\varepsilon$, let $l_{\delta,\varepsilon}$ be a
function such that
\[ l_{\delta,\varepsilon}(s)=\left\{\begin{array}{ll} 0 & \mbox{ if }
    s\leq \delta,\\
    \frac{2\sqrt{\delta}+1}{3\sqrt{\delta}}(s-\delta)^3  & \mbox{ if }
    \delta< s \leq \delta +\sqrt{\delta}, \\
    2\delta(\sqrt{\delta}+1)+s^2-(\delta+\sqrt{\delta})^2 & \mbox{ if
    } \delta +\sqrt{\delta} < s\leq \varepsilon, \\ \begin{array}{ll}
      & 2\delta(\sqrt{\delta}+1)-(\delta+\sqrt{\delta})^2
      +\varepsilon^2 \\
      &+\varepsilon^{3/2}-\sqrt{\varepsilon}(s-\varepsilon -
      \sqrt{\varepsilon})^2  \end{array} & \mbox{ if } \varepsilon < s
    \leq \varepsilon +\sqrt{\varepsilon}, \\
    2\delta(\sqrt{\delta}+1)-(\delta+\sqrt{\delta})^2 +\varepsilon^2
    +\varepsilon^{3/2} & \mbox{ if } \varepsilon +\sqrt{\varepsilon}
    <s. \end{array}\right.\]
It is easy to verify that
\be
\label{lde-1}
\begin{array}{ll}
0 \leq l_{\delta,\varepsilon}(s)\leq 5 \ve,   \ 0\leq l_{\delta,\varepsilon}'(s)\leq 2\sqrt{\varepsilon},  & \mbox{
  for } s\in
\R, \\
l_{\delta,\varepsilon}'(s)=0, & \mbox{ for }
s>\varepsilon+\sqrt{\varepsilon}.
\ea
\ee
Also, note that $l_{\delta,\varepsilon}\in \C^1(\R)$ is piecewise
differentiable with derivative
\[ l_{\delta,\varepsilon}''(s)=\left\{\begin{array}{ll} 0 & \mbox{ if
    } s\leq \delta,\\ \frac{4\sqrt{\delta}+2}{\sqrt{\delta}}(s-\delta)
    & \mbox{ if } \delta\leq s < \delta +\sqrt{\delta}, \\  2 & \mbox{
      if } \delta +\sqrt{\delta} < s< \varepsilon, \\
    -2\sqrt{\varepsilon}  & \mbox{ if } \varepsilon < s < \varepsilon
    +\sqrt{\varepsilon}, \\ 0 & \mbox{ if } \varepsilon
    +\sqrt{\varepsilon} <s. \end{array}\right.\]

We now use a standard mollifcation argument to construct a $\C^2(\R)$
function with similar properties.
 Let $\{\phi_n\in \C^\infty_c(\R), n\in \N\}$ be a sequence of
 non-negative functions with $\int_{\R} \phi_n(x) \, dx =  1$ and
 compact supports that shrink  to $\{0\}$.
 Define
\[ l^n_{\delta,\varepsilon}\doteq \phi_n\ast l_{\delta,\varepsilon},\]
where $\ast$ denotes the convolution operation. Then
$l^n_{\delta,\varepsilon}\in \C^\infty_c(\R)\oplus \R$ and for $n$ sufficiently
large,  there exist $\kappa_n > 0$ with $\lim_{n \ra \infty} \kappa_n = 0$
such that
\be
\label{lde-2}
\ba{ll}
(l^n_{\delta,\varepsilon})''(s)\geq 2  & \mbox{ if } \delta
+2\sqrt{\delta} \leq s \leq \varepsilon/2, \\
 |(l^n_{\delta,\varepsilon})''(s)|\leq 2\sqrt{\varepsilon} & \mbox{ if
 } s\geq \varepsilon - \kappa_n \\
(l^n_{\delta,\varepsilon})''(s)=0 & \mbox{ if } s\geq
2(\varepsilon+\sqrt{\varepsilon}) \\
l^n_{\delta,\varepsilon}(s)=2\delta(\sqrt{\delta}+1)-(\delta+\sqrt{\delta})^2 +\varepsilon^2
    +\varepsilon^{3/2} & \mbox{ if } s\geq
2(\varepsilon+\sqrt{\varepsilon}) \\
(l^n_{\delta,\varepsilon})''(s)\geq 0 & \mbox{  otherwise. }
\ea
\ee
Now, choose $n$ sufficiently large and let
\[ g_{\delta,\varepsilon}(y) \doteq \left\{\begin{array}{ll} l^n_{\delta,\varepsilon}(h(y)), & \mbox{ if }
 y \in \overline G \cap B_{r_x}(x) \\ 2\delta(\sqrt{\delta}+1)-(\delta+\sqrt{\delta})^2 +\varepsilon^2
    +\varepsilon^{3/2}, & \mbox{ if } y \in \overline G \setminus B_{r_x}(x). \end{array}\right. \]

We now show that $g_{\delta, \varepsilon}$ has the desired
properties.  Firstly,
 since $l_{\delta,\varepsilon}(s)=0$ when $s\leq \delta$, it follows
 that $g_{\delta,\varepsilon}$ is zero in a neighborhood of $x$.  In
 addition,  the second property in (\ref{lde-1}),  the fact that
 $\ve + \sqrt{\ve} < c_x \ve_0$ and  $h(y)
 \geq c_x|y-x|$ imply that
 $g_{\delta,\varepsilon}$ is  constant when $y\in \overline G
 \setminus B_{\varepsilon_0}(x)$.  By the choice of $\ve_0$, this
 implies that  $g_{\delta,\varepsilon}$ is constant in a neighborhood
 of $\MV$.
Also, by the nonnegativity of $\phi_n$ and $l_{\delta, \varepsilon}$
 as well as the property of $\nvec_x$
stated in Assumption \ref{ass:V}, it follows
 that for every $y\in \partial G$,
\[\lan \nabla g_{\delta,\varepsilon}(y) , d\ran =
  \left(\int_{\R}l_{\delta,\varepsilon}'(h(y)-z)\phi_n(z)dz\right)
  \left\lan \nvec_x , d\right\ran \geq 0, \qquad d \in d(y). \]
Thus,
$g_{\delta,\varepsilon}\in \C^2_c(\overline G)\oplus \R$  and
$-g_{\delta,\varepsilon}\in \MH$.
In addition,  the first property in  (\ref{lde-1})
shows that  $\sup_{y \in \overline G} |g_{\delta,\varepsilon}(y)|\leq
5\varepsilon$ and  $\sup_{y\in \overline G}|\nabla g_{\delta,\varepsilon}(y)|\leq
 2\sqrt{\varepsilon}$, while
the second property in (\ref{lde-1}) implies that
\[\left|\lan b(y), \nabla g_{\delta,\varepsilon}(y) \ran \right|\leq
2\sqrt{\varepsilon} \sup_{|y-x|\leq 2\varepsilon_0} |b(y)|, \qquad
y \in \overline G. \]

Next, note that
\[\sum_{i,j=1}^J a_{ij}(y)\frac{\partial^2g_{\delta,\varepsilon}(y)}{\partial
x_i\partial x_j}=(l^n_{\delta,\varepsilon})''(h(y))\nvec_x^Ta(y)\nvec_x.\]
The first property in (\ref{lde-2})  and the definition of $\alpha_x$
implies that 
\[\sum_{i,j=1}^J a_{ij}(y)\frac{\partial^2g_{\delta,\varepsilon}(y)}{\partial
x_i\partial x_j}\geq 2\alpha_x, \qquad \mbox{ if }
\delta+2\sqrt{\delta}\leq h(y)\leq \varepsilon/2, y\in \overline G \cap B_{r_x}(x). \]
Moreover, by the second and third  properties in
(\ref{lde-2}), and the fact that
$\ve + \sqrt{\ve} < c_x \ve_0$
it is clear that for each $y\in \overline G\cap B_{r_x}(x)$ with $h(y)\geq \varepsilon-\kappa_n$,
\[\left|\sum_{i,j=1}^J a_{ij}(y)\frac{\partial^2g_{\delta,\varepsilon}(y)}{\partial
x_i\partial x_j}\right|\leq
\left|(l^n_{\delta,\varepsilon})''(h(y))\right||a(y)|\leq
2\sqrt{\varepsilon} \sup_{|y-x|\leq 2\varepsilon_0} |a(y)|.\]
The fourth property in (\ref{lde-2}) shows that
 when $y\in \overline G\cap B_{r_x}(x)$ with $h(y)<\delta+2\sqrt{\delta}$ or
 $\varepsilon/2<h(y)<\varepsilon-\kappa_n$,
\[\sum_{i,j=1}^J a_{ij}(y)\frac{\partial^2g_{\delta,\varepsilon}(y)}{\partial
x_i\partial x_j}\geq 0.\] Thus, we have shown that
 that properties 1--6 hold with $c=2\min_{x \in {\mathcal V}} \alpha_x$ and
\[ C  = 5 \vee 2
\sup_{|y-x|\leq 2\varepsilon_0} (|a(y)| \vee |b(y)|)<\infty.\]

\end{document}